\documentclass[11pt,reqno]{amsart}

\usepackage{fullpage}

\usepackage[margin=1in]{geometry}
\usepackage{amsmath}
\usepackage{amsfonts}
\usepackage{amssymb}
\usepackage{amsthm,mathtools,mathabx,accents,}
\usepackage{newlfont}
\usepackage{graphicx}
\usepackage{xcolor,enumerate,mathrsfs, scalerel}
\usepackage{esint}
\usepackage{hyperref}
\newtheorem{thm}{Theorem}[section]
\newtheorem{cor}[thm]{Corollary}
\newtheorem{lem}[thm]{Lemma}
\newtheorem{prop}[thm]{Proposition}

\theoremstyle{definition}

\newtheorem{rem}[thm]{Remark}
\numberwithin{equation}{section}

\newcommand{\na}{\nabla}
\newcommand{\pa}{\partial}
\newcommand{\lec}{\lesssim}
\newcommand{\td}{\tilde}

\renewcommand{\div}{\operatorname{div}}

\newcommand\al{\alpha}

\newcommand\de{\delta}

\newcommand\Ga{\Gamma}

\newcommand{\la}{\lambda}

\newcommand{\T}{\mathbb{T}}
\newcommand{\R}{\mathbb{R}}

\newcommand{\supp}{\operatorname{supp}}

\newcommand{\I}{\operatorname{Id}}

\usepackage[usestackEOL]{stackengine}
\def\dint{\,\ThisStyle{\ensurestackMath{%
  \stackinset{c}{.2\LMpt}{c}{.5\LMpt}{\SavedStyle-}{\SavedStyle\phantom{\int}}}%
  \setbox0=\hbox{$\SavedStyle\int\,$}\kern-\wd0}\int}

\begin{document}

\title{The Onsager conjecture in 2D: a Newton-Nash iteration}

\author{Vikram Giri}
\address{Department of Mathematics, Princeton University
Fine Hall, Princeton, NJ 08544, USA}
\email{vgiri@math.princeton.edu}

\author{R\u{a}zvan-Octavian Radu}
\address{Department of Mathematics, Princeton University
Fine Hall, Princeton, NJ 08544, USA}
\email{rradu@math.princeton.edu}

 \date{\today}

%\subjclass[2010]{Primary: . Secondary: }
% \keywords{} 
\begin{abstract} 
   For any $\gamma<1/3$, we construct a nontrivial weak solution $u$ to the two-dimensional, incompressible Euler equations, which has compact support in time and satisfies $u\in C^\gamma(\mathbb R_t \times \mathbb T^2_x)$. In particular, the constructed solution does not conserve energy and, thus, settles the flexible part of the Onsager conjecture in two dimensions. The proof involves combining the Nash iteration technique with a new linear Newton iteration.
\end{abstract}

\maketitle

\section{Introduction} 
Consider the incompressible Euler equations
\begin{equation} \label{ie}
    \begin{cases}
        \partial_t u + \div(u \otimes u) + \nabla p = 0, \\ 
        \div u = 0,
    \end{cases}
\end{equation}
defined on $\R_t\times \T^2_x$, where $\T^2 = \mathbb R^2 / \mathbb Z^2$ is the two-dimensional torus; $u: \R\times \T^2 \to \R^2$ is the velocity field; and $p: \R\times \T^2 \to \R$ is the scalar pressure. In this paper, a weak solution is a pair $(u,p)$ which satisfies \eqref{ie} in the sense of distributions.

In his work on turbulent flows~\cite{Onsager49}, Onsager conjectured that weak solutions to the Euler equations \eqref{ie} with H\"older regularity greater than $1/3$ must conserve kinetic energy, while solutions with lower regularity need not do so. After partial results of Eyink~\cite{eyink}, the positive/rigid part of the conjecture was settled in the affirmative by Constantin, E, and Titi (\cite{ConstantinETiti94}; see also \cite{CCFS08} for a sharp version of energy-conservation for solutions with $1/3$ regularity in $L^3$-based Besov spaces). The first results towards the negative/flexible side of the assertion were those of Scheffer (\cite{Scheffer93}) and Shnirelman (\cite{Shnirelman00}), in which very low-regularity weak solutions were constructed. More recently, the problem was revisited by De Lellis and Sz\'ekelyhidi, who, in their seminal works \cite{DLS09}, \cite{DLS13}, and \cite{DLS14}, have constructed the first examples of non-conservative solutions with H\"older regularity. Their key observation was that flexible solutions of the Euler equations can be constructed by exploiting similar ideas to those introduced by Nash in his $C^1$ isometric embedding theorem (\cite{Nash}). This triggered a series of works (\cite{B15}, \cite{BDLIS15}, \cite{BDLS13}, \cite{BDLS16}, \cite{Choff}, \cite{ChDLS12},  \cite{DaneriSzekelyhidi17}, \cite{IsPhD}; see also the surveys \cite{BV20}, \cite{DLS12}, \cite{DLS17}, \cite{DLS22} for more complete descriptions of the developments), which culminated with the resolution of the flexible part of Onsager's conjecture by Isett, who, in \cite{Isett18}, constructed a non-trivial three-dimensional Euler flow with (almost) Onsager-critical regularity and compact support in time (see also \cite{cltv} for a simplification of those arguments, as well as for the extension of the result to admissible solutions). 
Recently, Novack and Vicol~\cite{NV22}, building on their work \cite{BMNV21} with Buckmaster and Masmoudi, have given a new proof of the flexible part of the Onsager conjecture in 3d, in which they use spatial intermittency to construct solutions with almost $1/3$ of a derivative in $L^3$-based spaces. 

We remark that none of the arguments which reach the Onsager-critical exponent from the flexibility side (\cite{Isett18}, \cite{cltv}, \cite{NV22}) extend to two-dimensions, while the arguments from the rigidity side (\cite{ConstantinETiti94}, \cite{CCFS08}) are dimension independent. Indeed, in the two dimensional case, the best H\"older exponent known previously was at most $1/5$. This result is proved in \cite{N20}[Theorem 1.2] where the author develops a convex integration scheme for the 3D quasi-geostrophic equations and connects it with the 2D Euler equations. In fact, the main theorem below also implies the Onsager conjecture for the 3D quasi-geostrophic equations, as described in \cite{N20}[\textsection\textsection 1.1.3-1.1.4].
% While this result is not explicitly proven in the literature, it follows by combining the 2D building blocks of \cite{ChDLS12} and \cite{Choff} with the Nash iteration schemes developed in \cite{IsPhD}, or \cite{BDLIS15}. 
The purpose of this paper is to prove the following theorem, closing the rigidity/flexibility gap also in $d=2$.

\begin{thm}[\textbf{Main Theorem}]\label{main_thm}
For any $0\leq \gamma<1/3$, there exist non-trivial weak solutions $(u,p)$ to \eqref{ie}, with compact support in time, and such that $u \in C^\gamma(\mathbb R_t \times \mathbb T^2_x)$.
\end{thm}

The proof of the main theorem \ref{main_thm} is based on the aforementioned Nash iteration technique. More precisely, the approach is to inductively define a sequence of smooth approximate solutions by adding, at each stage, highly oscillatory perturbations which interact through the nonlinearity to erase the error from being a solution, while, in turn, giving rise to much smaller errors. This sequence will, then, converge to a weak solution of \eqref{ie}. In order to define the highly oscillatory perturbations, one first decomposes the error into simpler errors, which, in the case of the Euler equations, correspond to a finite set of directions (c.f. lemma \ref{geom} below). Experience with the Nash iteration technique has shown that, in order to construct solutions with (almost) critical regularity, it is required that the perturbations corresponding to different directions do not interact, and, thus, the error-erasing is achieved only through self-interaction. 

This non-interaction property is achieved in both currently known proofs of the 3d Onsager conjecture (\cite{Isett18} and \cite{NV22}) by using as building blocks the Mikado flows introduced by Daneri and Sz\'ekelyhidi in \cite{DaneriSzekelyhidi17}. These flows are stationary, pressure-less solutions to the Euler equations, which are supported near straight (periodized) lines. In dimensions $d \geq 3$, given any finite set of directions, one can ensure that the Mikados corresponding to any two distinct directions have disjoint supports. 

In $d=2$, however, due to the elementary fact that any two non-parallel lines must intersect, Mikado flows are not a viable option to achieve the non-interaction. On the other hand, Cheskidov and Luo have recently introduced temporally oscillatory and intermittent perturbations in the context of Euler and Navier-Stokes equations (\cite{cluo}), and have shown that these can be used to achieve the non-intersection property (\cite{cluo2}). The idea is that one can erase the error at some set of times by a temporal corrector (i.e. through the time derivative in \eqref{ie}), while the spatially oscillatory Nash perturbation is used to erase the rest of the error. This procedure achieves, then, the non-interaction property by exploiting the extra dimension of time.

In the proof of theorem \ref{main_thm}, we also use oscillations in time to overcome the interaction problem. However, in order to reach the Onsager-critical regularity, we are led to defining the perturbation not as a temporal corrector, but as the solution to the Newtonian linearization of the Euler equations, where the forcing is augmented with a temporally oscillatory phase. Indeed, the perturbations defined by Cheskidov and Luo can then be seen as a first-order approximation of those we obtain by the device described above. A more precise description of this Newtonian iteration and the way it interacts with the Nash-type perturbation will be given in section \ref{heuristics}. 

We note that theorem~\ref{main_thm} furnishes a third proof of the Onsager theorem, which does not use intermittency or Mikado flows, and, which, moreover, yields non-conservative solutions in any dimension $d \geq 2$. Indeed, it is not difficult to see that any solution to the two-dimensional Euler equations can be trivially extended to a $d$-dimensional solution, with $d \geq 2$.

We remark that the non-conservative solutions constructed in theorem~\ref{main_thm} cannot arise as vanishing viscosity limits of Navier-Stokes solutions (under some natural assumptions) as such limits would conserve the total kinetic energy, c.f. Theorem 2 of~\cite{CFLS16}. This is in line with the marked difference one finds between two- and three- dimensional turbulence. One reason for this is the conservation of enstrophy, which is characteristic of dimension $d=2$. Another reason is that the main energy transfer mechanism in two-dimensions is the backwards cascade transferring energy from the small scale forcings to the large scale. This is different from the three-dimensional setting where the primary mechanism is the direct or forwards cascade which transfers energy from large scales to small scales where it is more readily dissipated. We refer the reader to section 3.1 of the PhD thesis of Drivas~\cite{Drivas} for a detailed discussion.

The paper is structured as follows: in section \ref{Section_Main_Prop}, we state the main iterative proposition \ref{Main_prop}, use it to prove theorem \ref{main_thm}, and, finally, describe the iteration at the level of heuristics. Sections \ref{section_Newton} and \ref{section_Nash} constitute the proof of proposition \ref{Main_prop}, the former consisting of the implementation of the Newton steps, while the latter containing the construction of the Nash perturbation and the estimation of the various resulting errors. In the appendix, we collect well-known results concerning H\"older spaces, mollification, transport equations, singular integral operators, tools of convex integration, and the linearized Euler equations. 

\subsection*{Acknowledgements} The authors would like to thank their advisor Camillo De Lellis for his support and encouragement regarding this work. We would also like to thank Gregory L. Eyink, L\'aszl\'o Sz\'ekelyhidi, Philip Isett, and Matthew Novack for their helpful comments regarding the introduction. Finally, we thank the anonymous referees for their suggestions that have improved the exposition. VG was supported by the National Science Foundation under grant DMS-FRG-1854344.

\section{The main iterative proposition} \label{Section_Main_Prop}

The proof of theorem \ref{main_thm} will be achieved by the iterative construction of smooth solutions $(u_q, p_q, R_q)$ to the Euler-Reynolds system 
\begin{equation} \label{ER}
    \begin{cases}
        \partial_t u_q + \div(u_q \otimes u_q) + \nabla p_q = \div R_q, \\ 
        \div u_q = 0,
    \end{cases}
\end{equation}
where the Reynolds stress $R_q$ is a symmetric $2$-tensor field. Here and throughout $q \in \mathbb N$ will denote the stage of the iteration. The goal is to construct this sequence so that $(u_q,p_q)$ converges in the required H\"older space, while $R_q$ converges to zero. In the limit, we will have, thus, recovered a weak solution to the Euler equations. 

We refer the reader to appendix~\ref{sec.a} for the notational conventions of the various norms that will appear throughout the paper.
% \subsection{Notational conventions} We describe here the notations for various norms that will appear throughout the paper. 
% \begin{enumerate}
%     \item For a time-independent tensor $F(x)$ and an integer $N$, we write $\|F\|_N$ to denote the $C^N(\T^2)$ norm. Also, for a real number $\gamma$, $\|F\|_\gamma$ denotes the $C^{\lfloor\gamma\rfloor, \gamma-\lfloor\gamma\rfloor}(\T^2)$ H\"older norm.
%     \item For a time-dependent tensor $F(x,t)$ and an integer $N$, we define
%     $$ \|F\|_N := \sup_{t\in\R} \|F(\cdot,t)\|_N $$
%     and make the analogous definition for $\|F\|_\gamma$ for a real number $\gamma$.
%     \item For a function of time $f:\R\to\R$, we define
%     $$ \|F\|_{N,\, \supp f} := \sup_{t \in \supp f} \|F(\cdot,t)\|_N\,. $$
% \end{enumerate}

\subsection{Parameters, inductive assumptions, main proposition} 
We begin by defining frequency parameters which will quantify the approximate Fourier support of $u_q$,
\begin{equation*}
    \lambda_q = 2\pi \lceil a^{b^q} \rceil,
\end{equation*}
as well as amplitude parameters 
\begin{equation*}
    \delta_q = \lambda_q^{-2\beta}.
\end{equation*}
The constant $a > 1$ will be chosen to be large, $b>1$ will be close to $1$, while $0 < \beta< 1/3$ will determine the H\"older regularity of the constructed solution.  

Let $L \in \mathbb N \setminus \{0\}$, $M > 0$ and $0< \alpha < 1$ be parameters whose precise values are chosen in proposition~\ref{Main_prop} below. We assume the following inductive estimates:

\begin{equation} \label{ind_est_1}
    \|u_q\|_0 \leq  M(1 - \delta_q^{1/2}), 
\end{equation}
\begin{equation} \label{ind_est_2}
    \|u_q\|_N \leq M \delta_q^{1/2} \lambda_q^N, \,\,\, \forall N \in \{1,2,..., L\},
\end{equation}
\begin{equation} \label{ind_est_3}
    \|p_q\|_N \leq M^2 \delta_q \lambda_q^N, \,\,\, \forall N \in \{1,2,..., L\},
\end{equation}
\begin{equation} \label{ind_est_4}
    \|R_q\|_N \leq \delta_{q+1} \lambda_q^{N-2\alpha}, \,\,\, \forall N \in \{0, 1,..., L\},
\end{equation}
\begin{equation} \label{ind_est_5}
    \|D_t R_q\|_{N} \leq \delta_{q+1} \delta_q^{1/2}\lambda_q^{N+1 - 2\alpha}, \,\,\, \forall N \in \{0, 1,..., L-1\},
\end{equation}
where here and throughout the paper $D_t$ denotes the material derivative corresponding to the velocity field $u_q$: 
\begin{equation*}
    D_t = \partial_t + u_q \cdot \nabla. 
\end{equation*}
Moreover, we assume the following on the temporal support of the stress: 
\begin{equation} \label{ind_est_6}
    \supp_t R_q \subset [-2 + (\delta_q^{1/2} \lambda_q)^{-1}, -1 - (\delta_q^{1/2} \lambda_q)^{-1}] \cup [1+(\de_q^{1/2}\la_q)^{-1}, 2-(\de_q^{1/2}\la_q)^{-1}],
\end{equation}
with the understanding that the constant $a$ in the definition of $\lambda_q$ is sufficiently large so that 
\begin{equation*}
    (\delta_0^{1/2} \lambda_0)^{-1} < \frac{1}{4}.
\end{equation*}

We are now ready to state the main iterative proposition.

\begin{prop} \label{Main_prop}
Let $L \geq 4$, $0< \beta < 1/3$ and
\begin{equation*}
    1 < b < \frac{1 + 3\beta}{6\beta}.
\end{equation*}
There exist $M_0 > 0$ depending only on $\beta$ and $L$, and a coefficient $0< \alpha_0 < 1$ depending on $\beta$ and $b$, such that for any $M > M_0$ and $0< \alpha < \alpha_0$, there exists $a_0 > 1$ depending on $\beta$, $b$, $\alpha$, $M_0$, $M$ and $L$, such that for any $a > a_0$ the following holds: given a smooth solution $(u_q, p_q, R_q)$ of \eqref{ER} satisfying \eqref{ind_est_1} - \eqref{ind_est_5}, as well as the condition \eqref{ind_est_6}, there exists a smooth solution $(u_{q+1}, p_{q+1}, R_{q+1})$ of \eqref{ER} satisfying \eqref{ind_est_1} - \eqref{ind_est_5} and condition \eqref{ind_est_6} with $q$ replaced by $q+1$ throughout. Moreover, it holds that 
\begin{equation} \label{Main_prop_eqn}
    \|u_{q+1} - u_{q}\|_0 + \frac{1}{\lambda_{q+1}} \|u_{q+1} - u_{q}\|_1 \leq 2M \delta_{q+1}^{1/2}, 
\end{equation}
and 
\begin{equation} \label{Main_prop_eqn_2}
    \supp_t (u_{q+1} - u_q) \subset (-2, -1) \cup (1, 2).
\end{equation}
\end{prop}

Next, we show that proposition \ref{Main_prop} implies the main theorem \ref{main_thm} and then discuss its proof at the heuristic level. Subsequently, the rest of the paper will be devoted to the proof of the iterative proposition \ref{Main_prop}.

\subsection{Proof of the main theorem \ref{main_thm}}

Let $L = 4$, and $\beta < 1/3$ such that $\gamma < \beta$, where $\gamma$ is the H\"older coefficient in the statement of the theorem. Fix $b$ so that it satisfies 
\begin{equation*}
    1 < b < \frac{1 + 3 \beta}{6\beta},
\end{equation*}
and let $M_0$ and $\alpha_0$ be the constants given by proposition \ref{Main_prop}. We fix also $M > \max\{M_0, 1\}$ and $\alpha < \min \{\alpha_0, 1/4\}$. Then, let $a_0$ be given by proposition \ref{Main_prop} in terms of these fixed parameters. We do not also fix $a > a_0$ until the end of the proof.

We now aim to construct the base case for the inductive proposition \ref{Main_prop}. Let $f : \R \to [0,1]$ be a smooth function supported in $[-7/4, 7/4]$, such that $f = 1$ on $[-5/4, 5/4]$. Consider
\begin{align*}
    u_0 (x,t) = f(t) \delta_0^{1/2}\cos ( \la_0 x_1) e_2\,, \,\,\,\, p_0(x,t) = 0\,,\,\,\,\, R_0(x,t) = f'(t) \frac{\delta_0^{1/2}}{\la_0}  \begin{pmatrix}
        0 & \sin ( \la_0 x_1)\\ \sin(\la_0 x_1) & 0
    \end{pmatrix},
\end{align*}
where $(x_1, x_2)$ denote the standard coordinates on $\T^2$ and $(e_1,e_2)$ are the associated unit vectors. It can be checked directly that the tuple $(u_0,p_0,R_0)$ solves the Euler-Reynolds system~\eqref{ER}. 

We have 
\begin{equation*}
    \|u_0\|_0 \leq M\delta_0^{1/2} \leq M(1 - \delta_0^{1/2}),
\end{equation*}
provided $a$ is chosen sufficiently large so that $\delta_0^{1/2} < 1/2$. The estimate \eqref{ind_est_1} is, thus, satisfied. Moreover, for any $N \geq 1$, 
\begin{equation*}
    \|u_0\|_N \leq M\delta_0^{1/2} \lambda_0^N,
\end{equation*}
and, so, \eqref{ind_est_2} also holds. Also, for any $N \geq 0$,
\begin{equation*}
    \|R_0\|_N \leq 2 \sup_t |f'(t)| \frac{\delta_0^{1/2}}{\lambda_0} \lambda_0^N.
\end{equation*}
Since it holds that $(2b - 1)\beta < 1/3$, we can ensure that 
\begin{equation*}
    2 \sup_t |f'(t)| < \delta_1 \delta_0^{-1/2} \lambda_0^{1/2},
\end{equation*}
by choosing $a$ sufficiently large. Then, 
\begin{equation*}
    \|R_0\|_N \leq \delta_1 \lambda_0^{-1/2} \lambda_0^N,
\end{equation*}
and it follows that \eqref{ind_est_4} holds, since we have chosen $\alpha < 1/4$. For the estimate concerning the material derivative, we calculate
\begin{equation*}
    \partial_t R_0 + u_0 \cdot \nabla R_0 = f''(t) \frac{\delta_0^{1/2}}{\lambda_0} \begin{pmatrix}
        0 & \sin ( \la_0 x_1)\\ \sin(\la_0 x_1) & 0
    \end{pmatrix}.
\end{equation*}
In order to ensure that \eqref{ind_est_5} is satisfied, it suffices to choose $a$ large enough so that 
\begin{equation*}
    2\sup_t |f''(t)| < \delta_1 \delta_0^{-1/2} \lambda_0^{1/2} (\delta_0^{1/2} \lambda_0) = \delta_1 \lambda_0^{3/2}.
\end{equation*}
Finally, we note that $\supp_t R_0 \subset [-7/4, 7/4] \setminus (-5/4, 5/4)$, and, thus, the condition \eqref{ind_est_6} is satisfied provided 
\begin{equation*}
    (\delta_0^{1/2} \lambda_0)^{-1} < \frac{1}{4},
\end{equation*}
which, once again, can be guaranteed by the choice of $a$.

We now finally fix $a$ so that all of the wanted inequalities are satisfied, and conclude that the tuple $(u_0, p_0, R_0)$ satisfies all the requirements to be the base case for the inductive proposition \ref{Main_prop}. Let, then, $\{(u_q, p_q, R_q)\}$ be the sequence of solutions to the Euler-Reynolds system \ref{ER} given by the proposition. Equation \eqref{Main_prop_eqn} implies that 
\begin{equation*}
    \|u_{q+1} - u_q\|_\gamma \lesssim \|u_{q+1} - u_q\|_0^{1-\gamma} \|u_{q+1} - u_q\|_1^\gamma \lesssim \delta_{q+1}^{1/2} \lambda_{q+1}^\gamma \lesssim \lambda_{q+1}^{\gamma - \beta}.
\end{equation*}
Therefore, $\{u_q\}$ is a Cauchy sequence in $C_tC^\gamma_x$ and, thus, it converges in this space to a velocity field $u$. Moreover, 
\begin{equation*}
    \|R_q\|_\gamma \lesssim \|R_q\|_0^{1-\gamma}\|R_q\|_1^{\gamma} \lesssim \delta_{q+1}\lambda_q^{\gamma} \lesssim \lambda_{q+1}^{\gamma - 2 \beta},
\end{equation*}
and, thus, $R_q$ converges to zero in $C_tC^{\gamma}_x$. Since $p_q$ satisfies 
\begin{equation*}
    \Delta p_q = \div \div(-u_q \otimes u_q + R_q),
\end{equation*}
it follows that $p_q - \fint p_q$ converges also to some $p$ in $C_tC_x^\gamma$, and, thus, $\nabla p_q \rightarrow \nabla p$ as distributions. We conclude, then, that $(u,p)$ is a weak solution to the Euler equations with $u \in C_tC^\gamma_x$, which, in view of \eqref{Main_prop_eqn_2}, moreover satisfies $\supp_t u \subset [-2, 2]$ and 
\begin{equation*}
    u(x,t) = \delta_0^{1/2} \cos(\lambda_0 x_1)e_2,
\end{equation*}
whenever $t \in [-1,1]$. 

The claimed regularity in time follows either by the result of \cite{IsetRegulTime}, or by the short argument given in the proof of the main theorem of \cite{cltv}. Theorem \ref{main_thm} is, therefore, proven.

\begin{rem}
    The constructed velocity field $u$ is of size $\delta_0^{1/2}$ (say, in $L^2$ or $C^0$), which becomes vanishingly small as $\beta \to 1/3$. We remark that we can obtain ``large'' solutions by simply using the scaling of the Euler equations
    $$ u(x,t) \to \Gamma u(x,\Gamma t)\,,\qquad p(x,t) \to \Gamma^2 p(x,\Gamma t) $$
    for any $\Gamma \gg 1$. The cost of the procedure is that it concentrates the temporal support. On the other hand, by a trivial modification of the construction, one can obtain solutions which are non-trivial in the interval $[-T,T]$, for arbitrary $T>0$, instead of $[-1,1]$ as achieved above.
\end{rem}

\subsection{Heuristic outline of the iteration stage}\label{heuristics}

We now present the main ideas of the proof of proposition \ref{Main_prop} at the level of heuristics. Before we begin, however, let us caution the reader that the values given below for the various parameters ($\tau_q$, $\mu_{q+1}$, $\Gamma$, etc.), as well as the definitions of the perturbations and the generated errors will not exactly match those which we will use in the proof. The reasons for these discrepancies are essentially of technical nature. The plan is the following: we first recall the temporal localization and the simple-tensor decomposition of the Reynolds stress (these are now standard in the context of Nash iterations for the Euler equations); then, we describe the temporally oscillatory profiles and the construction of the Newton perturbations; finally, we present the Nash perturbation and highlight the flow error, which is present due to the addition of the Newton perturbations and is a sharp error specific to this scheme. 

\subsubsection{Temporal localization and stress decomposition} By taking a partition of unity in time, we can assume that $R_q$ has temporal support in an interval of length $\tau_q = (\delta_q^{1/2} \lambda_q)^{-1}$, centered at some time $t_0$. Note that this localization procedure preserves the inductive estimate 
\begin{equation*}
    \|D_t R_q\|_0 \lesssim \delta_{q+1} \delta_q^{1/2} \lambda_q = \delta_{q+1} \tau_q^{-1}. 
\end{equation*}

We denote by $\Phi$ the backwards flow of $u_q$, with origin at $t_0$, which is characterized by 
\begin{equation*}
    \begin{cases}
        \partial_t \Phi + u_q \cdot \nabla \Phi = 0, \\ 
        \Phi \big|_{t = t_0} = x.
    \end{cases}
\end{equation*}
In the proof, we will in fact use the flow of a spatially mollified version of $u_q$, which we denote by $\bar u_q$. This allows us control on arbitrarily many derivatives of $\bar u_q$ in terms of the mollification parameter, at the cost of having to control various other errors. We ignore this technicality in this discussion. 

The geometric decomposition lemma (see lemma \ref{geom} in the appendix) can be employed to define: a finite set of directions $\Lambda \subset \mathbb Z^2$, which is fixed independently of the parameters of the construction; and, for each $\xi \in \Lambda$, amplitude functions $a_\xi$ such that 
\begin{equation} \label{HDeco}
    \sum_{\xi \in \Lambda} \underbrace{a_\xi^2 (\nabla \Phi)^{-1} \xi \otimes \xi (\nabla \Phi)^{-T}}_{A_\xi} = \de_{q+1} \left(\I - \frac{R_q}{\de_{q+1}}\right). 
\end{equation}
In fact, in the proof, $R_q$ will be replaced by a mollified version of itself, for the same reasons as those given above. The idea of this decomposition will become apparent once we describe the (approximate) cancellation with the low modes of the quadratic self-interaction of the Nash perturbation.

\subsubsection{Temporally oscillatory profiles and the Newton steps} Consider a set of $1$-periodic functions of time $\{g_\xi\}_{\xi \in \Lambda}$, which have pair-wise disjoint temporal supports and unit $L^2(0,1)$ norms. These profiles are used to quantify the temporal oscillations, as well as to achieve disjoint temporal supports for the Nash perturbations corresponding to different directions $\xi, \eta \in \Lambda$. We define 
\begin{equation*}
    f_\xi = 1 - g_\xi^2,
\end{equation*}
and the primitives 
\begin{equation*}
    f_\xi^{[1]}(t) = \int_0^t f_\xi(s) ds. 
\end{equation*}
Note that since the functions $f_\xi$ have zero mean on their period, the primitives $f_\xi^{[1]}$ are also periodic with the same period and are uniformly bounded on $\mathbb{R}$. Let also $\mu_{q+1} \gg \tau_q^{-1}$ be the temporal frequency parameter, which we do not fix for the moment. We can, then, define the first Newton perturbation $w_{q+1, 1}^{(t)}$ to be the solution to the Newtonian linearization of the Euler equations around $u_q$, where the forcing is augmented by the oscillatory phases $f_\xi(\mu_{q+1}\cdot)$:
\begin{equation} \label{heur_newt_eqn}
\begin{cases}
    \partial_t w_{q+1, 1}^{(t)} + u_q \cdot \nabla w_{q+1, 1}^{(t)} + w_{q+1, 1}^{(t)} \cdot \nabla u_q + \nabla p_{q+1, 1}^{(t)} = \sum_{\xi \in \Lambda} f_{\xi} (\mu_{q+1}t) \mathbb P \div A_\xi, \\ 
    \div w_{q+1,1}^{(t)} = 0, \\ 
    w_{q+1, 1}^{(t)} \big|_{t = t_0} = \frac{1}{\mu_{q+1}} \sum_{\xi \in \Lambda} f_\xi^{[1]}(\mu_{q+1}t_0) \mathbb P \div A_\xi \big|_{t = t_0}.
\end{cases}
\end{equation}
In the above, $\mathbb P$ is the Leray projection operator. We have, then, the following cancellation: 
\begin{eqnarray} \label{heuristic_cancel_Newt}
    \partial_t w_{q+1, 1}^{(t)} + u_q \cdot \nabla w_{q+1, 1}^{(t)} + w_{q+1, 1}^{(t)} \cdot \nabla u_q + \nabla p_{q+1, 1}^{(t)} + \div R_q &=& \sum_{\xi \in \Lambda} \mathbb P \div A_\xi + \div R_q \nonumber \\ 
    && \quad - \sum_{\xi \in \Lambda} g_\xi^2 \mathbb P \div A_\xi \nonumber \\ 
    & = & - \div  \sum_{\xi \in \Lambda} g_\xi^2 A_\xi + \nabla q, 
\end{eqnarray}
for some scalar function $q$.

Moreover, we remark that the first equation in \eqref{heur_newt_eqn} can be seen as a transport equation which is perturbed by the lower-order operator $w_{q+1,1}^{(t)} \cdot \nabla u_q + \nabla p_{q+1, 1}^{(t)}$.
As such, one can expect that 
\begin{eqnarray*}
    w_{q+1,1}^{(t)}(X, t)  &\approx&  \frac{1}{\mu_{q+1}} \sum_{\xi \in \Lambda} f_\xi^{[1]}(\mu_{q+1}t_0) \mathbb P \div A_\xi \big|_{t = t_0} + \int_{t_0}^{t} \sum_{\xi \in \Lambda} f_{\xi}(\mu_{q+1}s) \mathbb P \div A_\xi(X(\cdot, s), s) ds \\ 
    &=& \frac{1}{\mu_{q+1}}\sum_{\xi \in \Lambda} f_{\xi}^{[1]}(\mu_{q+1} t) \mathbb P \div A_\xi(X, t) - \sum_{\xi \in \Lambda} \int_{t_0}^{t} f_{\xi}^{[1]}(\mu_{q+1}s) \frac{D_t \mathbb P \div A_\xi}{\mu_{q+1}}(X(\cdot, s), s) ds,
\end{eqnarray*}
where $X$ denotes the Lagrangian flow of $u_q$ starting at $t = t_0$. We do not give a precise meaning to the symbol $\approx$ used above, but one can expect that that the two sides of the equation satisfy the same estimates. Since the cost of the material derivative applied to $\mathbb P \div A_\xi$ is expected to be $\tau_q^{-1} \ll \mu_{q+1}$, we conclude that the last term is negligible, and, thus, 
\begin{equation} \label{approx_CL}
    w_{q+1, 1}^{(t)} \approx \frac{1}{\mu_{q+1}} \sum_{\xi \in \Lambda} f_{\xi}^{[1]} \mathbb P \div A_\xi. 
\end{equation}
The expression on the right-hand-side is a variant of the perturbations introduced by Cheskidov and Luo (\cite{cluo}, \cite{cluo2}). It is in this sense that the Cheskidov-Luo perturbation can be seen as a first-order approximation of the (first) Newton perturbation used in this paper. 

We can infer that $w_{q+1, 1}^{(t)}$ satisfies the estimate 
\begin{equation*}
    \|w_{q+1, 1}^{(t)}\|_0 \lesssim \frac{\delta_{q+1} \lambda_q}{\mu_{q+1}}. 
\end{equation*}
The associated error, which will be part of the new stress $R_{q+1}$ is, then, the nonlinear part of the Euler operator: 
\begin{equation} \label{heuristic_Newton_error}
    \|R_{q+1}^{\text{Newton}}\|_0 = \|w_{q+1, 1}^{(t)} \otimes w_{q+1, 1}^{(t)}\|_0 \lesssim \bigg(\frac{\delta_{q+1}\lambda_q}{\mu_{q+1}}\bigg)^2. 
\end{equation}

We have, therefore, achieved the construction of a perturbation which satisfies the same estimates as that of Cheskidov and Luo, but for which all of the terms involved in the application of the linearized Euler operator contribute to the cancellation of $R_q$, not only the time derivative $\partial_t w_{q+1, 1}^{(t)}$. The cost of the procedure, however, consists in the fact that, unlike the right-hand-side of \eqref{approx_CL}, $w_{q+1,1}^{(t)}$ itself cannot be globally extended in time: specifically, it is required that we glue together the temporally localized perturbations defined in \eqref{heur_newt_eqn}. For this purpose, we let $\tilde \chi$ be a smooth cut-off function such that $\tilde \chi = 1$ on $\cup_\xi \supp A_\xi$, while $|\partial_t \tilde \chi| \lesssim \tau_q^{-1}$. The full perturbation will, then, be the superposition of perturbations $\tilde \chi w_{q+1, 1}^{(t)}$ corresponding to the temporal localizations of $R_q$. Therefore, an error related to this gluing procedure is incurred: 
\begin{equation*}
    R_q^{\text{glue}} = \div^{-1} \partial_t \tilde \chi w_{q+1, 1}^{(t)}. 
\end{equation*}

This error is reminiscent of the gluing error introduced by Isett in \cite{Isett18}, and, indeed, the techniques developed to deal with that error also apply in this context (specifically, in the proof we will use arguments similar to those employed in \cite{cltv}). From \eqref{approx_CL}, 
we have 
\begin{equation*}
    \div^{-1} w_{q+1, 1}^{(t)} \approx \frac{1}{\mu_{q+1}} \sum_{\xi \in \Lambda} f_\xi^{[1]} \div^{-1} \mathbb P \div A_\xi,
\end{equation*}
and we note that the operator $\div^{-1} \mathbb P \div$ is of zero order. Therefore, we expect the estimate
\begin{equation*}
    \|R_{q}^{\text{glue}}\|_0 \lesssim \frac{1}{\mu_{q+1}}|\partial_t \tilde \chi| \|A_\xi\|_0 \lesssim \frac{\delta_{q+1} \tau_q^{-1}}{\mu_{q+1}}. 
\end{equation*}
Moreover, 
\begin{equation*}
    \|D_t w_{q+1, 1}^{(t)}\|_{0, \supp \partial_t \tilde \chi} \lesssim \|w_{q+1, 1}^{(t)} \cdot \nabla u_q + \nabla p_{q+1}^{(t)}\|_0 + \underbrace{\bigg\|\sum_\xi f_\xi \mathbb P \div A_\xi\bigg\|_{0, \supp \partial_t \tilde \chi}}_{=0} \lesssim \|w_{q+1, 1}^{(t)}\|_0 \tau_q^{-1},
\end{equation*}
where the second term above vanishes since $\supp \partial_t \tilde \chi \cap \supp A_\xi = \emptyset$. In other words, a material derivative of the perturbation costs $\tau_q^{-1}$ on $\supp \partial_t \tilde \chi_k$. When the material derivative falls on $\partial_t \tilde \chi$, a loss of $\tau_q^{-1}$ is, likewise, incurred. Therefore, we expect 
\begin{equation*}
    \|D_t R_{q}^{\text{glue}}\|_0 \lesssim \frac{\delta_{q+1}\tau_q^{-1}}{\mu_{q+1}} \tau_q^{-1}. 
\end{equation*}
These estimates for $R_q^{\text{glue}}$ will not be good enough to place it into $R_{q+1}$. However, compared to $R_q$, $R_q^{\text{glue}}$ has improved estimates by a factor of $\tau_q^{-1}/\mu_{q+1}$. If, then, 
\begin{equation*}
    \frac{\tau_q^{-1}}{\mu_{q+1}} = \bigg(\frac{\lambda_q}{\lambda_{q+1}}\bigg)^{\epsilon}, 
\end{equation*}
for some $0 < \epsilon \ll 1$, the procedure can be repeated by replacing $R_q$ with $R_q^{\text{glue}}$ throughout the construction and by taking a new family of profiles $\{g_{\xi, 1}\}$, which have supports disjoint from each other and from the supports of the previously used profiles. This will cancel the gluing error $R_{q}^{\text{glue}}$ up to a remainder as in \eqref{heuristic_cancel_Newt}, and will, in turn, give rise to a new gluing error, $R_q^{\text{glue, } 2}$, that satisfies estimates which are further improved. This process can be continued inductively until, after finitely many iterations, the remaining gluing error is sufficiently small to be placed into $R_{q+1}$.

These are the iterative Newton steps which we describe rigorously in section \ref{section_Newton}. After $\Gamma \approx \epsilon^{-1}$ steps, the final gluing error will satisfy 
\begin{equation*}
    \|R_{q}^{\text{glue, } \Gamma}\|_0 \lesssim \delta_{q+1} \bigg(\frac{\tau_q^{-1}}{\mu_{q+1}}\bigg)^{\Gamma} \lesssim \delta_{q+1}\frac{\lambda_q}{\lambda_{q+1}},
\end{equation*}
which is, indeed, (more than) small enough to be compatible with $C^{1/3-}_x$ regularity. Since, as we will see upon fixing $\mu_{q+1}$, $\epsilon$ will only depend on $\beta$, the number of required Newton steps is finite and fixed throughout the stages of the iteration.  

\subsubsection{The Nash step and the flow error} We have, thus, constructed a new smooth solution to the Euler-Reynolds system \eqref{ER} with velocity field 
\begin{equation*}
    u_{q, \Gamma} = u_q + w_{q+1}^{(t)} = u_q + \sum_{n = 1}^\Gamma w_{q+1, n}^{(t)}. 
\end{equation*}
Moreover, the error $R_q$ has been cancelled out with the exception of a remainder 
\begin{equation*}
    R_q^{\text{rem}} = - \sum_{\xi \in \Lambda} g_\xi^2 A_\xi. 
\end{equation*}
Actually, as already noted, each Newton step will have left behind its own remainder error, but the one above is that which satisfies the worst estimates, and, so, we choose to ignore the other terms in this heuristic discussion. 

Let $\tilde \Phi$ be the backwards flow of $u_{q, \Gamma}$ with origin at $t = t_0$. We can define, then, the (principal part of) the Nash perturbation by 
\begin{equation*}
    w_{q+1}^{(p)} = \sum_{\xi \in \Lambda} g_\xi \bar a_\xi (\nabla \tilde \Phi)^{-1} \mathbb W_\xi (\lambda_{q+1} \tilde \Phi),
\end{equation*}
where $\bar a_\xi$ are related to $\tilde \Phi$ in similar fashion as $a_\xi$ are related to $\Phi$, and $\mathbb W_\xi$ are fixed shear flows in the directions $\xi \in \Lambda$ (see section \ref{Nash_Construction} for the precise choice used in the proof), which satisfy 
\begin{equation*}
    \fint \mathbb W_\xi \otimes \mathbb W_\xi = \xi \otimes \xi.  
\end{equation*}
In reality, the discrepancy between $\bar a_\xi$ and $a_\xi$ will also be due to a mollification along the flow of the decomposed stresses (see section \ref{Section_Moli_Along_flow}), but we ignore this technicality here. We use the flow $\tilde \Phi$ instead of $\Phi$ in the definition above in order to make sure that there is no significant interaction between the Newton perturbation $w_{q+1}^{(t)}$ and $w_{q+1}^{(p)}$.  Also, as is the case in all Nash iteration schemes starting with \cite{DLS13}, \cite{DLS14}, the full Nash perturbation will include a divergence-corrector term, which is designed to ensure the validity of the incompressibility condition. In view of the fact that $(\nabla \tilde \Phi)^{-1} \mathbb W_\xi(\lambda_{q+1} \tilde \Phi)$ is divergence-free, this will be a much smaller correction, and so we ignore it in the heuristic discussion. 

The main idea of the Nash perturbation is the following (approximate) quadratic cancellation: 
\begin{eqnarray*}
    R_{q}^\text{rem} + w_{q+1}^{(p)} \otimes w_{q+1}^{(p)} &=& - \sum_{\xi} g_\xi^2 A_\xi + \sum_\xi g_\xi^2 \bar a_\xi^2 (\nabla \tilde \Phi)^{-1} \mathbb W_\xi \otimes \mathbb W_\xi (\lambda_{q+1} \tilde \Phi) (\nabla \tilde \Phi)^{-T} \\ 
    & = & - \sum_\xi g_\xi^2 A_\xi + \sum_\xi g_\xi^2 \underbrace{\bar a_\xi^2 (\nabla \tilde \Phi)^{-1} \xi \otimes \xi (\nabla \tilde \Phi)^{-T}}_{\bar A_\xi} \\ 
    && + \sum_{\xi} g_{\xi}^2 \bar a_\xi^2 (\nabla \tilde \Phi)^{-1} (\mathbb P_{\neq 0} \mathbb W_\xi \otimes \mathbb W_\xi) (\lambda_{q+1} \tilde \Phi) (\nabla \tilde \Phi)^{-T} \\ 
    &=& \underbrace{\sum_\xi g_\xi^2 (\bar A_\xi - A_\xi)}_{R_{q+1}^{\text{flow}}} + \sum_{\xi} g_{\xi}^2 \bar a_\xi^2 (\nabla \tilde \Phi)^{-1} (\mathbb P_{\neq 0} \mathbb W_\xi \otimes \mathbb W_\xi) (\lambda_{q+1} \tilde \Phi) (\nabla \tilde \Phi)^{-T},
\end{eqnarray*}
where $\mathbb P_{\neq 0}$ denotes the projection onto non-zero Fourier modes, while $R_{q+1}^{\text{flow}}$ is the flow error, which is due to the disagreement between $\Phi$ and $\tilde \Phi$. In the above, we have used the fact that $g_\xi$ have pair-wise disjoint supports in the vanishing of the off-diagonal terms of $w_{q+1}^{(p)} \otimes w_{q+1}^{(p)}$. Therefore, the low frequency self-interactions of the Nash perturbations cancel the remainder $R_{q}^{\text{rem}}$ modulo the flow error $R_{q+1}^{\text{flow}}$.  

The errors generated by the Nash perturbation are, then, those obtained from the application of the linearized Euler operator around $u_{q, \Gamma}$, and the high frequency interaction above. It can be seen that the transport error 
\begin{equation*} 
    R_{q+1}^{\text{transport}} = \div^{-1} \bigg( (\partial_t + u_{q, \Gamma} \cdot \nabla) w_{q+1}^{(p)} \bigg)
\end{equation*}
will include a term which is characteristic to the use of temporal oscillations: 
\begin{equation} \label{heuristic_transport}
    \bigg\|\sum_\xi \partial_t g_{\xi} \bar a_\xi \div^{-1} \big(( \nabla \tilde \Phi)^{-1} \mathbb W_\xi(\lambda_{q+1} \tilde \Phi)\big)\bigg\|_0 \lesssim \frac{\mu_{q+1} \delta_{q+1}^{1/2}}{\lambda_{q+1}},
\end{equation}
where in the estimate above we use the heuristics that $\bar a_\xi \approx a_\xi \approx R_q^{1/2}$ and that an inverse-divergence applied to a $\lambda_{q+1}$-oscillatory function gains a factor of $\lambda_{q+1}^{-1}$. 

By optimizing between \eqref{heuristic_transport} and the Newton error \eqref{heuristic_Newton_error}, we are led to fixing 
\begin{equation*}
    \mu_{q+1} = \delta_{q+1}^{1/2} \lambda_q^{2/3} \lambda_{q+1}^{1/3}. 
\end{equation*}
Note, then, that 
\begin{equation*}
    \frac{\tau_q^{-1}}{\mu_{q+1}} = \bigg(\frac{\lambda_q}{\lambda_{q+1}}\bigg)^{1/3 - \beta},
\end{equation*}
and, thus, the Newton iteration can be carried out with $\epsilon = 1/3 - \beta > 0$ and $\Gamma \approx (1/3 - \beta)^{-1}$. Moreover, the two errors are estimated as 
\begin{equation*}
    \|R_{q+1}^\text{Newton}\|_0 + \|R_{q+1}^{\text{transport}}\|_0 \lesssim \delta_{q+1} \bigg(\frac{\lambda_q}{\lambda_{q+1}}\bigg)^{2/3}.
\end{equation*}
In view of the condition $\|R_{q+1}\| \leq \delta_{q+2}$, which is necessary for the inductive propagation of the estimates of proposition \ref{Main_prop}, this implies
\begin{equation*}
    \delta_{q+1} \bigg(\frac{\lambda_q}{\lambda_{q+1}}\bigg)^{2/3} \leq \delta_{q+2} \implies \beta \leq \frac{1}{3b},
\end{equation*}
which is compatible with $C_x^{1/3-}$ regularity. 

With the exception of the flow error, all of the other terms appear in other Nash iteration schemes, and their compatibility with $C_x^{1/3-}$ regularity is well-known. It remains, then, to present the heuristic estimation for $R_{q+1}^{\text{flow}}$. Since the discrepancy between $A_\xi$ and $\bar A_\xi$ is essentially due to the incongruity between $\Phi$ and $\tilde \Phi$, we can expect a bound of the form 
\begin{equation*}
    \|R_{q+1}^{\text{flow}}\|_0 \lesssim \delta_{q+1} \|\nabla \Phi - \nabla \tilde \Phi\|_0.
\end{equation*}
In view of the fact that 
\begin{equation*}
    \|w_{q+1}^{(t)}\|_0 \lesssim \delta_{q+1}^{1/2} \bigg(\frac{\lambda_q}{\lambda_{q+1}}\bigg)^{1/3},
\end{equation*}
we can regard $u_{q, \Gamma}$ as a slightly perturbed version $u_q$, so we can expect also a stability result related to the generated flows $\tilde \Phi$ and $\Phi$. Indeed, it holds that 
\begin{equation*}
    \begin{cases}
        \partial_t(\Phi - \tilde \Phi) + u_q \cdot \nabla (\Phi - \tilde \Phi) = w_{q+1}^{(t)} \cdot \nabla \tilde \Phi, \\ 
        (\Phi - \tilde \Phi) \big|_{t = t_0} = 0.
    \end{cases}
\end{equation*}
Standard estimates for transport equations imply that on time-scales of size $\tau_q$,
\begin{equation*}
    \|\nabla \Phi - \nabla \tilde \Phi\|_0 \lesssim \tau_q \delta_{q+1}^{1/2} \bigg(\frac{\lambda_q}{\lambda_{q+1}}\bigg)^{1/3} \lambda_q,
\end{equation*}
which yields the following estimate for the flow error: 
\begin{equation*}
    \|R_{q+1}^{\text{flow}}\|_0 \lesssim \delta_{q+1} \bigg(\frac{\delta_{q+1}}{\delta_q}\bigg)^{1/2}\bigg(\frac{\lambda_q}{\lambda_{q+1}}\bigg)^{1/3} = \delta_{q+1} \bigg( \frac{\lambda_q}{\lambda_{q+1}}\bigg)^{1/3 + \beta}.
\end{equation*}
Then, arguing as before, we obtain 
\begin{equation*}
    \delta_{q+1} \bigg( \frac{\lambda_q}{\lambda_{q+1}}\bigg)^{1/3 + \beta} \leq \delta_{q+2} \implies \beta \leq \frac{1}{3(2b-1)},
\end{equation*}
which is, once again, just barely compatible with $C_x^{1/3-}$ regularity. We conclude this section with a remark: it turns out that the flow error is always larger than the Newton error. As such, one can alternatively define $\mu_{q+1}$ by balancing the transport error with the flow error, and not with the Newton error as described above. Such a definition will not bring about any significant changes, as all of the estimates will differ by factors of $(\lambda_q/\lambda_{q+1})^{1/3 - \beta}$. On the other hand, it can be argued that the Newton error is conceptually the fundamental object, and not the flow error. Indeed, $R_{q+1}^\text{Newton}$ is the nonlinear correction to a linear approximation scheme, while $R_{q+1}^{\text{flow}}$ is simply due to the particular stress decomposition implemented in the Nash iteration. For this reasons, we opt for defining $\mu_{q+1}$ as above.

\begin{rem}
In \cite{cltv}, the authors modify the iteration used in \cite{Isett18} and show that any given smooth energy profile $e:[0,1] \to\mathbb{R}_{> 0}$ can be achieved by three-dimensional flexible Euler flows: i.e. the solution can be designed to satisfy $\int_{\mathbb{T}^3} |u(x,t)|^2\,dx = e(t)$. It is an interesting question whether the scheme in the present paper admits such a modification. In order to point out the difficulties, let us first briefly recall the ideas used in \cite{cltv}. The energy increment due to the addition of the Nash perturbation is given (approximately) by 
\begin{equation*}
    \delta e_{q+1}(t) \approx 2\int_{\mathbb T^3} w_{q+1}^{(p)} \cdot u_q + \int_{\mathbb{T}^3} |w_{q+1}^{(p)}|^2.
\end{equation*}
The Nash perturbation has the form 
\begin{equation*}
    w_{q+1}^{(p)} = \eta(x,t)\sum_{\xi \in \Lambda} a_{\xi} (\nabla \Phi)^{-1} \mathbb W_\xi(\lambda_{q+1}\Phi),
\end{equation*}
where $\eta$ is a ``squiggling stripes'' cut-off function (see \cite{cltv} for details), $a_\xi$ are low frequency amplitudes, and $\mathbb W_\xi$ are Mikado (pipe) flows. Stationary phase arguments show that the high-frequency terms in the integrals above are lower-order, and, thus, one has 
\begin{equation*}
    \delta e_{q+1}(t) \approx \int_{\mathbb T^3} \text{trace}\bigg(\eta^2(x,t)\sum_{\xi \in \Lambda}a_\xi^2 (\nabla \Phi)^{-1} \langle \mathbb W_\xi \otimes \mathbb W_\xi \rangle (\nabla \Phi)^{-T}\bigg),
\end{equation*}
where $\langle \mathbb W_\xi \otimes \mathbb W_\xi \rangle = \xi \otimes \xi$ denotes the mean. One can, therefore, control the frequency increment by controlling the trace in a decomposition similar to \eqref{HDeco}. In other words, the traceless part of the decomposition is used to eliminate the errors from being a solution, while the trace is employed to achieve an energy increment. There are two main points: first, Nash perturbations corresponding to each direction have to be present at each time-slice; and second, there is a point-wise in space cancellation of the traceless part, which hinges on the fact there is no support separation in the expression above. 

In contrast, the energy increment that could be achieved with the construction of the present paper is given by
\begin{equation*}
    \delta e_{q+1}(t) \approx \int_{\mathbb{T}^2} \text{trace}\bigg(\sum_{\xi \in \Lambda} g_\xi^2(\mu_{q+1}t) a_\xi^2 (\nabla \Phi)^{-1} \xi \otimes \xi (\nabla \Phi)^{-T}\bigg).
\end{equation*}
As such, modifications are needed to satisfy either of the two main points described above. Indeed, regarding the first point, Nash perturbations are only present on the supports of the oscillatory profiles $g_\xi$, and not at every time-slice. It seems likely that this can be rectified by replacing the purely time-dependent $g_\xi(\mu_{q+1} t)$ with ``squiggling stripes'' variants $g_\xi(x,\mu_{q+1 }t)$. The second point seems to raise more serious issues: the fundamental reason for which the scheme of the present paper works is that the supports corresponding to the different directions are disjoint -- this obstructs the point-wise cancellation of the traceless part (note that the cut-off $\eta(x,t)$ is the same for all directions $\xi \in \Lambda$). It seems that substantial additions are needed to overcome this obstruction. 

While we do believe that the energy profile can be prescribed, the considerations above show that constructing solutions satisfying this notion of flexibility raises important technical issues, and, as such, lies beyond the scope of this paper. 
\end{rem}

% \subsection 

\section{The Newton steps} \label{section_Newton}

\subsection{Preliminary: Spatial mollification} As in all Nash iteration schemes for the Euler equations starting with \cite{DLS13} and \cite{DLS14}, we begin the construction by mollifying the velocity field and the Reynolds stress. The aim of this procedure is to yield control of the higher order spatial derivatives of the solution to the Euler-Reynolds system and, thus, bypass the loss of derivative problem. 

Let $\zeta$ be a symmetric spatial mollifier (note, in particular, that $\zeta$ has vanishing first moments), and fix the spatial mollification scale
\begin{equation*}
    \ell_q = (\lambda_q \lambda_{q+1})^{-1/2}.
\end{equation*}
We denote
\begin{equation*}
    \bar u_q = u_q * \zeta_{\ell_q}, 
\end{equation*}
\begin{equation*}
    R_{q, 0} = R_q * \zeta_{\ell_q}.
\end{equation*} 
and record the relevant estimates in the following lemma. Here and throughout the paper we use the notation 
\begin{equation*}
    \bar D_t = \partial_t + \bar u_q \cdot \nabla
\end{equation*}
for the material derivative corresponding to $\bar u_q$.

\begin{lem} \label{smoli_estim}
  Assume $u_q$ and $R_q$ satisfy~\eqref{ind_est_1}-\eqref{ind_est_5}. Then, the following estimates hold: 
\begin{equation} \label{smoli_1}
    \|\bar u_q\|_N \lesssim \delta_q^{1/2} \lambda_q^N, \,\,\, \forall N \in \{1,2,...,L\},
\end{equation}
\begin{equation} \label{smoli_2}
    \|R_{q,0}\|_N \lesssim \delta_{q+1} \lambda_q^{N-2\alpha}, \,\,\, \forall N \in \{0, 1,..., L\},
\end{equation}
\begin{equation} \label{smoli_3}
    \|\bar D_t R_{q, 0}\|_N \lesssim \delta_{q+1} \delta_q^{1/2} \lambda_q^{N+1-2\alpha}, \,\,\, \forall N \in \{0, 1,..., L-1\}
\end{equation}
\begin{equation} \label{smoli_4}
    \|\bar u_q\|_{N + L} \lesssim \delta_q^{1/2} \lambda_q^{L} \ell_q^{-N}, \,\,\, \forall N \geq 0,
\end{equation}
\begin{equation} \label{smoli_5}
    \|R_{q,0}\|_{N+L} \lesssim \delta_{q+1} \lambda_q^{L-2\alpha} \ell_q^{-N}, \,\,\, \forall N \geq 0,
\end{equation}
\begin{equation} \label{smoli_6}
    \|\bar D_t R_{q,0}\|_{N+L-1} \lesssim \delta_{q+1} \delta_q^{1/2}\lambda_q^{L-2\alpha} \ell_q^{-N}, \forall N \geq 0,
\end{equation}
where the implicit constants depend on $M$ and $N$ \footnote{Here, and throughout, by dependence on $N$, we mean dependence on the norm being estimated. Strictly speaking, the constant in the estimate of $\|\cdot\|_{N+L}$ will depend on $N+L$.}. 
\end{lem}

\begin{proof}
    In light of the inductive assumptions, standard mollification estimates immediately imply all but inequalities \eqref{smoli_3} and \eqref{smoli_6}. Note that
    \begin{equation*}
        \| (\partial_t + \bar u_q \cdot \nabla) R_{q, 0} \|_N \lesssim \|[(\partial_t + u_q \cdot \nabla)R_q] * \zeta_{\ell_q}\|_N + \|\bar u_q \cdot \nabla R_{q, 0} - (u_q \cdot \nabla R_q)* \zeta_{\ell_q}\|_N.
    \end{equation*}
    The first term clearly obeys the wanted estimates by the inductive assumptions. For the second, the Constantin-E-Titi commutator estimate of proposition \ref{CET_comm} implies, for $N \leq L-2$,  
    \begin{eqnarray*}
         \|\bar u_q \cdot \nabla R_{q, 0} - [u_q \cdot \nabla R_q]* \zeta_{\ell_q}\|_N &\lesssim & \ell_q^{2}\big( \|u_q\|_1 \|\nabla R_q\|_{N+1} + \|u_q\|_{N+1}\|\nabla R_q\|_{1} \big) \\ 
         &\lesssim& \delta_{q+1} \delta_q^{1/2} \lambda_q^{N+3-2\alpha} \ell_q^2 \\ 
         &\lesssim& \delta_{q+1} \delta_q^{1/2} \lambda_q^{N+1-2\alpha},
    \end{eqnarray*}
    while for $N \geq L-1$,
    \begin{eqnarray*}
         \|\bar u_q \cdot \nabla R_{q, 0} - [u_q \cdot \nabla R_q]* \zeta_{\ell_q}\|_N &\lesssim & \ell_q^{L-N}\big( \|u_q\|_1 \|\nabla R_q\|_{L-1} + \|u_q\|_{L-1}\|\nabla R_q\|_{1} \big) \\ 
         &\lesssim& \delta_{q+1} \delta_q^{1/2} \lambda_q^{L+1-2\alpha} \ell_q^{L-N} \\ 
         &\lesssim& \delta_{q+1} \delta_q^{1/2} \lambda_q^{L-2\alpha} \ell_q^{L-1-N},
    \end{eqnarray*}
    and the conclusion follows. 
\end{proof}

\subsection{Transport estimates} We collect now standard estimates on the Lagrangian and backwards flows of $\bar u_q$. For $t \in \mathbb R$, the backwards flow $\Phi_t : \T^2 \times \R \to \T^2$ starting at $t$ is defined by 
\begin{equation} \label{Flow_t}
    \begin{cases}
        \partial_s \Phi_t(x,s) + \bar u_q(x,s) \cdot \nabla \Phi_t(x,s) = 0 \\ 
        \Phi_t \big|_{s = t}(x) = x,
    \end{cases}
\end{equation}
and the Lagrangian flow $X_t$ is defined by 
\begin{equation} \label{Lagr_t}
    \begin{cases}
        \frac{d}{ds}X_t(\alpha, s) = \bar u_q(X_t(\alpha, s), s) \\ 
        X_t(\alpha, t) = \alpha.
    \end{cases}
\end{equation}

\begin{lem} \label{Flow_estim}
    Let $t \in \mathbb R$ and $\tau \leq  \|\bar u_q\|_1^{-1}$. Let $\Phi_t$ be defined by \eqref{Flow_t}, and let $X_t$ denote the corresponding Lagrangian flow \eqref{Lagr_t}. Then, for any $|s - t| < \tau$, 
    \begin{equation} \label{Flow_estim_1}
        \|(\nabla \Phi_t)^{-1}(\cdot, s)\|_N + \|\nabla \Phi_t (\cdot, s)\|_N \lesssim \lambda_q^N, \,\,\, \forall N \in\{0,1,..., L-1\},
    \end{equation}
    \begin{equation} \label{Flow_estim_2}
        \|\bar D_t (\nabla \Phi_t)^{-1}(\cdot, s)\|_N + \|\bar D_t \nabla \Phi_t (\cdot, s)\|_N \lesssim \delta_q^{1/2} \lambda_q^{N+1}, \,\,\, \forall N \in\{0,1,..., L-1\},
    \end{equation}
    \begin{equation} \label{Lagr_estim_1}
         \|D X_t(\cdot, s)\|_N  \lesssim \lambda_q^N, \,\,\, \forall N \in \{0,1,..., L-1\},
    \end{equation}
    \begin{equation} \label{Flow_estim_3}
        \|(\nabla \Phi_t)^{-1}(\cdot, s)\|_{N+L-1} + \|\nabla \Phi_t(\cdot, s)\|_{N+L-1}  \lesssim \lambda_q^{L-1} \ell_q^{-N}, \, \, \, \forall N \geq 0,
    \end{equation}
    \begin{equation} \label{Flow_estim_4}
        \|\bar D_t (\nabla \Phi_t)^{-1}(\cdot, s)\|_{N+L-1} + \|\bar D_t \nabla \Phi_t(\cdot, s)\|_{N+L-1}  \lesssim \delta_q^{1/2} \lambda_q^L \ell_q^{-N}, \, \, \, \forall N \geq 0,
    \end{equation}
    \begin{equation} \label{Lagr_estim_2}
        \|D X_t(\cdot, s)\|_{N+L-1} \lesssim \lambda_q^{L-1} \ell_q^{-N}, \, \, \, \forall N \geq 0,
    \end{equation}
    where the implicit constants depend on $M$ and $N$.   
\end{lem}

\begin{proof}
    \textit{Estimates on $\nabla \Phi_t$.}
    The spatial derivative estimates on $\nabla \Phi_t$ follow directly from proposition \ref{transport_estim} and lemma \ref{smoli_estim}. For the material derivatives, we note that $\bar D_t \nabla \Phi_t = - \nabla \bar u_q \nabla \Phi_t$, which implies 
    \begin{equation*}
        \|\bar D_t \nabla \Phi_t\|_N \lesssim \|\bar u_q\|_{N+1} \|\nabla \Phi_t\|_0 + \|\bar u_q\|_{1} \|\nabla \Phi_t\|_N,
    \end{equation*}
    and the result follows from lemma \ref{smoli_estim} and the spatial derivative estimates of $\nabla \Phi_t$. 

    \textit{Estimates on $D X_t$.} By proposition \ref{comp_estim}, and the definition of $X_t$, we have 
    \begin{equation*}
        \bigg\|\frac{d}{dt} D^N X_t\bigg\|_0 \lesssim \|D \bar u_q\|_0 \|D X_t\|_{N-1} + \| \bar u_q\|_N \|D X_t\|_0^N.
    \end{equation*}
    The case $N=1$ now follows by Gr\"onwall's inequality, which implies, for $N > 1$, 
    \begin{equation*}
       \bigg\|\frac{d}{dt} D X_t\bigg\|_{N-1} \lesssim \|D \bar u_q\|_0 \|D X_t\|_{N-1} + \|\bar u_q\|_N.
    \end{equation*}
    Applying Gr\"onwall again, we obtain \eqref{Lagr_estim_1} and \eqref{Lagr_estim_2} from the results of lemma \ref{smoli_estim}.

    \textit{Estimates on $(\nabla \Phi_t)^{-1}$.} Finally, we note that 
    \begin{equation*}
        (\nabla \Phi_t)^{-1}(x, s) = DX_t (\Phi_t(x,s), s).
    \end{equation*}
    Therefore, 
    \begin{equation*}
        \|(\nabla \Phi_t)^{-1}\|_0 \lesssim 1,
    \end{equation*}
    and, for $N \geq 1$, proposition \ref{comp_estim} implies  
    \begin{equation*}
        \|D^N (\nabla \Phi_t)^{-1}\|_0 \lesssim \|DX_t\|_1 \|\nabla \Phi_t\|_{N-1} + \|DX_t\|_N \|\nabla \Phi_t\|_0^N,
    \end{equation*}
    from which \eqref{Flow_estim_1} and \eqref{Flow_estim_3} follow. 

    Finally, 
    \begin{eqnarray*}
        \bar D_t (\nabla \Phi_t)^{-1}(x,s) &=& \big(\frac{d}{ds} DX_t \big)(\Phi_t(x,s),s) \\ 
        &=& D \bar u_q(x,s) DX_t (\Phi_t(x,s),s)  \\ 
        &=& D \bar u_q(x,s) (\nabla \Phi_t)^{-1}(x,s).
    \end{eqnarray*}
    Therefore, \eqref{Flow_estim_2} and \eqref{Flow_estim_4} follow from \eqref{Flow_estim_1} and \eqref{Flow_estim_3} together with lemma \ref{smoli_estim}. 
\end{proof}

\subsection{Partition of unity, cut-offs, temporally oscillatory profiles} We now define the various time-dependent functions which will be used in the construction of the iterative Newton perturbations. For this purpose, we introduce the temporal parameter 
\begin{equation*}
    \tau_q = \frac{1}{\delta_q^{1/2} \lambda_q \lambda_{q+1}^\alpha}.
\end{equation*}
We remark here that $\tau_q$ is chosen so as to satisfy 
\begin{equation*}
    \|\bar u_q\|_{1+\alpha} \tau_q \lesssim \bigg(\frac{\lambda_q}{\lambda_{q+1}}\bigg)^\alpha \lesssim 1. 
\end{equation*}
Moreover, since
\begin{equation*}
     \|\bar u_q\|_1 \tau_q \leq  C \lambda_{q+1}^{-\alpha},
\end{equation*}
with $C>0$ depending only on $M$, lemma \ref{Flow_estim} holds with $\tau$ replaced with $\tau_q$ provided $a_0$ is chosen sufficiently large in terms of $M$ and $\alpha$ so that 
\begin{equation*}
    C \lambda_{q+1}^{-\alpha} \leq 1.
\end{equation*}

Let $t_k = k \tau_q$, for $k \in \mathbb{Z}$. We define a partition of unity in time by using cut-off functions $\{\chi_k\}_{k \in \mathbb Z}$ satisfying: 
\begin{itemize}
    \item The squared cut-offs form a partition of unity: 
    \begin{equation*}
        \sum_{k \in \mathbb Z} \chi_k^2(t) = 1;
    \end{equation*}
    \item $\supp \chi_k \subset (t_k - \frac{2}{3} \tau_q, t_k + \frac{2}{3} \tau_q)$. In particular, 
    \begin{equation*}
        \supp \chi_{k-1} \cap \supp \chi_{k+1} = \varnothing, \forall k \in \mathbb Z;
    \end{equation*}
    \item For any $N \geq 0$ and $k \in \mathbb Z$, 
    \begin{equation*}
        |\partial_t^N \chi_k| \lesssim \tau_q^{-N},
    \end{equation*}
    where the implicit constant depends only on $N$.
\end{itemize}
These will be used to decompose the iteratively-obtained errors into temporally localized pieces which will act as forcing for the solutions to the Newtonian linearization of the Euler equations around the background flow $\bar u_q$. Since estimates for the linearized Euler equations degenerate in time, we will glue together these temporally localized perturbations by another set of cut-off functions $\{\tilde \chi_k \}_{k \in \mathbb Z}$, which satisfy: 
\begin{itemize}
    \item $\supp \tilde \chi_k \subset (t_k - \tau_q, t_k + \tau_q)$ and $\tilde \chi_k = 1$ on $(t_k - \frac{2}{3} \tau_q, t_k + \frac{2}{3} \tau_q)$. Note in particular that 
    \begin{equation*}
        \chi_k \Tilde \chi_k = \chi_k, 
        \,\,\, \forall k \in \mathbb Z.
    \end{equation*}
    \item For any $N \geq 0$ and $k \in \mathbb Z$, 
    \begin{equation*}
        |\partial_t^N \tilde \chi_k| \lesssim \tau_q^{-N},
    \end{equation*}
    where the implicit constant depends only on $N$.
\end{itemize}

Finally, we define the time-periodic functions which will serve as building blocks for the temporally oscillatory profiles. The required number of profiles with pair-wise disjoint support is determined by the number of implemented Newton steps. Specifically, we choose 
\begin{equation}\label{def:ga}
    \Gamma = \bigg\lceil \frac{1}{1/3 - \beta} \bigg \rceil,
\end{equation}
and note that it depends only on $\beta$ and is, thus, independent of the iteration stage.

\begin{lem} \label{osc_prof}
Let $\Lambda \subset \mathbb Z^2$ be the set given by lemma \ref{geom}, and $\Gamma \in \mathbb N$. For any $\xi \in \Lambda$, there exist $2\Gamma$ smooth $1$-periodic functions $g_{\xi, e, n}, g_{\xi, o, n}:\mathbb R \rightarrow \mathbb R$ with $n \in \{1, 2,..., \Gamma\}$ such that
\begin{equation*}
    \int_0^1 g_{\xi, p, n}^2 = 1, \ \ \forall \xi \in \Lambda \text{, } p \in \{e, o\} \text{, and } n \in \{1,2,..., \Gamma\} ;
\end{equation*}
and 
\begin{equation*}
    \supp g_{\xi, p, n} \cap \supp g_{\eta, q, m} = \varnothing,
\end{equation*}
whenever $(\xi, p, n) \neq (\eta, q, m) \in \Lambda \times \{e, o\} \times \{1, 2,..., \Gamma\} $. 
\end{lem}

\begin{proof}
    Let $A = |\Lambda| \Gamma$, where $|\Lambda|$ denotes the cardinality of $\Lambda$. Choose $g \in C_c^\infty((-\epsilon, \epsilon))$, with $0 < 2 \epsilon < (4A)^{-1}$, and satisfying 
    \begin{equation*}
        \int_{\mathbb R} g^2 = 1. 
    \end{equation*}
    Then, the functions defined by 
    \begin{equation*}
        g_k (t) = g\big(t - k/(4A)\big), 
    \end{equation*}
    for $k \in \{-A, ..., A\}$, have disjoint supports which are all included in the set (-1/2, 1/2). The lemma is proven once we choose any injective assignment from the set of tuples
    $$(\xi, p, n) \in \Lambda  \times \{e, o\} \times \{1, 2,..., \Gamma\}$$
    to the set of $1$-periodic extensions of the functions $g_k$.
\end{proof}

\subsection{The Euler-Reynolds system after \texorpdfstring{$n$}{nnnn} Newton steps}

Let $n \in \{0, 1, ..., \Gamma - 1\}$. The system we will have obtained after $n$ perturbations will have the form 
\begin{equation} \label{steps}
    \begin{cases}
        \partial_t u_{q, n} + \div(u_{q, n} \otimes u_{q, n}) + \nabla p_{q,n} = \div R_{q, n} + \div S_{q, n} + \div P_{q + 1, n}, \\
        \div u_{q, n} = 0,
    \end{cases}
\end{equation}
where 
\begin{itemize}
    \item $u_{q, n}$ is the velocity which will be defined starting from $u_{q, 0} = u_q$ by adding $n$ perturbations; \\
    \item $p_{q, n}$ is the pressure, which will be inductively defined starting from $p_{q, 0} = p_q$; \\
    \item For $n \geq 1$, $R_{q, n}$ is the gluing error of the $n^{\text{th}}$ perturbation, while $R_{q, 0}$ is the already defined mollified stress;\\
    \item $S_{q, n}$ is the error which will be erased by the non-interacting highly-oscillatory Nash perturbations. It will be inductively defined starting from $S_{q, 0} = 0$;\\
    \item $P_{q + 1, n}$ is the error which is small enough to be placed into $R_{q + 1}$. It will, likewise, be inductively defined starting from $P_{q+1, 0} = R_q - R_{q, 0}$.
\end{itemize}
We remark that for $n = 0$, \eqref{steps} is just the Euler-Reynolds system \eqref{ER}.

\subsection{Construction of the Newton perturbations} 

In this section, we construct the Newton perturbation $w_{q+1, n+1}^{(t)}$ for the system \eqref{steps}, which will in turn determine all of the other quantities of the system at step $n+1$. 

We begin by decomposing the stress $R_{q, n}$ using the geometric lemma \ref{geom} adapted to the coordinates imposed by the coarse grain flow of $\bar u_q$. For this purpose, let $\Phi_k$ be the backwards flow characterized by 
\begin{equation*}
    \begin{cases}
        \partial_t \Phi_k + \bar u_q \cdot \nabla \Phi_k = 0 \\ 
        \Phi_k \big|_{t = t_k} = x,
    \end{cases}
\end{equation*}
and define, for $n \in \{0,1,...,\Gamma-1\}$, $k \in \mathbb Z$ and $\xi \in \Lambda$,
\begin{equation}\label{def.a.coeff}
    a_{\xi, k, n} =  \delta_{q+1, n}^{1/2} \chi_k \gamma_\xi \big(\nabla \Phi_k \nabla \Phi_k^T - \nabla \Phi_k \frac{R_{q, n}}{\delta_{q+1,n}} \nabla \Phi_k^T \big),
\end{equation}
where $\gamma_\xi$ are given by lemma \ref{geom}, and the amplitude parameters $\delta_{q+1, n}$ are defined by 
\begin{equation}\label{def.deqn1}
    \delta_{q+1, n} = \delta_{q+1} \bigg(\frac{\lambda_q}{\lambda_{q+1}}\bigg)^{n(1/3 - \beta)}.
\end{equation}
We now briefly argue that $a_{\xi, k, n}$ is well-defined. Proposition \ref{NewIter} below will guarantee that
\begin{equation*}
    \|R_{q, n}\|_0 \leq \delta_{q+1, n} \lambda_q^{-\alpha},
\end{equation*}
so, using lemma \ref{Flow_estim}, we see that
\begin{equation*}
    \|\nabla \Phi_k \frac{R_{q, n}}{\delta_{q+1,n}} \nabla \Phi_k^T \|_0 \lesssim \lambda_q^{-\alpha}
\end{equation*}
on $\supp \chi_k$. Note also that proposition \ref{transport_estim} immediately implies 
\begin{equation*}
    \|\I - \nabla \Phi_k\|_0 \lesssim \lambda_q^{-\alpha},
\end{equation*}
and, thus, 
\begin{equation*}
    \|\I - \nabla\Phi_k \nabla \Phi_k^T + \nabla \Phi_k \frac{R_{q, n}}{\delta_{q+1,n}} \nabla \Phi_k^T\| \lesssim \lambda_q^{-\alpha}.
\end{equation*}
Therefore, for any $\alpha>0$, we can choose $a_0$ sufficiently large so that 
$$\nabla \Phi_k \nabla \Phi_k^T - \nabla \Phi_k \frac{R_{q, n}}{\delta_{q+1,n}} \nabla \Phi_k^T \in B_{1/2}(\I),$$
and, thus, $a_{\xi, k, n}$ is indeed well-defined.

Let us also denote 
\begin{equation*}
    \mathbb Z_{q,n} := \big\{ k \in \mathbb Z \mid k \tau_q \in \mathcal{N}_{\tau_q}(\supp_t R_{q,n})\big\},
\end{equation*}
where $\mathcal N_{\tau}(A)$ stands for the neighbourhood of size $\tau$ of the set $A$. Then, for any $t \in \supp_t R_{q,n}$, it holds that 
\begin{equation*}
    \sum_{k \in \mathbb Z_{q,n}} \chi_k^2(t) = 1.
\end{equation*}
It follows, then, in view of lemma \ref{geom}, that 
\begin{equation} \label{decomp}
    \div \left[ \sum_{k \in \mathbb Z_{q,n}} \sum_{\xi \in \Lambda} a^2_{\xi, k, n} (\nabla \Phi_k)^{-1} \xi \otimes \xi (\nabla \Phi_k)^{-T} \right]  = \div \left(\sum_{k \in \mathbb Z_{q,n}} \chi_k^2\delta_{q+1, n} \I - R_{q,n}\right) = - \div R_{q,n}. 
\end{equation}

We define now the parameter which quantifies the frequency of temporal oscillations by 
\begin{equation*}
    \mu_{q+1} = \delta_{q+1}^{1/2} \lambda_q^{2/3} \lambda_{q+1}^{1/3} \lambda_{q+1}^{4\alpha},
\end{equation*}
and note that indeed $\mu_{q+1} > \tau_q^{-1}$: 
\begin{equation}\label{form.mutau}
    \mu_{q+1} \tau_q = \bigg( \frac{\lambda_{q+1}}{\lambda_q} \bigg)^{1/3 - \beta} \lambda_{q+1}^{3\alpha} > 1.
\end{equation}

To simplify notation, we denote
\begin{equation*}
    A_{\xi, k, n} = a_{\xi, k, n}^2  (\nabla \Phi_k)^{-1} \xi \otimes \xi (\nabla \Phi_k)^{-T}.
\end{equation*}
We let $w_{k, n + 1}$ to be the solution to the Newtonian linearization of the Euler equations with temporally oscillatory forcing. The well-posedness theory for smooth solutions of these equations is probably classical. Nevertheless, we give a self-contained account in appendix \ref{wl-psd}. With $\mathbb P$ standing for the Leray projection operator onto divergence-free vector fields, we define $w_{k, n+1}$ as the unique, mean-zero, divergence-free solution to
\begin{equation} \label{LocalNewt}
\begin{cases}
    \partial_t w_{k, n + 1} + \bar u_q \cdot \nabla w_{k, n + 1} + w_{k, n + 1} \cdot \nabla \bar u_q + \nabla p_{k, n + 1} = \sum_{\xi\in\Lambda} f_{\xi, k, n+1}(\mu_{q + 1} t) \mathbb P \div A_{\xi, k, n}(x,t), \\
    \div w_{k, n + 1} = 0, \\
    w_{k, n + 1} \big |_{t = t_k}(x) = \frac{1}{\mu_{q+1}}\sum_{\xi\in\Lambda} f^{[1]}_{\xi, k, n+1} (\mu_{q+1} t_k) \mathbb P \div A_{\xi, k, n}(x, t_k), \\
    \end{cases}
\end{equation}
where the function $f_{\xi, k, n+1}:\mathbb R \rightarrow \mathbb R$ is defined by 
\begin{equation*}
    f_{\xi, k, n+1} := 1 - g^2_{\xi, k, n+1},
\end{equation*}
and
\begin{equation}\label{def.gxikn1}
    g_{\xi, k, n+1} = 
    \begin{cases}
        g_{\xi, e, n+1} & \text{if } k \text{ is even}, \\ 
        g_{\xi, o, n+1}  & \text{if } k \text{ is odd}.
    \end{cases}
\end{equation}
The function $f^{[1]}_{\xi, k, n+1}$ denotes the primitive of $f_{\xi, k, n+1}$: 
\begin{equation*}
    f^{[1]}_{\xi, k, n+1} (t) = \int_0^t  f_{\xi, k, n+1} (s) ds.
\end{equation*}
Note that, in view of the fact that $g_{\xi, e, n+1}$ and $g_{\xi, o, n+1}$ have unit $L^2$ norm, $f^{[1]}_{\xi, k, n+1}$ is indeed a well-defined $1$-periodic function. Moreover, we emphasize that the functions $g_{\xi,k,n+1}$, $f_{\xi, k, n+1}$ and $f^{[1]}_{\xi, k, n+1}$ are independent of $q$, and that the total number of such functions is finite and depends only on $\Ga$ and the cardinality of $\Lambda$.

We can now define the $(n+1)^{\text{th}}$ Newton perturbation by the superposition of temporal localizations of the velocity fields $w_{k, n+1}$:
\begin{equation*}
    w^{(t)}_{q + 1 , n + 1} (x, t) = \sum_{k \in \mathbb Z_{q,n}} \tilde \chi_k (t) w_{k, n + 1}(x, t)\,.
\end{equation*}
It will also be useful to define:
\begin{equation*}
    p^{(t)}_{q + 1 , n + 1} (x, t) = \sum_{k \in \mathbb Z_{q,n}} \tilde \chi_k (t) p_{k, n + 1}(x, t)\,.
\end{equation*}

\subsection{The errors after the \texorpdfstring{$(n + 1)^{\text{th}}$}{nonono} step and the inductive proposition}
We now plug in the new velocity field $u_{q,n+1} = u_{q, n} + w_{q+1, n+1}^{(t)}$ into equation \eqref{steps} and compute the new error terms $R_{q,n+1}$, $S_{q,n+1}$, and $P_{q+1,n+1}$. With this aim, we note that $w_{q+1, n+1}^{(t)}$ satisfies
\begin{align*}
    \partial_t w^{(t)}_{q+1,n+1} + \bar u_q \cdot \nabla w^{(t)}_{q+1,n+1} + w^{(t)}_{q+1,n+1} \cdot \nabla \bar u_q + \nabla p^{(t)}_{q + 1 , n + 1} &= \sum_{k \in \mathbb Z_{q,n}} \sum_{\xi\in\Lambda} \tilde \chi_k (t) f_{\xi, k,n+1}(\mu_{q + 1} t) \mathbb P \div A_{\xi, k, n}\\
    &\qquad + \sum_{k \in \mathbb Z_{q,n}} \partial_t \tilde \chi_k w_{k, n+1}.
\end{align*}
Since $\supp A_{\xi, k, n} \subset \supp a_{\xi, k, n} \subset \supp \chi_k \times \mathbb{T}^2$, it holds that $\tilde \chi_k A_{\xi, k, n} = A_{\xi, k, n}$ for all $k \in \mathbb Z$. Therefore, using equation \eqref{decomp} and the definition of $f_{\xi, k, n+1}$, we obtain
\begin{eqnarray*}
    \sum_{k \in \mathbb Z_{q,n}} \sum_{\xi\in\Lambda} \tilde \chi_k (t) f_{\xi, k, n+1}(\mu_{q + 1} t) \mathbb P \div A_{\xi, k, n}
    & = & \sum_{k \in \mathbb Z_{q,n}} \sum_{\xi\in\Lambda} \mathbb P \div A_{\xi, k, n} - \sum_{k \in \mathbb Z_{q,n}} \sum_{\xi \in \Lambda}  g_{\xi, k, n+1}^2  \mathbb P \div A_{\xi, k, n} \\ 
    & = & - \mathbb P \div R_{q,n} - \sum_{k \in \mathbb Z_{q,n}} \sum_{\xi \in \Lambda}  g_{\xi, k, n+1}^2  \mathbb P \div A_{\xi, k, n}.
\end{eqnarray*}
It follows, then, that the system \eqref{steps} after the $(n+1)^{\text{th}}$ step  is satisfied with 
\begin{equation} \label{Newvelo}
    u_{q, n + 1} = u_{q, n} + w_{q+1, n + 1}^{(t)} = u_q + \sum_{m =1}^{n+1} w_{q+1, m}^{(t)},
\end{equation}
\begin{eqnarray} \label{Newpres}
    p_{q, n + 1} &= & p_{q, n} + p_{q+1, n + 1}^{(t)} - \Delta^{-1} \div \bigg[ \div R_{q,n} +  \sum_{k \in \mathbb Z_{q,n}}\sum_{\xi \in \Lambda}  g_{\xi, k, n+1}^2 \div A_{\xi, k, n} \bigg] \nonumber \\ 
    && - \langle w_{q+1, n+1}^{(t)}, \sum_{m=1}^n w_{q+1, m}^{(t)} \rangle - \frac{1}{2} |w_{q+1, n+1}^{(t)}|^2 - \langle w_{q+1, n+1}^{(t)}, u_q - \bar u_q \rangle,
\end{eqnarray}
\begin{equation} \label{NewStr}
    R_{q, n + 1} = \mathcal{R} \sum_{k \in \mathbb Z_{q,n}} \partial_t \tilde \chi_k w_{k, n+1},
\end{equation}
\begin{equation} \label{NewNash}
     S_{q, n + 1} = S_{q, n} - \sum_{k \in \mathbb Z_{q,n}} \sum_{\xi \in \Lambda}   g_{\xi, k, n+1}^2  A_{\xi, k, n},
\end{equation}
\begin{eqnarray} \label{SmalStr}
    P_{q + 1, n + 1} &= & P_{q + 1, n} + w_{q+1,n+1}^{(t)} \mathring \otimes w_{q+1, n+1}^{(t)} + \sum_{m = 1}^{n} \big (w_{q+1, n+1}^{(t)} \mathring \otimes  w_{q+1, m}^{(t)} +  w_{q+1, m}^{(t)} \mathring \otimes w_{q+1, n+1}^{(t)} \big ) \notag \\ 
    && + (u_q - \bar u_q) \mathring \otimes w_{q+1, n+1}^{(t)} + w_{q+1, n+1}^{(t)} \mathring \otimes (u_q - \bar u_q).
\end{eqnarray}
In expression \eqref{Newpres}, $\langle \,, \rangle$ denotes the standard inner product on $\mathbb R^2$, while in \eqref{SmalStr}, $\mathring \otimes$ denotes the trace-less part of the tensor product. The operator $\mathcal{R}$ above is the inverse-divergence operator described in appendix \ref{conv_int_tool}. We note that the new stress $R_{q,n+1}$ as defined in \eqref{NewStr} is well-defined as each $w_{k,n+1}$ has zero mean. We are now ready to state the main inductive proposition concerning the Newton perturbations.  

\begin{prop} \label{NewIter}
Assume $R_{q, n}$ satisfies 
\begin{equation} \label{NewIter_1}
    \|R_{q, n}\|_{N} \leq \delta_{q+1, n} \lambda_q^{N - \alpha}, \,\,\, \forall N \in \{0, 1,..., L-1\},
\end{equation} 
\begin{equation} \label{NewIter_2}
    \|\bar D_t R_{q, n}\|_N \leq \delta_{q+1,n} \tau_q^{-1} \lambda_q^{N - \alpha}, \,\,\, \forall N \in \{0, 1,..., L-1\},
\end{equation}
\begin{equation} \label{NewIter_3}
    \|R_{q, n}\|_{N+L-1} \lesssim \delta_{q+1, n} \lambda_q^{L-1 -\alpha} \ell_q^{-N}, \,\,\, \forall N \geq 0
\end{equation}
\begin{equation} \label{NewIter_4}
    \|\bar D_t R_{q,n}\|_{N + L -1} \lesssim \delta_{q+1, n} \tau_q^{-1} \lambda_q^{L-1 - \alpha} \ell_{q}^{-N}, \,\,\, \forall N \geq 0,
\end{equation}
where the implicit constants depend on $n$, $\Gamma$, $M$, $\alpha$ and $N$. Suppose, moreover, that 
\begin{eqnarray} \label{NewIter_5}
    \supp_t R_{q,n} &\subset& [-2 + (\de_q^{1/2}\la_q)^{-1} - 2n\tau_q, -1 -(\de_q^{1/2}\la_q)^{-1} + 2n\tau_q] \\ 
    && \cup [1 + (\de_q^{1/2}\la_q)^{-1} - 2n \tau_q, 2 - (\de_q^{1/2}\la_q)^{-1} + 2n \tau_q], \nonumber
\end{eqnarray}
Then, $R_{q, n+1}$ also satisfies \eqref{NewIter_1}-\eqref{NewIter_5} with $n$ replaced by $n+1$. 
\end{prop}

The claim concerning the temporal support is immediate from the definitions of $\tilde \chi_k$ and $\mathbb Z_{q,n}$, which imply 
\begin{equation*}
    \supp_t R_{q, n+1} \subset \overline{\mathcal{N}_{2\tau_q} (\supp_t R_{q, n})}.
\end{equation*}

We remark also that since $\tau_q < (\delta_q^{1/2} \lambda_q)^{-1}$, lemma \ref{smoli_estim} shows that the assumptions of proposition \ref{NewIter} are satisfied at $n = 0$. Indeed, we have that for all $N \leq {L-1},$
\begin{equation*}
    \|R_{q, 0}\|_N \leq C_L \delta_{q+1} \lambda_q^{N-2\alpha}, 
\end{equation*}
\begin{equation*}
    \|\bar D_t R_{q, 0}\|_N \leq C_L \delta_{q+1} \tau_q^{-1}\lambda_q^{N-2\alpha},
\end{equation*}
where $C_L>0$ is a constant depending only on $L$. Then, for any $\alpha > 0$, we choose $a_0$ sufficiently large so that 
\begin{equation*}
    C_L \lambda_q^{- \alpha} \leq 1.
\end{equation*}

In the rest of this section, we prove proposition \ref{NewIter} and obtain the estimates for the perturbation $w_{q+1, n+1}^{(t)}$ which will be used in the next section. 

\subsection{Proof of the inductive proposition}
Let $\psi_{k, n+1}$ denote the mean-zero stream-function of the vector field $w_{k, n+1}$. Then, 
\begin{equation}\label{form.rqn1}
    R_{q, n+1} = \mathcal{R} \nabla^{\perp} \sum_{k \in \mathbb Z_{q,n}} \partial_t \tilde \chi_k \psi_{k, n+1}. 
\end{equation}
Since $\mathcal{R} \nabla^{\perp}$ is of Calder\'on-Zygmund type \footnote{The operator $\mathcal{R}\nabla^\perp$ is a sum of operators of the form $\Delta^{-1}\partial_i \partial_j$, which can be written as a linear combination between the identity and a Calder\'on-Zygmund operator. Throughout the paper, we will refer to such operators as being of Calder\'on-Zygmund type. Note that the results of appendix \ref{SIO} apply.}, proposition \ref{NewIter} will follow once we obtain estimates for the stream functions. Moreover, such estimates will, of course, also offer control over the perturbation $w^{(t)}_{q+1, n+1}$. With this aim, we note that $\psi_{k, n+1}$ satisfies 
\begin{equation} \label{psi_eqn}
    \begin{cases}
        \partial_t \psi_{k, n + 1} + \bar u_q \cdot \nabla \psi_{k, n + 1} - 2 \Delta^{-1} \nabla^\perp \cdot \div (\psi_{k, n+1} \nabla^\perp \bar u_q) = \sum_{\xi \in \Lambda} f_{\xi, k, n}(\mu_{q+1} t) \Delta^{-1} \nabla^\perp \cdot \div A_{\xi, k, n} \\ 
        \psi_{k, n+1} \big|_{t = t_k} = \frac{1}{\mu_{q+1}}\sum_{\xi \in \Lambda} f_{\xi, k, n}^{[1]}(\mu_{q+1} t_k) \Delta^{-1} \nabla^\perp \cdot \div A_{\xi, k, n}\big|_{t = t_k},
    \end{cases}
\end{equation}
where 
\begin{equation*}
    \nabla^\perp \bar u_q = \begin{pmatrix}
        - \partial_2 \bar u_q^1 & \partial_1 \bar u_q^1 \\ 
        -\partial_2 \bar u_q^2 & \partial_1 \bar u_q^2
    \end{pmatrix}.
\end{equation*}
Equation \eqref{psi_eqn} can be obtained by applying the operator $\Delta^{-1} \nabla^\perp \cdot$ to \eqref{LocalNewt}.

We begin the analysis by obtaining estimates for $a_{\xi, k, n}$.

\begin{lem} \label{a_estim}
Under the assumptions of proposition \ref{NewIter}, the following hold: 
\begin{equation} \label{a_estim_1}
    \|a_{\xi, k, n}\|_N \lesssim \delta_{q+1, n}^{1/2} \lambda_q^N, \,\,\, \forall N \in \{0,1,..., L-1\}
\end{equation}
\begin{equation} \label{a_estim_2}
     \|\bar D_t a_{\xi, k, n}\|_N  \lesssim \delta_{q+1, n}^{1/2} \tau_q^{-1} \lambda_q^N,\,\,\, \forall N \in \{0,1,..., L-1\}
\end{equation}
\begin{equation} \label{a_estim_3}
    \|a_{\xi, k, n}\|_{N+L-1} \lesssim \delta_{q+1, n}^{1/2} \lambda_q^{L-1} \ell_q^{-N}, \,\,\, \forall N \geq 0
\end{equation}
\begin{equation} \label{a_estim_4}
    \|\bar D_t a_{\xi, k, n}\|_{N+L-1} \lesssim \delta_{q+1, n}^{1/2} \lambda_q^{L- 1}  \tau_q^{-1} \ell_q^{-N}, \,\,\, \forall N \geq 0,
\end{equation}
with implicit constants depending on $n$, $\Gamma$, $M$, $\alpha$, and $N$.
\end{lem}

\begin{proof}
    Note that since $\supp a_{\xi, k, n} \subset \supp \chi_k \times \mathbb T^2$, the estimates of lemma \ref{Flow_estim} apply. Thus, since on $\supp \chi_k$, 
    \begin{equation*}
        \left\|\nabla \Phi_k \nabla \Phi_k^T - \nabla \Phi_k \frac{R_{q,n}}{\delta_{q+1,n}} \nabla \Phi_k^T\right\|_0 \lesssim 1,
    \end{equation*}
    we use proposition \ref{comp_estim} to obtain that, on $\supp a_{\xi, k, n}$,
    \begin{eqnarray*}
        \|a_{\xi, k, n}\|_N & \lesssim & \delta_{q+1, n}^{1/2} \|\nabla \Phi_k \nabla \Phi_k^T - \nabla \Phi_k \frac{R_{q,n}}{\delta_{q+1,n}} \nabla \Phi_k^T \|_N \\ 
        & \lesssim & \delta_{q + 1, n}^{1/2} \big ( \|\nabla \Phi_k\|_N +  \| \nabla \Phi_k\|_N \|\frac{R_{q,n}}{\delta_{q+1,n}}\|_0 + \| \nabla \Phi_k\|_0 \|\frac{R_{q,n}}{\delta_{q+1,n}}\|_N \big), 
    \end{eqnarray*}
    for any $N > 0$. Then, \eqref{a_estim_1} and \eqref{a_estim_3} follow from lemma \ref{Flow_estim} and the assumptions of proposition \ref{NewIter}. 

    On the other hand, 
    \begin{eqnarray*}
        \|\bar D_t a_{\xi, k, n}\|_N &\lesssim & \delta_{q+1, n}^{1/2}|\partial_t \chi_k| \|\gamma_\xi \big(\nabla \Phi_k \nabla \Phi_k^T - \nabla \Phi_k \frac{R_{q,n}}{\delta_{q+1,n}} \nabla \Phi_k^T \big) \|_N \\
         &+&  \delta_{q+1,n}^{1/2} \|D\gamma_\xi \big( \nabla \Phi_k \nabla \Phi_k^T - \nabla \Phi_k \frac{R_{q,n}}{\delta_{q+1,n}} \nabla \Phi_k^T \big)\|_N \|\bar D_t \big(\nabla \Phi_k \nabla \Phi_k^T - \nabla \Phi_k \frac{R_{q,n}}{\delta_{q+1,n}} \nabla \Phi_k^T  \big) \|_0 \\
        & + & \delta_{q+1,n}^{1/2} \|D\gamma_\xi \big( \nabla \Phi_k \nabla \Phi_k^T - \nabla \Phi_k \frac{R_{q,n}}{\delta_{q+1,n}} \nabla \Phi_k^T \big)\|_0 \|\bar D_t \big(\nabla \Phi_k \nabla \Phi_k^T - \nabla \Phi_k \frac{R_{q,n}}{\delta_{q+1,n}} \nabla \Phi_k^T  \big) \|_N.
    \end{eqnarray*}
    The terms which do not involve the material derivatives are handled as before, by appealing to proposition \ref{comp_estim}. For the remaining ones, we note that
    \begin{eqnarray*}
        \|\bar D_t \big( \nabla \Phi_k \frac{R_{q,n}}{\delta_{q+1,n}} \nabla \Phi_k^T  \big) \|_N \lesssim && \|\bar D_t \frac{R_{q,n}}{\delta_{q+1,n}}\|_N  + \|\bar D_t \frac{R_{q,n}}{\delta_{q+1,n}}\|_0 \|\nabla \Phi_k\|_N + \| \frac{R_{q,n}}{\delta_{q+1,n}}\|_N \|\bar D_t \nabla \Phi_k\|_0  \\ 
        && + \| \frac{R_{q,n}}{\delta_{q+1,n}}\|_0 \|\bar D_t \nabla \Phi_k\|_N + \| \frac{R_{q,n}}{\delta_{q+1,n}}\|_0 \|\bar D_t \nabla \Phi_k\|_0 \|\nabla \Phi_k\|_N, 
    \end{eqnarray*}
    and 
    \begin{equation*}
        \|\bar D_t \big( \nabla \Phi_k \nabla \Phi_k^T \big)\|_N \lesssim \|\bar D_t \nabla \Phi_k\|_N + \|\bar D_t \nabla \Phi_k\|_0 \|\nabla \Phi_k\|_N. 
    \end{equation*}
    Then, \eqref{a_estim_2} and \eqref{a_estim_4} follow by lemmas \ref{smoli_estim} and \ref{Flow_estim}, together with the assumptions of proposition \ref{NewIter}.
\end{proof}

From the lemma above, we obtain estimates for $A_{\xi, k, n}$. 

\begin{cor} \label{a_cor}
Under the assumptions of proposition \ref{NewIter}, the following hold: 
\begin{equation}
    \|A_{\xi, k, n}\|_N \lesssim \delta_{q+1, n} \lambda_q^N, \,\,\, \forall N \in \{0,1,..., L-1\},
\end{equation}
\begin{equation}
    \|\bar D_t A_{\xi, k, n} \|_N  \lesssim \delta_{q+1, n} \tau_q^{-1} \lambda_q^N, \,\,\, \forall N \in \{0,1,..., L-1\},
\end{equation}
\begin{equation}
    \|A_{\xi, k, n}\|_{N+L-1} \lesssim \delta_{q+1, n} \lambda_q^{L-1} \ell_q^{-N},  \,\,\, \forall N \geq 0,
\end{equation}
\begin{equation}
    \|\bar D_t A_{\xi, k, n} \|_{N+ L - 1} \lesssim \delta_{q+1, n} \lambda_q^{L-1} \tau_q^{-1} \ell_q^{-N}, \,\,\, \forall N \geq 0,
\end{equation}
with implicit constants depending on $n$, $\Gamma$, $M$, $\alpha$, and $N$.
\end{cor}

\begin{proof}
    First note that by the same arguments as those given in the proof of lemma \ref{a_estim}, $a_{\xi, k, n}^2$ satisfies the same estimates as $a_{\xi, k, n}$ but with $\delta_{q, n+1}^{1/2}$ replaced by $\delta_{q, n+1}$ throughout. Then, 
    \begin{equation*}
        \|A_{\xi, k, n}\|_{N} \lesssim \|a^2_{\xi, k, n}\|_N  \|(\nabla \Phi_k)^{-1}\|_0^2 + \|a^2_{\xi, k, n}\|_0  \|(\nabla \Phi_k)^{-1}\|_N \|(\nabla \Phi_k)^{-1}\|_0,
    \end{equation*}
    and 
    \begin{eqnarray*}
        \|\bar D_t A_{\xi, k, n}\|_{N} \lesssim && \|\bar D_t a^2_{\xi, k, n}\|_N \|(\nabla \Phi_k)^{-1}\|_0^2 + \|\bar D_t a^2_{\xi, k, n}\|_0 \|(\nabla \Phi_k)^{-1}\|_N  \|(\nabla \Phi_k)^{-1}\|_0  \\ 
        && + \|a^2_{\xi, k, n}\|_N \|\bar D_t (\nabla \Phi_k)^{-1}\|_0 \|(\nabla \Phi_k)^{-1}\|_0 + \|a^2_{\xi, k, n}\|_0 \|\bar D_t (\nabla \Phi_k)^{-1}\|_N \|(\nabla \Phi_k)^{-1}\|_0 \\ 
        && + \|a^2_{\xi, k, n}\|_0 \|\bar D_t (\nabla \Phi_k)^{-1}\|_0 \|(\nabla \Phi_k)^{-1}\|_N.
    \end{eqnarray*}
    The conclusion follows by lemma \ref{Flow_estim}.
\end{proof}

The following can be considered the main technical lemma concerning the Newton perturbations. 

\begin{lem} \label{psi_estim}
   Under the assumptions of proposition \ref{NewIter}, the following hold on $\supp \tilde \chi_k$: 
   \begin{equation} \label{psi_estim_1}
       \|\psi_{k, n+1}\|_{N+\alpha} \lesssim \frac{\delta_{q+1, n} \lambda_q^{N} \ell_q^{-\alpha}}{\mu_{q+1}}, \,\,\, \forall N \in \{0,1,..., L-1\},
   \end{equation}
   \begin{equation} \label{psi_estim_2}
        \|\bar D_t \psi_{k, n+1}\|_{N+\alpha} \lesssim  \delta_{q+1, n} \lambda_q^N \ell_q^{-\alpha}, \,\,\, \forall N \in \{0, 1,..., L-1\},
   \end{equation}
   \begin{equation} \label{psi_estim_3}
       \|\psi_{k, n+1}\|_{N + L - 1 + \alpha} \lesssim \frac{\delta_{q+1, n} \lambda_q^{L-1} \ell_q^{-N-\alpha}}{\mu_{q+1}},  \,\,\, \forall N \geq 0,
   \end{equation}
   \begin{equation} \label{psi_estim_4}
       \|\bar D_t \psi_{k, n+1}\|_{N+ L - 1 + \alpha} \lesssim \delta_{q+1, n}  \lambda_q^{L-1} \ell_q^{-N-\alpha}, \,\,\, \forall N \geq 0.
   \end{equation}
   Moreover, on $\supp \partial_t \tilde \chi_k$, 
   \begin{equation} \label{psi_estim_5}
       \|\bar D_t \psi_{k, n+1}\|_{N+\alpha} \lesssim \frac{\delta_{q+1, n} \lambda_q^N \ell_q^{-\alpha}}{\mu_{q+1}\tau_q}  , \,\,\, \forall N \in \{0, 1,..., L-1\},
   \end{equation}
   \begin{equation} \label{psi_estim_6}
       \|\bar D_t \psi_{k, n+1}\|_{N+L -1 + \alpha} \lesssim \frac{\delta_{q+1, n}  \lambda_q^{L-1} \ell_q^{-N-\alpha}}{\mu_{q+1} \tau_q}, \,\,\, \forall N \geq 0,
   \end{equation}
   with implicit constants depending on $n$, $\Gamma$, $M$, $\alpha$, and $N$.
\end{lem}

\begin{proof}
    We begin by noting that, due to~\eqref{psi_eqn}, the stream-function $\psi_{k, n+1}$ satisfies 
    \begin{equation} \label{psidecomp}
        \psi_{k, n+1}  = \tilde \psi + \tilde \Xi + \Xi,
    \end{equation}
    where $\tilde \psi$ solves 
    \begin{equation*}
        \begin{cases}
            \bar D_t \tilde \psi =  2 \Delta^{-1} \nabla^\perp \cdot \div (\psi_{k, n+1} \nabla^\perp \bar u_q) \\
            \tilde \psi \big|_{t = t_k} = 0,
        \end{cases}
    \end{equation*}
    $\tilde \Xi$ is the solution to 
    \begin{equation*}
    \begin{cases}
        \bar D_t \tilde \Xi = - \frac{1}{\mu_{q+1}}  \sum_{\xi \in \Lambda} f^{[1]}_{\xi, k, n}(\mu_{q+1} \cdot ) \bar D_t \Delta^{-1} \nabla^\perp \cdot \div A_{\xi, k, n} \\ 
        \tilde \Xi \big |_{t = t_k} = 0,
        \end{cases}
    \end{equation*}
    and 
    \begin{equation*}
        \Xi = \frac{1}{\mu_{q+1}} \sum_{\xi \in \Lambda} f_{\xi, k, n}^{[1]}(\mu_{q+1} \cdot) \Delta^{-1} \nabla^\perp \cdot \div A_{\xi, k, n}. 
    \end{equation*}
    In view of the uniqueness of solutions to transport equations, this decomposition is verified once we apply the material derivative to \eqref{psidecomp}. 

    \textit{Estimates for $\tilde \psi$ when $N=0$.} Since $\Delta^{-1} \nabla^\perp \div$ is of Calder\'on-Zygmund type, we have
    \begin{equation*}
        \|\bar D_t \tilde \psi\|_\alpha \lesssim \|\psi_{k,n+1} \nabla^\perp \bar u_q\|_\alpha \lesssim \|\psi_{k,n+1}\|_\alpha \|\bar u_q\|_{1 + \alpha},
    \end{equation*}
    from which it follows, by proposition~\ref{transport_estim},
    that on $\supp \tilde \chi_k$,
    \begin{equation*}
        \|\tilde \psi(\cdot, t)\|_\alpha \lesssim \tau_q^{-1} \int_{t_k}^t \|\psi_{k, n+1}(\cdot, s) \|_\alpha ds.
    \end{equation*}

    \textit{Estimates for $\tilde \Xi$ when $N = 0$.} By similar arguments, we have 
    \begin{eqnarray*}
        \|\bar D_t \tilde \Xi\|_{\alpha} &\lesssim & \frac{1}{\mu_{q+1}}\sup_\xi \|\bar D_t A_{\xi, k, n}\|_{\alpha} + \frac{1}{\mu_{q+1}}\sup_\xi\|[\bar u_q \cdot \nabla, \Delta^{-1}\nabla^\perp \div] A_{\xi, k, n}\|_{\alpha} \\ 
        & \lesssim & \frac{1}{\mu_{q+1}}\sup_\xi\|\bar D_t A_{\xi, k, n}\|_{\alpha} + \frac{1}{\mu_{q+1}} \|\bar u_q\|_{1+\alpha} \sup_\xi\|A_{\xi, k, n}\|_\alpha \\ 
        &\lesssim & \frac{\delta_{q+1, n}  \lambda_q^\alpha}{\mu_{q+1} \tau_q}
    \end{eqnarray*}
    where for the second inequality we used the commutator estimate of proposition \ref{CZ_comm}, while the last one follows by interpolation from the conclusions of corollary \ref{a_cor}. By proposition \ref{transport_estim}, we conclude that, on $\supp \tilde \chi_k$,
    \begin{equation*}
        \|\tilde \Xi\|_\alpha \lesssim \frac{1}{\mu_{q+1}} \delta_{q+1, n} \lambda_q^{\alpha}.
    \end{equation*}
    
    \textit{Estimates for $\Xi$ when $N=0$.} Finally, for $\Xi$, we note that
    \begin{equation*}
        \|\Xi\|_{\alpha} \lesssim \frac{1}{\mu_{q+1}} \sup_\xi\|A_{\xi, k, n}\|_{\alpha} \lesssim \frac{\delta_{q+1, n} \lambda_q^{\alpha}}{\mu_{q+1}}.
    \end{equation*}

    Going back in \eqref{psidecomp}, we obtain 
    \begin{equation*}
        \|\psi_{k, n+1}(\cdot, t)\|_\alpha \lesssim \frac{\delta_{q+1, n} \lambda_q^{\alpha}}{\mu_{q+1}} + \tau_q^{-1} \int_{t_k}^t \|\psi_{k, n+1}(\cdot, s)\|_\alpha ds,  
    \end{equation*}
    from which Gr\"onwall's inequality implies that, on $\supp \tilde \chi_k$, 
    \begin{equation*}
        \|\psi_{k, n+1}\|_\alpha \lesssim \frac{\delta_{q+1, n} \lambda_q^{\alpha}}{\mu_{q+1}}.
    \end{equation*}

    \textit{Estimates for $\tilde \psi$ when $N \geq 1$.} Let $\theta$ be a multi-index with $|\theta| = N$. Then, 
    \begin{equation*}
        \|\bar D_t \partial^\theta \tilde \psi\|_\alpha \lesssim \|\partial^\theta \bar D_t  \tilde \psi\|_\alpha + \|[\bar u_q \cdot \nabla, \partial^\theta] \tilde \psi\|_\alpha 
    \end{equation*}
    On the one hand, 
    \begin{equation*}
        \|\partial^\theta \bar D_t  \tilde \psi\|_\alpha \lesssim \|\partial^\theta(\psi_{k, n+1} \nabla^\perp \bar u_q)\|_\alpha \lesssim\|\bar u_q\|_{1+\alpha} \|\psi_{k, n+1}\|_{N+\alpha} + \|\bar u_q\|_{N+ 1+\alpha}\|\psi_{k, n+1}\|_{\alpha},
    \end{equation*}
    while on the other, 
    \begin{eqnarray*}
        \|[\bar u_q \cdot \nabla, \partial^\theta] \tilde \psi\|_\alpha &\lesssim& \|\bar u_q\|_{N+\alpha}\|\tilde \psi\|_{1+\alpha} + \|\bar u_q\|_{1+\alpha}\|\tilde \psi\|_{N+\alpha} \\ 
        &\lesssim & \|\bar u_q\|_{1+\alpha}\|\tilde \psi\|_{N+\alpha} + \|\bar u_q\|_{N+ 1+\alpha}\|\tilde \psi\|_{\alpha},
    \end{eqnarray*}
    where the last inequality is obtained by interpolation and Young's inequality for products. It follows, then, by proposition \ref{transport_estim}, that 
    \begin{align*}
        \|\tilde \psi (\cdot, t)\|_{N+ \alpha} &\lesssim  \frac{\delta_{q+1, n} \lambda_q^{\alpha}\tau_q }{\mu_{q+1}} \|\bar u_q\|_{N+1+\alpha} + \|\bar u_q\|_{1+\alpha} \int_{t_k}^t \|\psi_{k, n+1}(\cdot, s) \|_{N+\alpha}ds \\ 
        &\qquad + \|\bar u_q\|_{1+\alpha} \int_{t_k}^t \|\tilde \psi(\cdot, s) \|_{N+\alpha}ds.
    \end{align*}
    Then, by Gr\"onwall's inequality, we conclude: 
    \begin{equation*}
        \|\tilde \psi (\cdot, t)\|_{N+ \alpha} \lesssim \frac{\delta_{q+1, n} \lambda_q^{\alpha} \tau_q}{\mu_{q+1}} \|\bar u_q\|_{N+1+\alpha}  + \tau_q^{-1} \int_{t_k}^t \|\psi_{k, n+1}(\cdot, s) \|_{N+\alpha}ds.
    \end{equation*}

    \textit{Estimates for $\tilde \Xi$ when $N \geq 1$.} Let $\theta$ be a multi-index as above. Then, 
    \begin{eqnarray*}
        \|\bar D_t \partial^\theta \tilde \Xi\|_\alpha &\lesssim& \|\bar D_t \tilde \Xi\|_{N+\alpha} + \|[\bar u_q \cdot \nabla, \partial^\theta] \tilde \Xi\|_\alpha  \\ 
        &\lesssim & \frac{1}{\mu_{q+1}}\sup_\xi \big( \|\bar D_t A_{\xi, k, n}\|_{N+\alpha} + \|[\bar u_q \cdot \nabla, \Delta^{-1} \nabla^\perp \div] A_{\xi, k, n}\|_{N+\alpha} \big) + \|[\bar u_q \cdot \nabla, \partial^\theta] \tilde \Xi\|_\alpha \\ 
         &\lesssim & \frac{1}{\mu_{q+1}}\sup_\xi \big( \|\bar D_t A_{\xi, k, n}\|_{N+\alpha} + \|\bar u_q\|_{1+\alpha} \|A_{\xi, k, n}\|_{N+\alpha} + \|\bar u_q\|_{N+1 + \alpha} \|A_{\xi, k, n}\|_\alpha \big) \\ 
         && + \|\bar u_q\|_{N+1+\alpha} \|\tilde \Xi\|_\alpha + \|\bar u_q\|_{1+\alpha} \|\tilde \Xi\|_{N+ \alpha}, 
    \end{eqnarray*}
    where we have used proposition \ref{CZ_comm} and the same argument as before for the term involving the commutator $[\bar u_q \cdot \nabla, \partial^\theta]$. Then, arguing by proposition \ref{transport_estim} and Gr\"onwall's inequality as we did for $\tilde \psi$, we obtain 
    \begin{equation*}
        \|\tilde \Xi\|_{N+\alpha} \lesssim \frac{1}{\mu_{q+1}} \sup_\xi \big( \tau_q \|\bar D_t A_{\xi, k, n} \|_{N+\alpha} + \|A_{\xi, k, n}\|_{N+\alpha}\big) + \frac{\delta_{q+1, n} \lambda_q^\alpha \tau_q}{\mu_{q+1}} \|\bar u_q\|_{N+1+\alpha}.
    \end{equation*}

    \textit{Estimates for $\Xi$ when $N \geq 1$.} Finally, we note:
    \begin{equation*}
        \|\Xi\|_{N+\alpha} \lesssim \frac{1}{\mu_{q+1}} \sup_\xi \|A_{\xi, k, n}\|_{N+\alpha}. 
    \end{equation*}

    Then, we have 
    \begin{align*}
        \|\psi_{k, n+1}(\cdot, t)\|_{N+\alpha} & \lesssim   \frac{1}{\mu_{q+1}}\sup_\xi \big( \tau_q \|\bar D_t A_{\xi, k, n}\|_{N+\alpha} + \|A_{\xi, k, n}\|_{N+\alpha} + \delta_{q+1, n} \lambda_q^\alpha \tau_q \|\bar u_q\|_{N+1+\alpha} \big) \\ 
        &\qquad + \tau_q^{-1} \int_{t_k}^t \|\psi_{k, n+1}(\cdot, s)\|_{N+ \alpha} ds,
    \end{align*}
    from which we obtain, by Gr\"onwall's inequality, 
    \begin{equation*}
        \|\psi_{k, n+1}\|_{N+\alpha} \lesssim \frac{1}{\mu_{q+1}}\sup_\xi\big( \tau_q \|\bar D_t A_{\xi, k, n}\|_{N+\alpha} + \|A_{\xi, k, n}\|_{N+\alpha} + \delta_{q+1, n} \lambda_q^\alpha \tau_q \|\bar u_q\|_{N+1+\alpha} \big).
    \end{equation*}
    This, in view of lemma \ref{smoli_1} and corollary \ref{a_cor}, implies \eqref{psi_estim_1} and \eqref{psi_estim_3}. 

    To obtain the claimed estimates for the material derivative of $\psi_{k, n+1}$, we note that, from \eqref{psi_eqn}, 
    \begin{eqnarray*}
        \|\bar D_t \psi_{k, n+1}(\cdot, t)\|_{N+\alpha}  &\lesssim & \|\psi_{k, n+1} \nabla^\perp \bar u_q\|_{N+\alpha} + \sup_\xi\|A_{\xi, k, n}(\cdot, t)\|_{N+\alpha} \\ 
         &\lesssim & \|\psi_{k, n+1}\|_{N+\alpha} \| \bar u_q\|_{1+\alpha} + \|\psi_{k, n+1}\|_{\alpha} \|\bar u_q\|_{N+1+\alpha} + \sup_\xi \|A_{\xi, k, n}(\cdot, t)\|_{N+\alpha}.
    \end{eqnarray*}
    Then, \eqref{psi_estim_2} and \eqref{psi_estim_4} follow from corollary \ref{a_cor}, together with \eqref{psi_estim_1} and \eqref{psi_estim_3}, while \eqref{psi_estim_5} and \eqref{psi_estim_6} follow likewise once we note that $A_{\xi, k, n} = 0$ on $\supp \partial_t \tilde \chi_k$. 
\end{proof}

We are now ready to prove the main result of this section.

\begin{proof} [Proof of proposition \ref{NewIter}]
Recalling the form of $R_{q,n+1}$ from~\eqref{form.rqn1}, that $\mathcal R \nabla^\perp$ is of Calder\'on-Zygmund type, and that the set $\{\tilde \chi_k\}$ is locally finite, we obtain 
\begin{equation*}
    \|R_{q, n+1}\|_N \lesssim \|R_{q, n+1}\|_{N+\alpha} \lesssim \tau_q^{-1} \sup_{k \in \mathbb Z_{q,n}} \|\psi_{k, n+1}\|_{N+\alpha}.
\end{equation*}
Then by \eqref{psi_estim_1} of lemma \ref{psi_estim}, for $N \in \{0,1,..., L-1\}$, there exists a constant which is independent of $a > a_0$ and $q$ such that 
\begin{align*}
    \|R_{q, n+1}\|_N &\leq C \frac{\tau_q^{-1}}{\mu_{q+1}}\delta_{q+1, n}\lambda_q^N\ell_q^{-\alpha}  \\
    &\leq C \delta_{q+1, n}\bigg(\frac{\lambda_q}{\lambda_{q+1}}\bigg)^{1/3 - \beta} (\lambda_{q+1} \ell_q)^{-\alpha} \lambda_{q+1}^{-2\alpha} \lambda_q^N \leq (C \lambda_{q+1}^{-\alpha })\delta_{q+1, n+1} \lambda_q^{-\alpha} \lambda_q^N,
\end{align*}
where we have used~\eqref{form.mutau} in the second inequality and~\eqref{def.deqn1} in the third. For any $\alpha > 0$, $a_0$ can be chosen sufficiently large so that 
\begin{equation*}
    C \lambda_{q+1}^{-\alpha} \leq 1,
\end{equation*}
and so \eqref{NewIter_1} follows. Likewise, \eqref{NewIter_3} follows from lemma \ref{psi_estim}.

Moreover, 
\begin{eqnarray*}
    \|\bar D_t R_{q, n+1}\|_{N+\alpha} &\lesssim& \sup_{k \in \mathbb Z_{q,n}} \big( \|\bar D_t(\partial_t \tilde \chi_k \psi_{k, n+1})\|_{N+\alpha} + \|[\bar u_q \cdot \nabla, \mathcal{R} \nabla^\perp] \partial_t \tilde \chi_k \psi_{k, n+1}\|_{N+\alpha} \big) \\ 
    &\lesssim & \sup_{k \in \mathbb Z_{q,n}} \big( \tau_q^{-2} \|\psi_{k, n+1}\|_{N+\alpha} + \tau_q^{-1} \|\bar D_t \psi_{k, n+1}\|_{N+\alpha, \,\, \supp \partial_t \tilde \chi_k} \\ 
    && + \tau_q^{-1} \|\bar u_q\|_{1+\alpha}\|\psi_{k, n+1}\|_{N+\alpha} + \tau_q^{-1} \|\bar u_q\|_{N+1+\alpha} \|\psi_{k,n+1}\|_\alpha \big),
\end{eqnarray*}
where, once again, we have used proposition \ref{CZ_comm}. Then, for $N \in \{0,1,...,L-1\}$ and some constant $C$ which is independent of $a > a_0$ and $q$, lemmas \ref{smoli_estim} and \ref{psi_estim} imply
\begin{equation*}
    \|\bar D_t R_{q, n+1}\|_{N} \leq C \tau_q^{-1} \delta_{q+1, n} \bigg(\frac{\lambda_q}{\lambda_{q+1}}\bigg)^{1/3 - \beta} (\lambda_{q+1} \ell_q)^{-\alpha} \lambda_{q+1}^{-2\alpha} \lambda_q^N,
\end{equation*}
and the conclusion follows for $a_0$ sufficiently large depending on $L$, $\alpha$, $M$ and $\Gamma$ (thus, on $\beta$) by the same arguments as above. 
\end{proof}

\subsection{Estimates for the total Newton perturbation and the perturbed flow}

We now turn to obtaining estimates for the total Newton perturbation 
\begin{equation*}
    w_{q+1}^{(t)} = \sum_{n = 1}^\Gamma w_{q+1, n}^{(t)}.
\end{equation*}
Since proposition \ref{NewIter} holds and, as already noted, its assumptions are indeed satisfied by $R_{q, 0}$, the conclusions of lemma \ref{psi_estim} also hold for all $n \in \{0, 1, ... ,\Gamma -1\}$. The following is, then, a direct consequence.

\begin{lem} \label{w_t_estim}
    The following estimates hold: 
    \begin{equation} \label{w_t_estim_1}
        \|w_{q+1}^{(t)}\|_N \lesssim \frac{\delta_{q+1} \lambda_q^{N + 1} \ell_q^{-\alpha}}{\mu_{q+1}}, \,\,\, \forall N \in \{0,1,...,L-2\},
    \end{equation}
    \begin{equation} \label{w_t_estim_2}
        \|\bar D_t w_{q+1}^{(t)}\|_{N} \lesssim \delta_{q+1} \lambda_q^{N+1} \ell_q^{-\alpha}, \,\,\, \forall N \in  \{0,1,...,L-2\},
    \end{equation}
    \begin{equation} \label{w_t_estim_3}
        \|w_{q+1}^{(t)}\|_{N+L -2 } \lesssim \frac{\delta_{q+1} \lambda_q^{L-1} \ell_q^{-N-\alpha}}{\mu_{q+1}}, \,\,\, \forall N \geq 0,
    \end{equation}
    \begin{equation} \label{w_t_estim_4}
        \|\bar D_t w_{q+1}^{(t)}\|_{N+L-2} \lesssim \delta_{q+1} \lambda_q^{L-1} \ell_q^{-N - \alpha}, \,\,\, \forall N \geq 0,
    \end{equation}
    with implicit constants depending on $\Gamma$, $M$, $\alpha$ and $N$. Moreover, it holds that 
    \begin{eqnarray}
        \supp_t w_{q+1}^{(t)} &\subset& [-2 + (\delta_q^{1/2} \lambda_q)^{-1} - 2\Gamma \tau_q, -1 - (\delta_q^{1/2} \lambda_q)^{-1} + 2\Gamma \tau_q] \\ 
        && \cup [1 + (\delta_q^{1/2} \lambda_q)^{-1} - 2 \Gamma \tau_q, 2 - (\delta_q^{1/2} \lambda_q)^{-1} + 2 \Gamma \tau_q]. \nonumber
    \end{eqnarray}
\end{lem}

\begin{proof}
    It suffices to argue for $w_{q+1, n+1}^{(t)}$ and note that $\delta_{q+1, n} \leq \delta_{q+1}$ for all $n$. For \eqref{w_t_estim_1} and \eqref{w_t_estim_3}, we simply have that, for all $N \geq 0$,  
    \begin{equation*}
        \|w_{q+1, n+1}^{(t)}\|_N \lesssim \sup_{k \in \mathbb Z_{q,n}} \|\psi_{k, n+1}\|_{N+1+\alpha}, 
    \end{equation*}
    and we apply lemma \ref{psi_estim}. 

    For the remaining estimates we write 
    \begin{equation*}
        \bar D_t w_{q+1, n+1}^{(t)} = \sum_{k \in \mathbb Z_{q,n}} \big( \partial_t \tilde \chi_k \nabla^\perp \psi_{k, n+1} + \tilde \chi_k \nabla^\perp \bar D_t \psi_{k, n+1} - \tilde \chi_k \nabla^\perp \bar u_q \nabla \psi_{k, n+1} \big), 
    \end{equation*}
    from which it follows that 
    \begin{eqnarray*}
        \|\bar D_t w_{q+1, n+1}^{(t)}\|_N \lesssim && \sup_{k \in \mathbb Z_{q,n}} \big( \tau_q^{-1} \|\psi_{k, n+1}\|_{N+1+\alpha} + \|\bar D_t \psi_{k, n+1}\|_{N + 1 + \alpha} \\ 
        && + \|\bar u_q\|_{N+1} \|\psi_{k, n}\|_{1+\alpha} + \|\bar u_q\|_1 \|\psi_{k, n+1}\|_{N+1 + \alpha} \big).
    \end{eqnarray*}
    The largest term on the right hand side is $\|\bar D_t \psi_{k,n+1}\|_{N+1+\alpha}$. Appealing to the estimates of lemmas \ref{smoli_estim} and \ref{psi_estim} concludes the proof of \eqref{w_t_estim_2} and \eqref{w_t_estim_4}. 

    The claimed property on the temporal support is clear by the definition of $w_{q+1}^{(t)}$ and proposition \ref{NewIter}.
\end{proof}

In the construction of the Nash perturbation, which will be detailed in the next section, we will make use of the backwards flow $\tilde \Phi_t$ starting at $t \in \mathbb R$ of 
\begin{equation}\label{def.buqg}
     \bar u_{q, \Gamma} = \bar u_{q} + w_{q+1}^{(t)},
\end{equation}
which is characterized by the equation 
\begin{equation} \label{Flow_t_gam}
    \begin{cases}
        \partial_s \tilde \Phi_t(x,s) + \bar u_{q, \Gamma}(x,s) \cdot \nabla \tilde \Phi_t (x,s) = 0, \\ 
        \tilde \Phi_t(x, t) = x,
    \end{cases}
\end{equation}
as well as of the Lagraingian flow 
\begin{equation} \label{Lagr_t_gam}
    \begin{cases}
        \frac{d}{ds} \tilde X_t(\alpha, s) = \bar u_{q, \Gamma}(X_t(\alpha, s),s), \\ 
        X_t(\alpha, t) = \alpha.
    \end{cases}
\end{equation}

We note that for $N \in \{1, 2,..., L-2\}$, 
\begin{equation*}
    \|w_{q+1}^{(t)}\|_N \lesssim \delta_{q}^{1/2} \lambda_q^N \frac{\delta_{q+1} \lambda_q \ell_q^{-\alpha}}{\mu_{q+1} \delta_q^{1/2}} \lesssim \delta_q^{1/2} \lambda_q^N \bigg( \frac{\delta_{q+1}}{\delta_q} \bigg)^{1/2} \bigg( \frac{\lambda_q}{\lambda_{q+1}}\bigg)^{1/3}. 
\end{equation*}
This shows that $\bar u_{q, \Gamma}$ is a small perturbation of $\bar u_q$. This observation, together with similar considerations for $N \geq L-1$, immediately implies the following corollary. 

\begin{cor} \label{Gamma_velo_estim}
    The following estimates hold: 
    \begin{equation}
        \|\bar u_{q, \Gamma} \|_N \lesssim \delta_{q}^{1/2} \lambda_q^N, \,\,\, \forall N \in \{1,2,...,L-2\},
    \end{equation}
    \begin{equation}
        \|\bar u_{q, \Gamma}\|_{N + L-2} \lesssim \delta_{q}^{1/2} \lambda_q^{L-2} \ell_q^{-N}, \,\,\, \forall N \geq 0,
    \end{equation}
    where the implicit constants depend on $\Gamma$, $M$, $\alpha$, and $N$.
\end{cor}

We will also use the notation $\bar D_{t, \Gamma}$ for the material derivative corresponding to $\bar u_{q, \Gamma}$: 
\begin{equation*}
    \bar D_{t, \Gamma} = \partial_t + \bar u_{q, \Gamma} \cdot \nabla.
\end{equation*}
Then, arguing precisely as in lemma \ref{Flow_estim}, we conclude that the following corollary holds true. 
\begin{cor} \label{Flow_gam_estim}
Let $t \in \mathbb R$ and $\tau \leq  \|\bar u_{q, \Gamma}\|_1^{-1}$. Let $\tilde \Phi_t$ be defined by \eqref{Flow_t_gam}, and let $ \tilde X_t$ denote the corresponding Lagrangian flow \eqref{Lagr_t_gam}. Then, for any $|s - t| < \tau$, 
    \begin{equation} \label{Flow_gam_estim_1}
        \|(\nabla \tilde \Phi_t)^{-1}(\cdot, s)\|_N + \|\nabla \tilde \Phi_t (\cdot, s)\|_N \lesssim \lambda_q^N, \,\,\, \forall N \in\{0,1,..., L-3\},
    \end{equation}
    \begin{equation} \label{Flow_gam_estim_2}
        \|\bar D_{t, \Gamma} (\nabla \tilde \Phi_t)^{-1}(\cdot, s)\|_N + \|\bar D_{t, \Gamma} \nabla \tilde \Phi_t (\cdot, s)\|_N \lesssim \delta_q^{1/2} \lambda_q^{N+1}, \,\,\, \forall N \in\{0,1,..., L-3\},
    \end{equation}
    \begin{equation} \label{Flow_gam_estim_3}
         \|D \tilde X_t(\cdot, s)\|_N  \lesssim \lambda_q^N, \,\,\, \forall N \in \{0,1,..., L-3\},
    \end{equation}
    \begin{equation} \label{Flow_gam_estim_4}
        \|(\nabla \tilde \Phi_t)^{-1}(\cdot, s)\|_{N+L-3} + \|\nabla \tilde \Phi_t(\cdot, s)\|_{N+L-3}  \lesssim \lambda_q^{L-3} \ell_q^{-N}, \, \, \, \forall N \geq 0,
    \end{equation}
    \begin{equation} \label{Flow_gam_estim_5}
        \|\bar D_{t, \Gamma} (\nabla \tilde \Phi_t)^{-1}(\cdot, s)\|_{N+L-3} + \|\bar D_{t, \Gamma} \nabla \tilde \Phi_t(\cdot, s)\|_{N+L-3}  \lesssim \delta_q^{1/2} \lambda_q^{L-2} \ell_q^{-N}, \, \, \, \forall N \geq 0,
    \end{equation}
    \begin{equation} \label{Flow_gam_estim_6}
        \|D \tilde X_t(\cdot, s)\|_{N+L-3} \lesssim \lambda_q^{L-3} \ell_q^{-N}, \, \, \, \forall N \geq 0
    \end{equation}
    where the implicit constants depend only on $\Gamma$, $M$, $\alpha$, and $N$.   
\end{cor}

\begin{rem} \label{remark_flow_bd}
Note that 
\begin{equation*}
    \tau_q \|\bar u_{q, \Gamma}\|_1 \lesssim \lambda_{q+1}^{-\alpha},
\end{equation*}
and, thus, corollary \ref{Flow_gam_estim} is satisfied with $\tau = \tau_q$. Moreover, with this choice of $\tau$, proposition \ref{transport_estim}, in fact, establishes that
\begin{equation*}
    \|\I - \nabla \tilde \Phi\|_0 \lesssim \lambda_{q+1}^{-\alpha},
\end{equation*}
In particular, since for any $\alpha>0$, and $C>0$ independent of $a > a_0$ and $q$, $a_0$ can be chosen sufficiently large such that 
\begin{equation*}
    C \lambda_{q+1}^{- \alpha} \leq 1,
\end{equation*}
we conclude that
$\|\nabla \tilde \Phi\|_0$ can be bounded independently of the parameters of the construction.
\end{rem}

We end the discussion concerning the perturbed flow with an elementary stability lemma which will play a role in the next section.

\begin{lem} \label{flow_stabil}
    Let $t \in \mathbb R$ and $\tau \leq (\|\bar u_{q, \Gamma}\|_1 + \|\bar u_q\|_1)^{-1}$. Let $\tilde \Phi_t$ be backwards flow of $\bar u_{q, \Gamma}$, defined by \eqref{Flow_t_gam}, and $\Phi_t$ the backwards flow of $\bar u_q$, defined in \eqref{Flow_t}. Then, for any $|s-t| < \tau$ and $N \in \{0,1,...,L-4\}$,
    \begin{equation}
        \|\nabla \Phi_t(\cdot, s) - \nabla \tilde \Phi_t(\cdot, s)\|_N + \|(\nabla \Phi_t(\cdot, s))^{-1} - (\nabla \tilde \Phi_t(\cdot, s))^{-1}\|_N \lesssim \tau \frac{\delta_{q+1} \lambda_q^2 \ell_q^{-\alpha}}{\mu_{q+1}}\lambda_q^{N}, 
    \end{equation}
    while for $N \geq 0$,
    \begin{equation}
        \|\nabla \Phi_t(\cdot, s) - \nabla \tilde \Phi_t(\cdot, s)\|_{N + L-4} + \|(\nabla \Phi_t(\cdot, s))^{-1} - (\nabla \tilde \Phi_t(\cdot, s))^{-1}\|_{N+ L-4} \lesssim \tau \frac{\delta_{q+1} \lambda_q^2 \ell_q^{-\alpha}}{\mu_{q+1}}\lambda_q^{L-4} \ell_q^{-N}, 
    \end{equation}
    with implicit constants depending on $\Gamma$, $M$, $\al$ and $N$. 
\end{lem}

\begin{proof}
    We note that $\Phi_t - \Tilde \Phi_t$ satisfies 
    \begin{equation*}
        \begin{cases}
            (\partial_t + \bar u_q \cdot \nabla)(\Phi_t - \Tilde \Phi_t)(x,s) = w_{q+1}^{(t)}\cdot \nabla \tilde \Phi_t (x,s) \\ 
            (\Phi_t - \Tilde \Phi_t)\big|_{s = t} = 0.
        \end{cases}
    \end{equation*}
    Then, by proposition \ref{transport_estim}, 
    \begin{eqnarray*}
        \|\Phi_t - \Tilde \Phi_t\|_N &\lesssim &\tau \|w_{q+1}^{(t)}\cdot \nabla \tilde \Phi_t\|_N + \tau^2 \|\bar u_q\|_N \|w_{q+1}^{(t)} \cdot \nabla \tilde \Phi_t\|_1 \\ 
        & \lesssim & \tau \|w_{q+1}^{(t)}\|_N + \tau \|w_{q+1}^{(t)}\|_0 \|\nabla \tilde \Phi_t\|_N + \tau^2 \|\bar u_q\|_N \big( \|w_{q+1}^{(t)}\|_1 + \|w_{q+1}^{(t)}\|_0 \|\nabla \tilde \Phi_t\|_1 \big).
    \end{eqnarray*}
    The claimed estimates, then, follow by lemmas \ref{smoli_estim} and \ref{w_t_estim}, and corollary \ref{Flow_gam_estim}. 

    Since $\Phi_t$ and $\tilde \Phi_t$ are measure preserving, it holds that 
    \begin{equation*}
        (\nabla \Phi_t)^{-1} =
        \begin{pmatrix}
        \partial_2 \Phi_t^2 & - \partial_2 \Phi_t^1 \\ 
        -\partial_1 \Phi_t^2 & \partial_1 \Phi_t^1
        \end{pmatrix},
    \end{equation*}
    and similarly for $(\nabla \tilde \Phi_t)^{-1}$. Therefore, the remaining estimates follow from those already obtained for $\nabla \Phi_t - \nabla \tilde \Phi_t$.
\end{proof}

\subsection{Summary}

Let us now describe the situation after $\Gamma$ steps of the Newton iteration. The perturbed system is 
\begin{equation}
    \partial_t u_{q, \Gamma} + \div(u_{q, \Gamma} \otimes u_{q, \Gamma}) + \nabla p_{q,\Gamma} = \div S_{q, \Gamma} + \div\big( R_{q, \Gamma} + P_{q + 1, \Gamma} \big),
\end{equation}
where: 
\begin{itemize}
    \item The perturbed velocity is defined by 
    \begin{equation}
        u_{q, \Gamma} = u_q + w_{q+1}^{(t)} = u_q + \sum_{n = 1}^\Gamma w_{q+1, n}^{(t)};
    \end{equation}
    
    \item By the inductive definition \eqref{Newpres} of $p_{q, \Gamma}$, we have
    \begin{equation} \label{p_q_Gam}
        p_{q, \Gamma} =  p_q + \sum_{n = 1}^\Gamma p_{q+1, n}^{(t)} - \sum_{n = 0}^{\Gamma-1} \Delta^{-1} \div \div \big( R_{q, n} + \sum_{\xi, k} g_{\xi, k, n+1}^2 A_{\xi, k, n }\big) - \frac{|w_{q+1}^{(t)}|^2}{2} + \langle \bar u_q - u_q, w_{q+1}^{(t)} \rangle;
    \end{equation}

    \item The error $S_{q, \Gamma}$ is ``well-prepared'' to be erased by the Nash perturbation: 
    \begin{eqnarray}
        S_{q, \Gamma} &=& - \sum_{n=0}^{\Gamma -1} \sum_{\xi \in \Lambda} \sum_{k \in \mathbb Z_{q,n}} g_{\xi, k, n+1}^2 A_{\xi, k, n} \\ \nonumber
        & = & - \sum_{\xi, k, n} g_{\xi, k, n+1}^2 a_{\xi, k, n}^2 (\nabla \Phi_k)^{-1} \xi  \otimes \xi (\nabla \Phi_k)^{-T}; 
    \end{eqnarray}

    \item The error $R_{q, \Gamma}$ is made, by virtue of proposition \ref{NewIter}, sufficiently small so that it can be placed into $R_{q+1}$;

    \item The residual error $P_{q+1, \Gamma}$, in view of \eqref{SmalStr}, is given by 
    \begin{equation} \label{big_P_err}
    P_{q+1, \Gamma} = R_{q} - R_{q, 0} + w_{q+1}^{(t)} \mathring \otimes w_{q+1}^{(t)} + (u_q - \bar u_q) \mathring \otimes   w_{q+1}^{(t)} +  w_{q+1}^{(t)} \mathring \otimes (u_q - \bar u_q).
    \end{equation}
\end{itemize}

In the following section we aim to approximately erase the error $S_{q, \Gamma}$ by the nonlinear self-interaction of highly-oscillatory Nash perturbations, as well as to show that the tuple $(p_{q, \Gamma}, R_{q, \Gamma}, P_{q+1, \Gamma})$ satisfies already estimates compatible with the $(q+1)^{\text{th}}$ stage. We remark that, in view of the temporal support properties of proposition \ref{NewIter} and lemma \ref{w_t_estim}, it holds that 
\begin{eqnarray*}
     \supp_t S_{q, \Gamma} \cup \supp_t R_{q, \Gamma} \cup \supp_t P_{q+1, \Gamma} &\subset& [-2 + (\delta_q^{1/2} \lambda_q)^{-1} - 2\Gamma \tau_q, -1 - (\delta_q^{1/2} \lambda_q)^{-1} + 2\Gamma \tau_q] \\ 
        && \cup [1 + (\delta_q^{1/2} \lambda_q)^{-1} - 2 \Gamma \tau_q, 2 - (\delta_q^{1/2} \lambda_q)^{-1} + 2 \Gamma \tau_q]. \nonumber
\end{eqnarray*}

\section{The Nash step} \label{section_Nash}

\subsection{Preliminary: Mollification along the flow} \label{Section_Moli_Along_flow} Before we turn to the construction of the Nash perturbation, we first perform another regularization procedure which aims to solve the loss of material derivative problem. Specifically, in order to be able to propagate the estimates \eqref{ind_est_5} on the material derivative of the new Reynolds stress, we will need estimates on the second material derivative of the old error. This can be seen by considering the transport error, which already involves one material derivative of the stress at the $q^{\text{th}}$ stage. For this purpose, we use the mollification along the flow introduced in \cite{IsPhD}. 

Let $\tilde X_t$ denote the Lagrangian flow of $ \bar u_{q,\Gamma}$ starting at $t$, which we have already introduced in \eqref{Lagr_t_gam}, and fix a standard temporal mollifier $\varrho$ as well as the material mollification scale 
\begin{equation*}
    \ell_{t, q} = \delta_{q}^{-1/2}\lambda_q^{-1/3} \lambda_{q+1}^{-2/3}.
\end{equation*}
Note that 
\begin{equation*}
    \delta_{q}^{1/2} \lambda_q < \ell_{t,q}^{-1} < \delta_{q+1}^{1/2} \lambda_{q+1}, 
\end{equation*}
and, in fact, for all $\alpha > 0$ sufficiently small, it also holds that 
\begin{equation*}
    \tau_{q+1} < \ell_{t, q} < \mu_{q+1}^{-1} <  \tau_q. 
\end{equation*}
We define 
\begin{equation} \label{tmoli_def}
    \bar R_{q,n}(x, t) = \int_{-\ell_{t,q}}^{\ell_{t, q}} R_{q,n} (\tilde X_t(x, t + s), t + s) \varrho_{\ell_{t, q}}(s) ds.
\end{equation}

We now state the main estimates on $\bar R_{q,n}$. These were proven for example in \cite{IsPhD}, or \cite{BDLS16}. We repeat the proof here for the sake of completeness. 

\begin{lem} \label{tmoli_estim}
Assume the $R_{q,n}$ satisfy~\eqref{NewIter_1}-\eqref{NewIter_4}. Then, the following estimates hold: 
\begin{equation} \label{tmoli1}
    \|\bar R_{q,n}\|_N \lesssim \delta_{q+1,n} \lambda_q^{N-\alpha}, \,\,\, \forall N\in\{0,1,...,L-2\},
\end{equation}
\begin{equation} \label{tmoli2}
    \|\bar D_{t, \Gamma} \bar R_{q,n} \|_N \lesssim \tau_q^{-1} \delta_{q+1,n} \lambda_q^{N- \alpha}, \,\,\, \forall N\in \{0,1,...,L-2\},
\end{equation}
\begin{equation} \label{tmoli3}
    \|\bar D_{t, \Gamma}^2 \bar R_{q,n} \|_N \lesssim \ell_{t,q}^{-1} \tau_q^{-1} \delta_{q+1,n} \lambda_q^{N- \alpha}, \,\,\, \forall N\in \{0,1,...,L-2\}
\end{equation}
\begin{equation} \label{tmoli4}
    \|\bar R_{q,n}\|_{N+ L - 2} \lesssim \delta_{q+1,n}\lambda_q^{L-2-\alpha} \ell_q^{-N}, \,\,\, \forall N \geq 0
\end{equation}
\begin{equation} \label{tmoli5}
    \|\bar D_{t, \Gamma} \bar R_{q,n} \|_{N+L-2} \lesssim \tau_q^{-1} \delta_{q+1,n} \lambda_{q}^{L-2-\alpha} \ell_q^{-N}, \,\,\, \forall N \geq 0.
\end{equation}
\begin{equation} \label{tmoli6}
    \|\bar D_{t, \Gamma}^2 \bar R_{q,n} \|_{N+L-2} \lesssim \ell_{t,q}^{-1} \tau_q^{-1} \delta_{q+1,n} \lambda_{q}^{L-2-\alpha} \ell_q^{-N}, \,\,\, \forall N \geq 0.
\end{equation}
with implicit constants depending on $\Gamma$, $M$, $\alpha$ and $N$.
\end{lem}

\begin{proof}
To prove \eqref{tmoli1} and \eqref{tmoli4}, we note that since $\ell_{t, q} \leq \|\bar u_{q, \Gamma}\|_1^{-1}$ for all sufficiently large choices of $a_0$, the estimates of corollary \ref{Flow_gam_estim} apply. Then, by proposition \ref{comp_estim},
\begin{equation*}
    \|\bar R_{q,n}\|_N \lesssim \|R_{q,n} \circ \tilde X_t\|_N \lesssim \|R_{q,n}\|_N \|D \tilde X_t\|_0^N + \|R_{q,n}\|_1 \|D \tilde X_t\|_{N-1}, 
\end{equation*}
and, thus, \eqref{tmoli1} and \eqref{tmoli4} follow from proposition \ref{NewIter} and corollary \ref{Flow_gam_estim}.

The idea of definition \eqref{tmoli_def} is that the following hold:
\begin{equation*}
    \bar D_{t, \Gamma} \bar R_{q,n}(x, t) = \int_{-\ell_{t, q}}^{\ell_{t,q}} \bar D_{t,\Gamma} R_{q,n}(\tilde X_t(x, t+s), t+s) \varrho_{\ell_{t,q}}(s) ds,
\end{equation*}
and 
\begin{equation*}
    \bar D_{t, \Gamma}^2 \bar R_{q,n}(x, t) = - \ell_{t,q}^{-1} \int_{-\ell_{t, q}}^{\ell_{t,q}} \bar D_{t, \Gamma} R_{q,n}(\tilde X_t(x, t+s), t+s) \varrho_{\ell_{t,q}}'(s) ds.
\end{equation*}

We have 
\begin{equation*}
    \|\bar D_{t, \Gamma} R_{q,n}\|_N \lesssim \|\bar D_t R_{q,n}\|_N + \|w_{q+1}^{(t)} \cdot \nabla R_{q,n}\|_N \lesssim \|\bar D_t R_{q,n}\|_N + \|w_{q+1}^{(t)}\|_0\|R_{q,n}\|_{N+1} + \|w_{q +1}^{(t)}\|_N \|R_{q,n}\|_1.
\end{equation*}
Then, using that 
\begin{equation*}
    \frac{\delta_{q+1} \lambda_q \ell_q^{-\alpha}}{\mu_{q+1}} \lambda_q \leq \tau_q^{-1},
\end{equation*}
we obtain:
\begin{equation*}
    \|\bar D_{t, \Gamma} R_{q,n}\|_N \lesssim \tau_q^{-1} \delta_{q+1,n} \lambda_q^{N-\alpha}, \,\,\, \forall N \in \{0,1,...,L-2\}
\end{equation*}
\begin{equation*}
    \|\bar D_{t, \Gamma} R_{q,n}\|_{N + L -2} \lesssim \tau_q^{-1} \delta_{q+1,n} \lambda_q^{L-2-\alpha} \ell_q^{-N}, \,\,\, \forall N \geq 0,
\end{equation*}
where we have used proposition \ref{NewIter} and lemma \ref{w_t_estim}. Therefore, \eqref{tmoli2} and \eqref{tmoli5} follow by appealing to proposition \ref{comp_estim} to write 
\begin{equation*}
    \|\bar D_{t, \Gamma} \bar R_{q,n}\|_N \lesssim \|\bar D_{t, \Gamma} R_{q,n}\|_{N} \|D \tilde X_t\|_0^N + \|\bar D_{t, \Gamma} R_{q,n}\|_1 \|D \tilde X_t\|_{N-1}.
\end{equation*}

Similarly, \eqref{tmoli3} and \eqref{tmoli6} follow from 
\begin{equation*}
    \|\bar D_{t, \Gamma}^2 \bar R_{q,n}\|_N \lesssim \ell_{t,q}^{-1}\big( \|\bar D_{t, \Gamma} R_{q,n}\|_{N} \|D \tilde X_t\|_0^N + \|\bar D_{t, \Gamma} R_{q,n}\|_1 \|D \tilde X_t\|_{N-1} \big).
\end{equation*}
\end{proof}

\subsection{Construction of the Nash perturbation} \label{Nash_Construction}
In order to quantify the oscillatory behaviour of the Nash perturbation, we use as building blocks shear flows in the directions $\xi$ given by lemma \ref{geom}. For each $\xi \in \Lambda$, we define $\mathbb{W}_{\xi}:\mathbb T_{2\pi}^2 \rightarrow \mathbb R^2$ by
\begin{equation} \label{defW}
    \mathbb{W}_\xi (x) = \frac{1}{\sqrt{2}} \big ( e^{i \xi^\perp \cdot x} + e^{- i \xi^\perp \cdot x} \big) \xi, 
\end{equation}
where $\mathbb T_{2\pi}^2 = \mathbb R^2/(2\pi \mathbb Z)^2$. Let us also denote the corresponding stream-function by
\begin{equation*}
    \Psi_\xi(x) = \frac{i}{\sqrt{2}}\big( e^{i\xi^\perp \cdot x} - e^{-i \xi^\perp \cdot x}\big). 
\end{equation*}
The relevant properties of these vector fields are gathered in the following simple lemma. 

\begin{lem}
The vector fields $\mathbb{W}_{\xi}:\mathbb T_{2\pi}^2 \rightarrow \mathbb R^2$ defined by \eqref{defW} satisfy 
\begin{equation*}
    \begin{cases}
        \div (\mathbb{W}_{\xi} \otimes \mathbb{W}_{\xi}) = 0, \\ 
        \div \mathbb{W}_{\xi} = 0.
    \end{cases}
\end{equation*}
and 
\begin{equation*}
    \fint_{\mathbb T_{2\pi} ^2} \mathbb{W}_{\xi} \otimes \mathbb{W}_{\xi} = \xi \otimes \xi.
\end{equation*}
\end{lem}

In the following, $\tilde\Phi_k$ will stand for  the backwards flow of $\bar u_{q,\Ga}$, having origin at $t = t_k$:
\begin{equation*}
    \begin{cases}
        \partial_t \td\Phi_k + \bar u_{q,\Ga} \cdot \nabla \td\Phi_k = 0, \\ 
        \Phi \big|_{t = t_k} = x.
    \end{cases}
\end{equation*}
With 
\begin{equation}\label{def.ba}
    \bar a_{\xi, k, n} = \delta_{q+1, n}^{1/2} \chi_k \gamma_\xi \big(\nabla \tilde \Phi_k \nabla \tilde \Phi_k^T - \nabla \tilde \Phi_k \frac{\bar R_{q,n}}{\delta_{q+1, n}} \nabla \tilde \Phi_k^T \big),
\end{equation}
we define the principal part of the Nash perturbation by 
\begin{equation}
    w_{q+1}^{(p)} := \sum_{n = 0}^{\Gamma - 1} \sum_{k \in \mathbb Z_{q,n}} \sum_{\xi\in\Lambda} g_{\xi,k,n+1}(\mu_{q+1}t)\bar a_{\xi,k,n} (\na\tilde\Phi_k)^{-1} \mathbb{W}_\xi (\la_{q+1} \tilde\Phi_k),
\end{equation}
where $g_{\xi,k,n+1}$ is defined in ~\eqref{def.gxikn1}. As argued before, since 
\begin{equation*}
    \|\nabla \tilde \Phi_k \frac{\bar R_{q,n}}{\delta_{q+1, n}}\nabla \tilde \Phi_k^T\|_0 \lesssim \lambda_q^{-\alpha},
\end{equation*}
and 
\begin{equation*}
    \|\I - \nabla \tilde \Phi_k\|_0 \lesssim \lambda_q^{-\alpha},
\end{equation*}
for any $\alpha > 0$, we can choose $a_0$ sufficiently large so that $\bar a_{\xi, k, n}$ is indeed well-defined. 

Note also that while 
\begin{equation*}
    (\na\tilde\Phi_k)^{-1} \mathbb{W}_\xi (\la_{q+1} \tilde\Phi_k) = \frac{1}{\lambda_{q+1}} \nabla^\perp \big( \Psi_\xi(\lambda_{q+1} \tilde \Phi_k)\big) 
\end{equation*}
is divergence-free, $w^{(p)}_{q+1}$ need not be. To rectify this, we define the ``corrector'' part of the Nash perturbation as
\begin{equation}\label{w.c}
    w_{q+1}^{(c)} := \frac{1}{\lambda_{q+1}} \sum_{n = 0}^{\Gamma - 1} \sum_{k \in \mathbb Z_{q,n}} \sum_{\xi\in\Lambda} g_{\xi,k,n+1} \nabla^\perp \bar a_{\xi, k, n} \Psi_\xi(\lambda_{q+1}\tilde \Phi_k). 
\end{equation}
The total Nash perturbation $w_{q+1}^{(s)} := w_{q+1}^{(p)} + w_{q+1}^{(c)}$ is, then, manifestly divergence-free: 
\begin{equation} \label{total_Nash}
    w_{q+1}^{(s)} = \frac{1}{\lambda_{q+1}} \nabla^\perp \bigg( \sum_{n = 0}^{\Gamma - 1} \sum_{k \in \mathbb Z_{q,n}} \sum_{\xi\in\Lambda} g_{\xi, k, n+1} \bar a_{\xi, k, n} \Psi_\xi(\lambda_{q+1} \tilde \Phi_k) \bigg).
\end{equation}
Let us also denote the total perturbation by 
\begin{equation*}
    w_{q+1} = w_{q+1}^{(t)} + w_{q+1}^{(s)}.
\end{equation*}

We remark here that the terms in the sum above have pair-wise disjoint temporal supports. Indeed, since $\supp_t \bar a_{\xi, k, n} \subset \supp_t \chi_k$, we have that 
\begin{equation*}
    \supp_t \bar a_{\xi, k, n} \cap \supp_t \bar a_{\eta, j, m} \neq \emptyset \implies |k-j| \leq 1,
\end{equation*}
but, then, either $k$ and $j$ have distinct parities, in which case $\supp_t g_{\xi, k, n+1} \cap \supp_t g_{\eta, j, m+1} = \emptyset$, or $k=j$, in which case 
\begin{equation*}
    \supp_t g_{\xi, k, n+1} \cap \supp_t g_{\xi, k, m+1} \neq \emptyset \implies (\xi, n) = (\eta, m),
\end{equation*}
by lemma \ref{osc_prof}.

\subsection{The Euler-Reynolds system after the \texorpdfstring{$(q+1)^{\text{th}}$}{qoth} stage} 

We let the velocity field $u_{q+1}$ of proposition \ref{Main_prop} be 
\begin{equation}
    u_{q+1} = u_{q, \Gamma} + w_{q+1}^{(s)} = u_q + w_{q+1}^{(t)} +  w_{q+1}^{(s)} = u_q + w_{q+1},
\end{equation}
which leads to the Euler-Reynolds system 
\begin{eqnarray*}
    \partial_t u_{q+1} + \div (u_{q+1} \otimes u_{q+1}) + \nabla p_{q+1} = \div R_{q+1}, 
\end{eqnarray*}
with 
\begin{equation}\label{def.pq1}
    p_{q+1} = p_{q, \Gamma}  + \langle \bar u_q - u_q, w_{q+1}^{(s)} \rangle,
\end{equation}
and 
\begin{equation}
    R_{q+1} = R_{q+1, L} + R_{q+1, O} +  R_{q+1, R},
\end{equation}
where 
\begin{itemize}
    \item $R_{q+1, L}$ is the linear error and it is defined by 
    \begin{equation*}
    R_{q+1, L} = \mathcal{R}\big(\bar D_{t, \Gamma} w_{q+1}^{(s)} + w_{q+1}^{(s)}\cdot \nabla \bar u_{q,\Gamma} \big);
     \end{equation*}

    \item $R_{q+1, O}$ is the oscillation error and it is defined by
    \begin{equation*}
        R_{q+1, O} = \mathcal{R} \div(S_{q, \Gamma} + w_{q+1}^{(s)} \otimes w_{q+1}^{(s)});
    \end{equation*}
    
     \item $R_{q+1, R}$ is the residual error and it is defined by
     \begin{equation*}
         R_{q+1, R}=  R_{q, \Gamma} + P_{q+1, \Gamma} + w_{q+1}^{(s)} \mathring \otimes (u_q - \bar u_q) + (u_q - \bar u_q)\mathring \otimes w_{q+1}^{(s)}.
     \end{equation*}
\end{itemize}
We note that $R_{q+1, O}$ and $R_{q+1,L}$ are indeed well-defined, since the operator $\mathcal{R}$ is applied to vector fields which are either a divergence or a curl.

\begin{rem} 
We have already noted that $w_{q+1}^{(t)}$, $R_{q, \Gamma}$, $P_{q+1, \Gamma}$ and $S_{q,\Gamma}$ have temporal supports contained in the set 
\begin{equation*}
    [-2 + (\delta_q^{1/2} \lambda_q)^{-1} - 2\Gamma \tau_q, -1 - (\delta_q^{1/2} \lambda_q)^{-1} + 2\Gamma \tau_q] \cup [1 + (\delta_q^{1/2} \lambda_q)^{-1} - 2 \Gamma \tau_q, 2 - (\delta_q^{1/2} \lambda_q)^{-1} + 2 \Gamma \tau_q].
\end{equation*}
It is clear from the definition \eqref{total_Nash} that the same holds for $w_{q+1}^{(s)}$, and, thus, also for $w_{q+1}$ and $R_{q+1}$. But, for any $\alpha > 0$, $a_0$ can be chosen sufficiently large in terms of $\alpha$, $b$ and $\beta$ so that
\begin{equation*}
    (\delta_{q+1}^{1/2} \lambda_{q+1})^{-1} + 2 \Gamma (\delta_{q}^{1/2} \lambda_q)^{-1} \lambda_{q+1}^{-\alpha} < (\delta_q^{1/2} \lambda_q)^{-1}.
\end{equation*}
It follows, then, that the condition \eqref{Main_prop_eqn_2} on the temporal support of the total perturbation and the inductive propagation of \eqref{ind_est_6} to level $q+1$ hold true.
\end{rem}

\subsection{Estimates for the Nash perturbation}
We start by collecting the required estimates on the amplitudes of the Nash perturbation. These results follow by precisely the same arguments as those given in the proof of lemma~\ref{a_estim}, taking this time into account the bounds of corollary~\ref{Flow_gam_estim} and lemma \ref{tmoli_estim}.

\begin{lem} \label{a_bar_estim}
    The functions $\bar a_{\xi,k,n}$ defined in~\eqref{def.ba} satisfy the following estimates: 
    \begin{equation} \label{a_bar_estim_1}
    \|\bar a_{\xi, k, n}\|_N \lesssim \delta_{q+1, n}^{1/2} \lambda_q^N, \,\,\, \forall N \in \{0,1,..., L-3\}
\end{equation}
\begin{equation} \label{a_bar_estim_2}
     \|\bar D_{t, \Gamma} \bar a_{\xi, k, n}\|_N \lesssim \delta_{q+1, n}^{1/2} \tau_q^{-1} \lambda_q^N \,\,\, \forall N \in \{0,1,..., L-3\}
\end{equation}
\begin{equation} \label{a_bar_estim_3}
    \|\bar a_{\xi, k, n}\|_{N+L-3} \lesssim \delta_{q+1, n}^{1/2} \lambda_q^{L-3} \ell_q^{-N}, \,\,\, \forall N \geq 0
\end{equation}
\begin{equation} \label{a_bar_estim_4}
    \|\bar D_{t, \Gamma} \bar a_{\xi, k, n}\|_{N+L-3}  \lesssim \delta_{q+1, n}^{1/2} \lambda_q^{L- 3}  \tau_q^{-1} \ell_q^{-N}, \,\,\, \forall N \geq 0,
\end{equation}
where the implicit constants depend on $\Gamma$, $M$, $\al$ and $N$.
\end{lem}

We are now ready to estimate the perturbation $w_{q+1}^{(s)}$ and, thus, fix the constant $M_0$ and verify the inductive propagation of \eqref{ind_est_1} and \eqref{ind_est_2}, as well as the validity of \eqref{Main_prop_eqn}.

\begin{lem} \label{velo_estim_fin}
    There exists a constant $M_0 > 0$, depending only on $\beta$ and $L$, such that 
    \begin{equation} \label{spatial_pert_est}
        \|w_{q+1}^{(s)}\|_{N} \leq \frac{M_0}{2} \delta_{q+1}^{1/2} \lambda_{q+1}^N \,\,\, \forall N \in \{0, 1,..., L\},
    \end{equation}
    and, consequently,
    \begin{equation} \label{total_pert_est}
        \|w_{q+1}\|_N \leq M_0 \delta_{q+1}^{1/2} \lambda_{q+1}^N \,\,\, \forall N \in \{0,1,...,L\}.
    \end{equation}
\end{lem}

\begin{proof}
Since $\{g_{\xi, k, n+1} \bar a_{\xi, k, n}\}_{\xi, k, n}$ have disjoint temporal supports, it follows from \eqref{total_Nash} that
\begin{eqnarray*}
    [w_{q+1}^{(s)}]_N &\leq& \frac{1}{\lambda_{q+1}} \sup_{\xi, k, n} |g_{\xi, k, n+1}| [\bar a_{\xi, k, n} \Psi_\xi(\lambda_{q+1} \tilde \Phi_k)]_{N+1} \\ 
    & \leq & \frac{C_N}{\lambda_{q+1}} \sup_{\xi, k, n} |g_{\xi, k, n+1}| \big([\Psi_\xi(\lambda_{q+1} \tilde \Phi_k)]_{N+1} \|\bar a_{\xi, k, n}\|_0  + \|\Psi_\xi(\lambda_{q+1} \tilde \Phi_k)\|_0 \|\bar a_{\xi, k, n}\|_{N+1}\big),
\end{eqnarray*}
where, in the context of this proof, $C_N$ is a constant depending only on $N$ which might change from line to line.

Also, by proposition \ref{comp_estim},
\begin{eqnarray*}
    [\Psi_\xi(\lambda_{q+1} \tilde \Phi_k)]_{N+1} & \leq & C_N\big(\|D\Psi_\xi(\lambda_{q+1}\cdot)\|_N \|\nabla \tilde \Phi\|_0^{N+1}    + [\Psi_\xi(\lambda_{q+1}\cdot)]_1 \|\nabla \tilde \Phi_k\|_N) \\ 
    & \leq & C_N(\lambda_{q+1}^{N+1} + \tilde C \lambda_{q+1} \ell_q^{-N}),
\end{eqnarray*}
where $\tilde C$ is the constant depending on $\Ga$, $M$, $\alpha$, and $N$ which is implicit in corollary \ref{Flow_gam_estim} (recall also remark \ref{remark_flow_bd}). But if $a_0$ is sufficiently large in terms of $\Gamma$ (thus, on $\beta$), $M$, $\alpha$, $b$ and $L$, we have
\begin{equation*}
    \tilde C \ell_q^{-N} \lambda_{q+1} \leq \lambda_{q+1}^{N+1}, \,\,\, \forall N \in \{0,1,...,L\}.
\end{equation*}
It is also clear by definition that $\|\bar a_{\xi, k, n}\|_0 \leq C \delta_{q+1}^{1/2}$, where $C$ depends only on the geometric functions $\gamma_\xi$ of proposition \ref{geom}. Moreover,
\begin{equation*}
    \|\Psi_\xi(\lambda_{q+1} \tilde \Phi_k)\|_0 \|\bar a_{\xi, k, n}\|_{N+1} \lesssim \delta_{q+1}^{1/2} \ell_q^{-N-1},
\end{equation*}
and so, we can choose, as before, $a_0$ large enough in order to finally ensure that
\begin{equation*}
    [w_{q+1}^{(s)}]_N \leq C_L \delta_{q+1}^{1/2} \lambda_{q+1}^N \sup_{\xi, k, n} |g_{\xi, k, n}|, \,\,\, \forall N \in \{0,1,...,L\},
\end{equation*}
where $C_L$ is a constant depending only on $L$. It is clear from the definition of $g_{\xi, k, n}$ in lemma \ref{osc_prof} that $\sup |g_{\xi, k, n}|$ depends only on $\Gamma$ and, thus, only on $\beta$. We define, then, 
\begin{equation*}
    M_0 = 2 C_L \sup_{\xi, k, n} |g_{\xi, k, n}|, 
\end{equation*}
and conclude the proof of \eqref{spatial_pert_est}. 

Finally, from lemma \ref{w_t_estim}, we have that 
\begin{equation*}
    \|w_{q+1}^{(t)}\|_N \lesssim \frac{\delta_{q+1}\lambda_q}{\mu_{q+1}} \ell_q^{-N-\alpha} \lesssim  \delta_{q+1}^{1/2}\bigg(\frac{\lambda_q}{\lambda_{q+1}}\bigg)^{1/3} \lambda_{q+1}^{N}.
\end{equation*}
Then, by choosing $a_0$ sufficiently large, we can ensure that
\begin{equation*}
    \|w_{q+1}^{(t)}\|_N \leq \frac{M_0}{2} \delta_{q+1}^{1/2} \lambda_{q+1}^N, \,\,\, \forall N \in \{0,1,...,L\}.
\end{equation*}
The conclusion follows. 
\end{proof}

\begin{cor} \label{velo_corrol_fin}
The following hold: 
\begin{equation} \label{propag_velo_1}
    \|u_{q+1}\|_0 \leq M (1 - \delta_{q+1}^{1/2}), 
\end{equation}
\begin{equation} \label{propag_velo_2}
    \|u_{q+1}\|_N \leq M \delta_{q+1}^{1/2} \lambda_{q+1}^N, \,\,\, \forall N \in \{1,2,...,L\}
\end{equation}
\begin{equation} \label{propag_velo_3}
    \|u_{q+1} - u_q\|_0 + \frac{1}{\lambda_{q+1}} \|u_{q+1} - u_q\|_1 \leq 2 M \delta_{q+1}^{1/2}.
\end{equation}
\end{cor}

\begin{proof}
    Equation \eqref{propag_velo_3} follows immediately from lemma \ref{velo_estim_fin}. Moreover, 
    \begin{equation*}
        \|u_{q+1}\|_0 \leq \|u_q\|_0 + \|u_{q+1}-u_q\|_0 \leq M(1 - \delta_q^{1/2} + \delta_{q+1}^{1/2}) \leq M(1 - \delta_{q+1}^{1/2}),
    \end{equation*}
    where the last inequality holds whenever $a_0$ is sufficiently large. 

    Finally, from lemma \ref{velo_estim_fin}, we have
    \begin{equation*}
        \|u_{q+1}\|_N \leq M\delta_{q}^{1/2}\lambda_q^N + M_0 \delta_{q+1}^{1/2} \lambda_{q+1}^N,
    \end{equation*}
    and, therefore, $a_0$ can be chosen large enough depending on $M$, $M_0$, $b$ and $\beta$ to ensure that 
    \begin{equation*}
        M \delta_{q}^{1/2} \lambda_q \leq (M-M_0) \delta_{q+1}^{1/2} \lambda_{q+1},
    \end{equation*}
    from which the remaining conclusion \eqref{propag_velo_2} follows.
\end{proof}

\subsection{Estimates for the linear error \texorpdfstring{$R_{q + 1, L}$}{rrrr}} 

We write 
\begin{equation*}
    R_{q+1, L} = \underbrace{\mathcal{R} (w_{q+1}^{(s)} \cdot \nabla \bar u_{q, \Gamma})}_{\text{Nash error}} + \underbrace{\mathcal{R}( \bar D_{t, \Gamma} w_{q+1}^{(s)})}_{\text{Transport error}},
\end{equation*}
and we estimate the two terms separately. Let us begin, however, by collecting some preliminary estimates on material derivatives. 

\begin{lem} \label{high_mat_der}
    The vector field $u_{q,\Ga}$ defined in~\eqref{def.buqg} and the functions $\bar a_{\xi,k,n}$ defined in~\eqref{def.ba} satisfy the following estimates: 
    \begin{equation} \label{mat_spa_velo_1}
        \|\bar D_{t, \Gamma} \nabla \bar u_{q, \Gamma}\|_{N} \lesssim \delta_q \lambda_q^{N+2}, \,\,\, \forall N\in \{0,1,...,L-4\},
    \end{equation}
    \begin{equation} \label{mat_spa_velo_2}
        \|\bar D_{t, \Gamma} \nabla \bar u_{q, \Gamma}\|_{N + L -4} \lesssim \delta_q \lambda_q^{L-2} \ell_q^{-N}, \,\,\, \forall N\geq 0.
    \end{equation}
    \begin{equation} \label{a_bar_2nd_material}
     \|\bar D_{t, \Gamma}^2 \bar a_{\xi, k, n}\|_N \lesssim \delta_{q+1, n}^{1/2} \tau_q^{-1} \ell_{t,q}^{-1}\lambda_q^N, \,\,\, \forall N \in \{0,1,..., L-4\},
    \end{equation}
    \begin{equation} \label{a_bar_2nd_material_2}
    \|\bar D_{t, \Gamma}^2 \bar a_{\xi, k, n}\|_{N + L-4 } \lesssim \delta_{q+1, n}^{1/2} \lambda_q^{L-4}  \tau_q^{-1} \ell_{t,q}^{-1}\ell_q^{-N}, \,\,\, \forall N \geq 0,
\end{equation}
where the implicit constants depend on $\Gamma$, $M$, $\alpha$ and $N$. 
\end{lem}

\begin{proof}
    We begin by writing 
    \begin{equation*}
        \bar D_{t, \Gamma} \nabla \bar u_{q, \Gamma} = \bar D_{t} \nabla \bar u_{q} + \bar D_{t} \nabla w_{q+1}^{(t)} + w_{q+1}^{(t)}\cdot \nabla \nabla\bar u_{q, \Gamma}.
    \end{equation*}
    For the first term we note that 
    \begin{equation*}
        \bar D_{t} \nabla \bar u_{q} = \nabla \bar D_{t} \bar u_q - (\nabla \bar u_q)^2.
    \end{equation*}
    By mollifying the Euler-Reynolds system \eqref{ER}, we obtain 
    \begin{equation*}
        \bar D_t \bar u_q + \div \big( (u_q \otimes u_q)*\zeta_{\ell_q} - \bar u_q \otimes \bar u_q \big) + \nabla ( p_q * \zeta_{\ell_q} ) = \div R_{q, 0},     
    \end{equation*}
    and, therefore, 
    \begin{eqnarray*}
        \|\bar D_t \nabla \bar u_q\|_N & \lesssim & \|\bar D_t \bar u_q\|_{N+1} + \|\bar u_q\|_{N+1}\|\bar u_q\|_1 \\ 
        & \lesssim & \|(u_q \otimes u_q)*\zeta_{\ell_q} - \bar u_q \otimes \bar u_q\|_{N+2} + \|p_q * \zeta_{\ell_q}\|_{N+2} + \|R_{q, 0}\|_{N+2} + \|\bar u_q\|_{N+1}\|\bar u_q\|_1.
    \end{eqnarray*}
    For the first term, we use the Constantin-E-Titi commutator estimate of proposition \ref{CET_comm} together with the inductive assumptions on $u_q$ to conclude that 
    \begin{equation*}
        \|(u_q \otimes u_q)*\zeta_{\ell_q} - \bar u_q \otimes \bar u_q\|_{N+2} \lesssim \delta_q \lambda_q^{N+2}, \,\,\, \forall N \in \{0,1,...,L-4\}, 
    \end{equation*}
    \begin{equation*}
        \|(u_q \otimes u_q)*\zeta_{\ell_q} - \bar u_q \otimes \bar u_q\|_{N+L-2} \lesssim \ell_q^{2-N} \delta_q \lambda_q^{L} \lesssim \delta_q \lambda_q^{L-2} \ell_q^{-N}, \,\,\, \forall N \geq 0.
    \end{equation*}
    In view of the inductive assumptions, the same bounds also hold for the second and fourth terms. The third satisfies better estimates, as already stated in lemma \ref{smoli_estim}. We conclude, then, that $\bar D_{t} \nabla \bar u_q$ satisfies the bounds claimed in \eqref{mat_spa_velo_1} and \eqref{mat_spa_velo_2}. Also, we have by similar considerations that
    \begin{equation*}
        \|\bar D_{t} \nabla w_{q+1}^{(t)}\|_N \lesssim \|\bar D_t w_{q+1}^{(t)}\|_{N+1} + \|\bar u_q\|_{N+1}\|w_{q+1}^{(t)}\|_1 +  \|\bar u_q\|_{1}\|w_{q+1}^{(t)}\|_{N+1}.
    \end{equation*}
    Using now lemma \ref{w_t_estim} and taking into account that $\mu_{q+1} > \delta_q^{1/2} \lambda_q$, we obtain
    \begin{equation*}
        \|\bar D_{t} \nabla w_{q+1}^{(t)}\|_N \lesssim \delta_{q+1} \lambda_q^{N+2} \ell_q^{-\alpha}, \,\,\, \forall N \in \{0, 1,..., L-4\},
    \end{equation*}
    \begin{equation*}
        \|\bar D_{t} \nabla w_{q+1}^{(t)}\|_{N+L-4} \lesssim \delta_{q+1} \lambda_q^{L-2} \ell_q^{-N-\alpha}, \,\,\, \forall N \geq 0.
    \end{equation*}
    Since $\alpha$ can be chosen small enough so that 
    \begin{equation*}
        \delta_{q+1}\ell_q^{-\alpha} < \delta_q,
    \end{equation*}
    we conclude that also $\bar D_t \nabla w_{q+1}^{(t)}$ satisfies estimates compatible with \eqref{mat_spa_velo_1} and \eqref{mat_spa_velo_2}. Finally, 
    \begin{equation*}
        \|w_{q+1}^{(t)} \cdot \nabla \nabla \bar u_{q, \Gamma}\|_N \lesssim \|w_{q+1}^{(t)}\|_N \|\bar u_{q, \Gamma}\|_2 + \|w_{q+1}^{(t)}\|_{0} \|\bar u_{q, \Gamma}\|_{N+2},
    \end{equation*}
    which implies, in view of lemma \ref{w_t_estim} and corollary \ref{Gamma_velo_estim}, 
    \begin{equation*}
        \|w_{q+1}^{(t)} \cdot \nabla \nabla \bar u_{q, \Gamma}\|_N \lesssim \delta_{q+1} \frac{\delta_q^{1/2} \lambda_q}{\mu_{q+1}} \ell_q^{- \alpha} \lambda_q^{N+2}, \,\,\, \forall N \in \{0,1,...,L-4\},
    \end{equation*}
    \begin{equation*}
        \|w_{q+1}^{(t)} \cdot \nabla \nabla \bar u_{q, \Gamma}\|_{L-4+N} \lesssim \delta_{q+1} \frac{\delta_q^{1/2} \lambda_q}{\mu_{q+1}}  \lambda_q^{L-2} \ell_q^{- N - \alpha}, \,\,\, \forall N \geq 0,
    \end{equation*}
    which, once again, is compatible with \eqref{mat_spa_velo_1} and \eqref{mat_spa_velo_2} provided $\alpha$ is chosen sufficiently small. Therefore, \eqref{mat_spa_velo_1} and \eqref{mat_spa_velo_2} are proven.

    We now turn to proving \eqref{a_bar_2nd_material} and \eqref{a_bar_2nd_material_2}. First of all, we have 
    \begin{eqnarray*}
        \bar D_{t, \Gamma}^2 \bar a_{\xi, k, n} &= & \underbrace{\delta_{q+1, n}^{1/2} \partial_t^2 \chi_k \gamma_\xi \big( \nabla \tilde \Phi_k \nabla \tilde \Phi_k^T  - \nabla \tilde \Phi_k \frac{\bar R_{q,n}}{\delta_{q+1, n}} \nabla \tilde \Phi_k^T\big)}_{T_1} \\
        && +  \underbrace{2\delta_{q+1, n}^{1/2} \partial_t \chi_k D\gamma_\xi \big( \nabla \tilde \Phi_k \nabla \tilde \Phi_k^T  - \nabla \tilde \Phi_k \frac{\bar R_{q,n}}{\delta_{q+1, n}} \nabla \tilde \Phi_k^T\big) \bar D_{t, \Gamma} \big( \nabla \tilde \Phi_k \frac{\delta_{q+1, n} \I - \bar R_{q,n}}{\delta_{q+1, n}} \nabla \tilde \Phi_k^T\big)}_{T_2} \\
        && +  \underbrace{\delta_{q+1, n}^{1/2} \chi_k \bar D_{t, \Gamma}\bigg[ D\gamma_\xi \big( \nabla \tilde \Phi_k \nabla \tilde \Phi_k^T  - \nabla \tilde \Phi_k \frac{\bar R_{q,n}}{\delta_{q+1, n}} \nabla \tilde \Phi_k^T\big)\bigg] \bar D_{t, \Gamma} \big( \nabla \tilde \Phi_k \frac{\delta_{q+1, n} \I - \bar R_{q,n}}{\delta_{q+1, n}} \nabla \tilde \Phi_k^T\big)}_{T_3} \\ 
        && + \underbrace{\delta_{q+1, n}^{1/2} \chi_k D\gamma_\xi \big(\nabla \tilde \Phi_k \nabla \tilde \Phi_k^T  - \nabla \tilde \Phi_k \frac{\bar R_{q,n}}{\delta_{q+1, n}} \nabla \tilde \Phi_k^T\big) \bar D_{t, \Gamma}^2 \big( \nabla \tilde \Phi_k \frac{\delta_{q+1, n} \I - \bar R_{q,n}}{\delta_{q+1, n}} \nabla \tilde \Phi_k^T\big)}_{T_4}.
    \end{eqnarray*}
    We estimate each of the four terms above by appealing to the results of corollary \ref{Flow_gam_estim} and lemma \ref{tmoli_estim}. Arguing as in the proof of lemma \ref{a_estim}, we have
    \begin{equation*}
        \|\gamma_\xi \big(\nabla \tilde \Phi_k \nabla \tilde \Phi_k ^T - \nabla \tilde \Phi_k \frac{\bar R_{q,n}}{\delta_{q+1, n}} \nabla \tilde \Phi_k^T\big)\|_N \lesssim \lambda_q^{N}, \,\,\, \forall N \in \{0,1,...,L-4\},
    \end{equation*}
    \begin{equation*}
        \|\gamma_\xi \big( \nabla \tilde \Phi_k \nabla \tilde \Phi_k ^T - \nabla \tilde \Phi_k \frac{\bar R_{q,n}}{\delta_{q+1, n}} \nabla \tilde \Phi_k^T\big)\|_{N+L -4} \lesssim \lambda_q^{L - 4}\ell_q^{-N}, \,\,\, \forall N \geq 0,
    \end{equation*}
    \begin{equation*}
        \|\bar D_{t, \Gamma} \gamma_\xi \big( \nabla \tilde \Phi_k \nabla \tilde \Phi_k ^T - \nabla \tilde \Phi_k \frac{\bar R_{q,n}}{\delta_{q+1, n}} \nabla \tilde \Phi_k^T\big)\|_N \lesssim \lambda_q^{N} \tau_q^{-1}, \,\,\, \forall N \in \{0,1,...,L-4\}
    \end{equation*}
    \begin{equation*}
        \|\bar D_{t, \Gamma} \gamma_\xi \big( \nabla \tilde \Phi_k \nabla \tilde \Phi_k ^T - \nabla \tilde \Phi_k \frac{\bar R_{q,n}}{\delta_{q+1, n}} \nabla \tilde \Phi_k^T\big)\|_{N+L - 4} \lesssim \lambda_q^{L - 4}\tau_q^{-1}\ell_q^{-N}, \,\,\, \forall N \geq 0,
    \end{equation*}
    and likewise when $D\gamma_\xi$ replaces $\gamma_\xi$ in the above (as is the case in the expressions for $T_2$, $T_3$, and $T_4$). We can, then, infer that
    \begin{equation*}
        \|T_1\|_N + \|T_2\|_N + \|T_3\|_N \lesssim \delta_{q+1, n}^{1/2} \tau_q^{-2} \lambda_q^N, \,\,\, \forall N \in \{0,1,...,L-4\}
    \end{equation*}
    \begin{equation*}
        \|T_1\|_{N+L -4} + \|T_2\|_{N+L -4} + \|T_3\|_{N+L-4} \lesssim \delta_{q+1, n}^{1/2} \tau_q^{-2} \lambda_q^{L-4} \ell_q^{-N}, \,\,\, \forall N \geq 0, 
    \end{equation*}
    and it remains to obtain the estimates for $T_4$. For this purpose, we note that 
    \begin{eqnarray*}
        \|\bar D_{t, \Gamma}^2 \big( \nabla \tilde \Phi_k \frac{ \delta_{q+1, n} \I - \bar R_{q,n}}{\delta_{q+1, n}} \nabla \tilde \Phi_k^T\big)\|_N & \lesssim & \|\bar D_{t, \Gamma}^2 \nabla \tilde \Phi_k\|_N + \|\bar D_{t, \Gamma}^2 \nabla \tilde \Phi_k\|_0 \big \| \frac{\delta_{q+1,n}\I- \bar R_{q,n}}{\delta_{q+1, n}}\|_N \\ 
        && + \|\bar D_{t, \Gamma}^2 \nabla \tilde \Phi_k\|_0\|\nabla \tilde \Phi_k\|_N  + \|\bar D_{t, \Gamma} \nabla \tilde \Phi_k\|_N \big \|\bar D_{t, \Gamma} \frac{\bar R_{q,n}}{\delta_{q+1, n}}\big \|_0 \\
        && + \|\bar D_{t, \Gamma} \nabla \tilde \Phi_k\|_0 \big \|\bar D_{t, \Gamma} \frac{\bar R_{q,n}}{\delta_{q+1, n}}\big \|_N + \|\bar D_{t, \Gamma} \nabla \tilde \Phi_k\|_0 \big \|\bar D_{t, \Gamma} \frac{\bar R_{q,n}}{\delta_{q+1, n}}\big \|_0 \|\nabla \tilde \Phi_k\|_N \\ 
        && + \big \| \bar D_{t, \Gamma}^2 \frac{\bar R_{q,n}}{\delta_{q+1, n}}\big\|_N + \big \| \bar D_{t, \Gamma}^2 \frac{\bar R_{q,n}}{\delta_{q+1, n}}\big\|_0 \|\nabla \tilde \Phi_k\|_N \\ 
        && + \|\bar D_{t, \Gamma} \nabla \tilde \Phi_k\|_N \|\bar D_{t, \Gamma} \nabla \tilde \Phi_k\|_0 + \|\bar D_{t, \Gamma} \nabla \tilde \Phi_k\|_0^2 \| \frac{\delta_{q+1,n}\I- \bar R_{q,n}}{\delta_{q+1, n}}\|_N.
    \end{eqnarray*}

    We already have at our disposal estimates for all of the quantities above with the exception of $\bar D_{t, \Gamma}^2 \nabla \tilde \Phi_k$. To estimate this, we write 
    \begin{equation*}
        \bar D_{t, \Gamma} \nabla \tilde \Phi_k = - \nabla \bar u_{q, \Gamma} \nabla \tilde \Phi_k, 
    \end{equation*}
    and, so, taking one more material derivative yields 
    \begin{equation*}
        \bar D_{t, \Gamma}^2 \nabla \tilde \Phi_k = - \bar D_{t, \Gamma} \nabla \bar u_{q, \Gamma} \nabla \tilde \Phi_k - \nabla \bar u_{q, \Gamma} \bar D_{t, \Gamma} \nabla \tilde \Phi_k. 
    \end{equation*}
    Then, 
    \begin{eqnarray*}
        \|\bar D_{t, \Gamma}^2 \nabla \tilde \Phi_k\|_N & \lesssim & \|\bar D_{t, \Gamma} \nabla \bar u_{q, \Gamma}\|_N + \|\bar D_{t, \Gamma} \nabla \bar u_{q, \Gamma}\|_0 \|\nabla \tilde \Phi_k\|_N \\
        && + \|\nabla \bar u_{q, \Gamma}\|_N \|\bar D_{t, \Gamma} \nabla \tilde \Phi_k\|_0 + \|\nabla \bar u_{q, \Gamma}\|_0 \|\bar D_{t, \Gamma} \nabla \tilde \Phi_k\|_N,
    \end{eqnarray*}
    which, together with \eqref{mat_spa_velo_1}, \eqref{mat_spa_velo_2}, corollary \ref{Gamma_velo_estim}, and corollary \ref{Flow_gam_estim}, implies 
    \begin{equation*}
        \|\bar D_{t, \Gamma}^2 \nabla \tilde \Phi_k\|_N \lesssim \delta_q \lambda_q^2 \lambda_q^N \lesssim \tau_q^{-2} \lambda_q^N, \,\,\, \forall N\in \{0,1,...,L-4\},
    \end{equation*}
    \begin{equation*}
        \|\bar D_{t, \Gamma}^2 \nabla \tilde \Phi_k\|_{N+L-4} \lesssim \delta_q \lambda_q^2 \lambda_q^{L-4} \ell_q^{-N} \lesssim \tau_q^{-2}  \lambda_q^{L-4} \ell_q^{-N}, \,\,\, \forall N \geq 0.
    \end{equation*}
    Using this, corollary \ref{Flow_gam_estim} and lemma \ref{tmoli_estim}, we obtain 
    \begin{equation*}
        \|\bar D_{t, \Gamma}^2 \big( \nabla \tilde \Phi_k \frac{\delta_{q+1, n}\I - \bar R_{q,n}}{\delta_{q+1, n}} \nabla \tilde \Phi_k^T\big)\|_N \lesssim \tau_q^{-1} \ell_{t, q}^{-1} \lambda_q^{N}, \,\,\, \forall N\in \{0,1,...,L-4\},
    \end{equation*}
    \begin{equation*}
        \|\bar D_{t, \Gamma}^2 \big( \nabla \tilde \Phi_k \frac{\delta_{q+1,n}\I - \bar R_{q,n}}{\delta_{q+1, n}} \nabla \tilde \Phi_k^T\big)\|_{N+L-4}\lesssim \tau_q^{-1} \ell_{t, q}^{-1} \lambda_q^{L-4} \ell_q^{-N}, \,\,\, \forall N\geq 0,
    \end{equation*}
    where we note that the terms satisfying the worst bounds are those involving two material derivatives of $\bar R_{q,n}$. The above can now be used to estimate $T_4$ and, thus, conclude the proof of the lemma.
\end{proof}

\begin{lem} \label{Nash_err_estim}
The following estimates hold for the Nash error: 
\begin{equation}
    \|\mathcal{R} (w_{q+1}^{(s)} \cdot \nabla \bar u_{q, \Gamma})\|_N \lesssim \frac{\delta_{q}^{1/2} \delta_{q+1}^{1/2} \lambda_q}{\lambda_{q+1}^{1-\alpha}} \lambda_{q+1}^N, \,\,\, \forall N \geq 0,
\end{equation}
\begin{equation}
     \|\bar D_{t,\Ga}\mathcal{R} (w_{q+1}^{(s)} \cdot \nabla \bar u_{q, \Gamma})\|_N \lesssim \frac{\mu_{q+1}\delta_{q}^{1/2} \delta_{q+1}^{1/2} \lambda_q}{\lambda_{q+1}^{1-\alpha}} \lambda_{q+1}^N, \,\,\, \forall N \geq 0,
\end{equation}
where the implicit constants depend on $\Gamma$, $M$, $\alpha$, and $N$. 
\end{lem}

\begin{proof}
We begin by writing
\begin{equation*}
    \mathcal{R} (w_{q+1}^{(s)} \cdot \nabla \bar u_{q, \Gamma}) = - \frac{1}{\lambda_{q+1}} \mathcal{R} \div \sum_{\xi, k, n} g_{\xi, k, n+1} \bar a_{\xi, k, n} \Psi_\xi (\lambda_{q+1}\tilde \Phi_k) \nabla^\perp \bar u_{q, \Gamma}.
\end{equation*}
Since $\{g_{\xi, k, n+1} \bar a_{\xi, k, n}\}$ have disjoint supports, and $\mathcal{R} \div$ is a sum of operators of Calder\'on-Zygmund type, we obtain
\begin{eqnarray*}
    \|\mathcal{R} (w_{q+1}^{(s)} \cdot \nabla \bar u_{q, \Gamma})\|_{N+\alpha} &\lesssim& \frac{1}{\lambda_{q+1}} \sup_{\xi, k, n} \|\bar a_{\xi, k, n} \Psi(\lambda_{q+1} \tilde \Phi_k) \nabla^\perp \bar u_{q, \Gamma}\|_{N+\alpha} \\ 
    & \lesssim & \frac{1}{\lambda_{q+1}}\sup_{\xi, k, n} \big(\|\bar a_{\xi, k, n} \nabla^\perp \bar u_{q, \Gamma}\|_{N+\alpha} \|\Psi_\xi(\lambda_{q+1} \tilde \Phi_k)\|_0 \\
    && + \|\bar a_{\xi, k, n} \nabla^\perp \bar u_{q, \Gamma}\|_0 \|\Psi_\xi(\lambda_{q+1} \tilde \Phi_k)\|_{N+ \alpha}\big).
\end{eqnarray*}
By proposition \ref{comp_estim}, and corollary \ref{Flow_gam_estim},
\begin{equation*}
    \|\Psi_\xi(\lambda_{q+1} \tilde \Phi_k)\|_N \lesssim \lambda_{q+1}^N + \lambda_{q+1} \ell_q^{-N+1} \lesssim \lambda_{q+1}^N,
\end{equation*}
whenever $N \geq 1$. Moreover, lemma \ref{a_bar_estim} and corollary \ref{Gamma_velo_estim} imply 
\begin{equation*}
    \|\bar a_{\xi, k, n} \nabla^\perp \bar u_{q, \Gamma}\|_{N+\alpha} \lesssim \delta_{q+1}^{1/2} \delta_q^{1/2} \lambda_q \ell_q^{-N-\alpha}.
\end{equation*}
Therefore, 
\begin{equation*}
    \|\mathcal{R} (w_{q+1}^{(s)} \cdot \nabla \bar u_{q, \Gamma})\|_{N+\alpha} \lesssim \frac{1}{\lambda_{q+1}} \delta_{q}^{1/2} \delta_{q+1}^{1/2} \lambda_q(\lambda_{q+1}^{N+\alpha} + \ell_q^{-N-\alpha}) \lesssim \frac{\delta_q^{1/2}\delta_{q+1}^{1/2} \lambda_q}{\lambda_{q+1}^{1-\alpha}}\lambda_{q+1}^{N}\,.
\end{equation*}
To estimate the material derivative, we write
\begin{align*}
    \bar D_{t,\Ga} \mathcal{R} (w_{q+1}^{(s)} \cdot \nabla \bar u_{q, \Gamma}) &= - \underbrace{\frac{1}{\lambda_{q+1}} \mathcal{R} \div \bar D_{t,\Ga} \sum_{\xi, k, n} g_{\xi, k, n+1} \bar a_{\xi, k, n} \Psi_\xi (\lambda_{q+1}\tilde \Phi_k) \nabla^\perp \bar u_{q, \Gamma}}_{T_1}\\
    &\qquad - \underbrace{\frac{1}{\lambda_{q+1}} \left[\bar u_{q,\Ga} \cdot\na , \mathcal{R} \div \right] \sum_{\xi, k, n} g_{\xi, k, n+1} \bar a_{\xi, k, n} \Psi_\xi (\lambda_{q+1}\tilde \Phi_k) \nabla^\perp \bar u_{q, \Gamma}}_{T_2}
\end{align*}
The first term $T_1$ can be estimated as
\begin{align*}
        \|T_1\|_{N+\alpha} &\lec \frac{1}{\lambda_{q+1}} \sup_{\xi, k, n} \big(\|\Psi_{\xi}(\lambda_{q+1} \tilde \Phi_k)\|_{N+\alpha} \|\bar D_{t, \Gamma}(g_{\xi,k,n+1} \bar a_{\xi, k, n} \nabla^\perp \bar u_{q, \Gamma})\|_0 \\
        & \qquad + \|\Psi_{\xi}(\lambda_{q+1} \tilde \Phi_k)\|_{0} \|\bar D_{t, \Gamma}(g_{\xi,k,n+1} \bar a_{\xi, k, n} \nabla^\perp \bar u_{q, \Gamma})\|_{N+\alpha}\big),
\end{align*}
where we use the fact that $\Psi_\xi(\lambda_{q+1} \tilde \Phi_k)$ is transported by the flow of $\bar u_{q, \Gamma}$.
Using corollary \ref{Gamma_velo_estim} and the results of lemmas \ref{high_mat_der} and \ref{a_bar_estim}, we compute 
\begin{eqnarray*}
    \|\bar D_{t, \Gamma}\big(g_{\xi, k, n+1} \bar a_{\xi, k, n} \nabla^\perp \bar u_{q, \Gamma}\big)\|_N &\lesssim& |\partial_t g_{\xi, k, n+1}|\|\bar a_{\xi, k, n} \nabla^\perp \bar u_{q, \Gamma}\|_N + \|\bar D_{t, \Gamma} \bar a_{\xi, k, n} \nabla^\perp \bar u_{q, \Gamma}\|_N \\
    && + \|\bar a_{\xi, k, n} \bar D_{t, \Gamma} \nabla^\perp \bar u_{q, \Gamma}\|_N \\ 
    &\lesssim& \mu_{q+1} \delta_{q+1}^{1/2}\delta_q^{1/2} \lambda_q \ell_q^{-N} + \tau_q^{-1} \delta_{q+1}^{1/2}\delta_q^{1/2} \lambda_q \ell_q^{-N} \\
    && + (\delta_q^{1/2}\lambda_q)\delta_{q+1}^{1/2}\delta_q^{1/2} \lambda_q \ell_q^{-N} \\ 
    &\lesssim& \mu_{q+1} \delta_{q+1}^{1/2}\delta_q^{1/2} \lambda_q \ell_q^{-N},
\end{eqnarray*}
where we use the fact that 
\begin{equation*}
    \delta_q^{1/2}\lambda_q < \tau_q^{-1} < \mu_{q+1}.
\end{equation*}
Then, 
\begin{equation*}
    \|T_1\|_{N+\alpha} \lesssim \frac{\mu_{q+1}\delta_q^{1/2}\delta_{q+1}^{1/2}}{\lambda_{q+1}}(\lambda_{q+1}^{N+\alpha} + \ell_{q}^{-N-\alpha}) \lesssim \frac{\mu_{q+1}\delta_q^{1/2}\delta_{q+1}^{1/2}}{\lambda_{q+1}^{1-\alpha}}\lambda_{q+1}^N.
\end{equation*}
The second term $T_2$ can be estimated using Proposition~\ref{CZ_comm}, since $\mathcal{R} \div$ is a sum of Calder\'on-Zygmund operators. So, we have
\begin{align*}
        \|T_2\|_{N+\alpha} &\lec \frac{1}{\lambda_{q+1}} \sup_{\xi,k,n} \big(\|\bar u_{q,\Ga}\|_{1+\al} \|g_{\xi, k, n+1} \bar a_{\xi, k, n} \Psi_\xi (\lambda_{q+1}\tilde \Phi_k) \nabla^\perp \bar u_{q, \Gamma} \|_{N+\al}\\
        &\qquad + \|\bar u_{q,\Ga}\|_{N+1+\al} \|g_{\xi, k, n+1} \bar a_{\xi, k, n} \Psi_\xi (\lambda_{q+1}\tilde \Phi_k) \nabla^\perp \bar u_{q, \Gamma} \|_{\al} \big) \\ 
        & \lesssim \frac{1}{\lambda_{q+1}} \tau_q^{-1} \delta_{q}^{1/2}\delta_{q+1}^{1/2}\lambda_q( \lambda_{q+1}^{N+\alpha} + \ell_q^{-N}\lambda_{q+1}^\alpha) \lesssim \frac{\tau_q^{-1}\delta_q^{1/2}\delta_{q+1}^{1/2}\lambda_q}{\lambda_{q+1}^{1-\alpha}} \lambda_{q+1}^N,
\end{align*}
where for the second inequality we use the fact that 
\begin{equation*}
    \|\bar u_{q, \Gamma}\|_{N+1+\alpha} \lesssim \tau_q^{-1} \ell_q^{-N}, \,\,\, \forall N \geq 0.
\end{equation*}
The claimed bounds, then, follow.
\end{proof}

We now turn to proving the required bounds for the transport error.

\begin{lem} \label{Tranp_err_estim}
The following estimates hold for the transport error:
\begin{equation}
    \|\mathcal{R}(\bar D_{t, \Gamma} w^{(s)}_{q+1})\|_N \lesssim \frac{\delta_{q+1} \lambda_q^{2/3}}{\lambda_{q+1}^{2/3 - 5\alpha}} \lambda_{q+1}^N, \,\,\, \forall N \geq 0,
\end{equation}
\begin{equation}
     \|\bar D_{t,\Ga} \mathcal{R}(\bar D_{t, \Gamma} w^{(s)}_{q+1})\|_N \lesssim \frac{\ell_{t,q}^{-1}\delta_{q+1} \lambda_q^{2/3}}{\lambda_{q+1}^{2/3 - 5\alpha}} \lambda_{q+1}^N, \,\,\, \forall N \geq 0,
\end{equation}
where the implicit constants depend on $\Gamma$, $M$, $\alpha$, and $N$.
\end{lem}

\begin{proof}
    Since $\bar u_{q, \Gamma}$ is divergence-free, it holds that 
    \begin{align*}
        \mathcal{R}(\bar D_{t, \Gamma} w_{q+1}^{(s)}) =& \underbrace{\frac{1}{\lambda_{q+1}} \mathcal{R}\nabla^\perp\bigg( \bar D_{t, \Gamma}\sum_{\xi, k, n} g_{\xi, k, n+1} \bar a_{\xi, k, n} \Psi_\xi(\lambda_{q+1} \tilde \Phi_k)\bigg)}_{T_1} \\
        & - \underbrace{\frac{1}{\lambda_{q+1}} \mathcal{R}\div\bigg(\sum_{\xi, k, n} g_{\xi, k, n+1} \bar a_{\xi, k, n} \Psi_\xi(\lambda_{q+1}\tilde \Phi_k) \nabla^\perp \bar u_{q, \Gamma}\bigg)}_{T_2}.
    \end{align*}
    Note that the second term, $-T_2$, is precisely the Nash error estimated in the previous lemma. For the first term, since $\mathcal{R} \nabla^\perp$ is an operator of Calder\'on-Zygmund type, we have
    \begin{eqnarray*}
        \|T_1\|_{N+\alpha} \lesssim && \frac{1}{\lambda_{q+1}} \sup_{\xi, k, n} \big(\|\Psi_{\xi}(\lambda_{q+1} \tilde \Phi_k)\|_{N+\alpha} \|\bar D_{t, \Gamma}(g_{\xi,k,n+1} \bar a_{\xi, k, n})\|_0 \\
        && + \|\Psi_{\xi}(\lambda_{q+1} \tilde \Phi_k)\|_0 \|\bar D_{t, \Gamma}(g_{\xi,k,n+1} \bar a_{\xi, k, n})\|_{N+\alpha}\big).
    \end{eqnarray*}
    In the above, we have used the fact that $\Psi_\xi(\lambda_{q+1} \tilde \Phi_k)$ is transported by the flow of $\bar u_{q, \Gamma}$. We have also already described the bounds for $\Psi_\xi(\lambda_{q+1} \tilde \Phi_k)$ in the previous lemma. Recalling that $g_{\xi, k, n}$ is $(\mu_{q+1})^{-1}$-periodic and that  $\mu_{q+1} > \tau_q^{-1}$, lemma \ref{a_bar_estim} implies 
    \begin{equation*}
        \|\bar D_{t, \Gamma}(g_{\xi, k, n+1} \bar a_{\xi, k, n}) \|_{N+\alpha} \lesssim \delta_{q+1}^{1/2} \mu_{q+1} \ell_q^{-N-\alpha} \lesssim \delta_{q+1} \lambda_q^{2/3}\lambda_{q+1}^{1/3 + 4\alpha} \ell_q^{-N-\alpha}.
    \end{equation*}
    Then, 
    \begin{equation*}
        \|T_1\|_{N+\alpha} \lesssim \frac{\delta_{q+1} \lambda_q^{2/3}}{\lambda_{q+1}^{2/3 - 4\alpha}}(\lambda_{q+1}^{N+\alpha}+ \ell_q^{-N-\alpha} ) \lesssim \frac{\delta_{q+1} \lambda_q^{2/3}}{\lambda_{q+1}^{2/3 - 5 \alpha}} \lambda_{q+1}^N.
    \end{equation*}
    The inequality 
    \begin{equation*}
        \delta_{q+1}^{1/2} \lambda_q^{2/3} \lambda_{q+1}^{1/3} > \delta_q^{1/2}\lambda_{q}
    \end{equation*}
    shows that $T_2$ obeys better estimates (i.e. those of lemma \ref{Nash_err_estim}) than those obtained for $T_1$ above.

    It now remains to estimate the material derivative of $T_1$. We write
    \begin{align*}
    \bar D_{t,\Ga} T_1 &= \underbrace{\frac{1}{\lambda_{q+1}} \mathcal{R} \na^\perp \bar D_{t,\Ga}^2 \sum_{\xi, k, n} g_{\xi, k, n+1} \bar a_{\xi, k, n} \Psi_\xi (\lambda_{q+1}\tilde \Phi_k)}_{T_{11}}\\
    &\qquad + \underbrace{\frac{1}{\lambda_{q+1}} \left[\bar u_{q,\Ga} \cdot\na , \mathcal{R} \na^\perp \right] \bar D_{t,\Ga} \sum_{\xi, k, n} g_{\xi, k, n+1} \bar a_{\xi, k, n} \Psi_\xi (\lambda_{q+1}\tilde \Phi_k)}_{T_{12}}
\end{align*}
We estimate $T_{11}$ as
\begin{eqnarray*}
        \|T_{11}\|_{N+\alpha} \lesssim && \frac{1}{\lambda_{q+1}} \sup_{\xi, k, n} \big(\|\Psi_{\xi}(\lambda_{q+1} \tilde \Phi_k)\|_{N+\alpha} \|\bar D_{t, \Gamma}^2 (g_{\xi,k,n+1} \bar a_{\xi, k, n})\|_0 \\
        && + \|\Psi_{\xi}(\lambda_{q+1} \tilde \Phi_k)\|_0 \|\bar D_{t, \Gamma}^2 (g_{\xi,k,n+1} \bar a_{\xi, k, n})\|_{N+\alpha}\big).
    \end{eqnarray*}
    As before, we note that lemmas \ref{a_bar_estim} and \ref{high_mat_der} imply 
    \begin{eqnarray*}
        \|\bar D_{t, \Gamma}^2 (g_{\xi, k, n+1} \bar a_{\xi, k, n}) \|_N &\lesssim& |\partial_t^2 g_{\xi, k, n+1}|\|\bar a_{\xi, k, n}\|_N + |\partial_t g_{\xi, k, n+1}| \|\bar D_{t, \Gamma} \bar a_{\xi, k, n}\|_N + \|\bar D_{t, \Gamma}^2 \bar a_{\xi, k, n}\|_N \\
        &\lesssim& 
        \delta_{q+1}^{1/2}  \ell_q^{-N}( \mu_{q+1}^2 + \mu_{q+1}\tau_q^{-1} + \tau_q^{-1} \ell_{t,q}^{-1}) \\
        &\lesssim& \ell_{t,q}^{-1}\mu_{q+1} \delta_{q+1}^{1/2} \ell_q^{-N},
    \end{eqnarray*}
    where we use the fact that 
    \begin{equation*}
        \tau_q^{-1} < \mu_{q+1} < \ell_{t,q}^{-1}, 
    \end{equation*}
    whenever $\alpha > 0$ is sufficiently small. Then, 
    \begin{equation*}
        \|T_{11}\|_{N+\alpha} \lesssim \frac{\delta_{q+1} \lambda_q^{2/3}\ell_{t,q}^{-1}}{\lambda_{q+1}^{2/3 - 4\alpha}}(\lambda_{q+1}^{N+\alpha} + \ell_q^{-N-\alpha}) \lesssim \frac{\ell_{t,q}^{-1}\delta_{q+1} \lambda_q^{2/3}}{\lambda_{q+1}^{2/3 - 5 \alpha}} \lambda_{q+1}^N.
    \end{equation*}
We estimate $T_{22}$ using Proposition~\ref{CZ_comm} as
\begin{align*}
        \|T_{22}\|_{N+\alpha} &\lec \frac{1}{\lambda_{q+1}} \sup_{\xi,k,n} \big(\|\bar u_{q,\Ga}\|_{1+\al} \|\bar D_{t,\Ga} (g_{\xi, k, n+1} \bar a_{\xi, k, n}) \Psi_\xi (\lambda_{q+1}\tilde \Phi_k) \|_{N+\al}\\
        &\qquad + \|\bar u_{q,\Ga}\|_{N+1+\al} \|\bar D_{t,\Ga} (g_{\xi, k, n+1} \bar a_{\xi, k, n}) \Psi_\xi (\lambda_{q+1}\tilde \Phi_k) \|_{\al} \big) \\
        & \lesssim \frac{\tau_q^{-1}\mu_{q+1}\delta_{q+1}^{1/2}}{\lambda_{q+1}}(\lambda_{q+1}^{N+\alpha} + \lambda_{q+1}^\alpha \ell_q^{-N}) \lesssim \frac{\tau_q^{-1}\delta_{q+1} \lambda_q^{2/3}}{\lambda_{q+1}^{2/3 - 5 \alpha}} \lambda_{q+1}^N,
\end{align*}
and the conclusion follows.
\end{proof}

\subsection{Estimates for the oscillation error \texorpdfstring{$R_{q+1, O}$}{rororo}}
Let us begin by rewriting the error as 
\begin{equation*}
    R_{q+1, O} = \underbrace{\mathcal{R}\div(S_{q, \Gamma} + w_{q+1}^{(p)} \otimes w_{q+1}^{(p)})}_{\text{Principal oscillation error}} + \underbrace{\mathcal{R} \div (w_{q+1}^{(p)} \otimes w_{q+1}^{(c)} + w_{q+1}^{(c)}\otimes w_{q+1}^{(p)} + w_{q+1}^{(c)} \otimes w_{q+1}^{(c)})}_{\text{Divergence corrector error}}.
\end{equation*}
The main idea of the Nash step is the cancellation between the error $S_{q, \Gamma}$ and the low frequency part of the quadratic self-interaction of $w_{q+1}^{(p)}$. Before we present this cancellation, let us recall the notation 
\begin{equation*}
    A_{\xi, k, n} = a^2_{\xi, k, n} (\nabla  \Phi_k)^{-1} \xi \otimes \xi (\nabla \Phi_k)^{-T},
\end{equation*}
and analogously denote 
\begin{equation}\label{def.bA}
    \bar A_{\xi, k, n} = \bar a^2_{\xi, k, n} (\nabla \tilde \Phi_k)^{-1} \xi \otimes \xi (\nabla \tilde \Phi_k)^{-T}.
\end{equation}
Then, using the fact that $\{g_{\xi, k, n+1} \bar a_{\xi, k, n}\}_{\xi, k, n}$ have pair-wise disjoint supports, we compute:
\begin{eqnarray*}
    w_{q+1}^{(p)} \otimes w_{q+1}^{(p)}
    & = & \sum_{\xi, k, n} g_{\xi, k, n+1}^2 \bar a_{\xi, k, n}^2 (\nabla \tilde \Phi_k)^{-1} (\mathbb W_{\xi} \otimes \mathbb W_{\xi})(\lambda_{q+1} \tilde \Phi_k) (\nabla \tilde \Phi_k)^{-T} \\ 
    & = & \sum_{\xi, k, n} g_{\xi, k, n+1}^2 \bar a_{\xi, k, n}^2 (\nabla \tilde \Phi_k)^{-1} \xi \otimes \xi (\nabla \tilde \Phi_k)^{-T} \\
    && + \sum_{\xi, k, n} g^2_{\xi, k, n+1} \bar a_{\xi, k, n}^2 (\nabla \tilde \Phi_k)^{-1} (\mathbb{P}_{\neq 0} \mathbb{W}_\xi \otimes \mathbb{W}_\xi)(\la_{q+1} \tilde\Phi_k) (\nabla \tilde \Phi_k)^{-T} \\ 
    &=&  -S_{q, \Gamma}  + \sum_{\xi, k, n} g_{\xi, k, n+1}^2 (\bar A_{\xi,k,n} - A_{\xi, k, n}) \\ 
    && + \sum_{\xi, k, n} g^2_{\xi, k, n+1} \bar a_{\xi, k, n}^2 (\nabla \tilde \Phi_k)^{-1} (\mathbb{P}_{\neq 0} \mathbb{W}_\xi \otimes \mathbb{W}_\xi)(\la_{q+1} \tilde\Phi_k) (\nabla \tilde \Phi_k)^{-T},
\end{eqnarray*}
where $\mathbb P_{\neq 0}$ denotes the Fourier projection onto mean-zero $2$-tensors: 
\begin{equation*}
    \mathbb P_{\neq 0} \mathbb W_\xi \otimes \mathbb W_\xi =\mathbb W_\xi \otimes \mathbb W_\xi - \fint_{\mathbb T_{2\pi} ^2} \mathbb{W}_{\xi} \otimes \mathbb{W}_{\xi}.
\end{equation*}
Note, then, that we can write 
\begin{equation*}
    \mathbb P_{\neq 0} \mathbb W_\xi \otimes \mathbb W_\xi = \Omega_\xi \xi \otimes \xi,
\end{equation*}
where $\Omega_\xi$ is defined by 
\begin{equation*}
    \Omega_\xi(x) = \frac{1}{2}(e^{2i\xi^\perp \cdot x} + e^{-2i\xi^\perp \cdot x}).
\end{equation*}
We can, then, decompose the principal oscillation error as:
\begin{eqnarray*}
    \mathcal{R}\div(S_{q, \Gamma} + w_{q+1}^{(p)} \otimes w_{q+1}^{(p)}) &=& \underbrace{\mathcal{R} \div \bigg(\sum_{\xi, k, n} g^2_{\xi, k, n+1} \Omega_\xi(\lambda_{q+1} \tilde \Phi_k) \bar A_{\xi, k, n}\bigg)}_{\text{High-high-high oscillation error}} \\
    && + \underbrace{\mathcal{R}\div \bigg(\sum_{\xi, k, n} g_{\xi, k, n+1}^2 ( \bar A_{\xi,k,n} - A_{\xi, k, n})\bigg)}_{\text{Flow-mollification error}}. 
\end{eqnarray*}

The flow-mollification error above arises for two reasons. On the one hand, it encapsulates the error related to the mollification along the flow, which, in turn, can be understood as the error resulting from the loss of material derivative. On the other hand, it quantifies the deviation of the flow of the perturbed velocity field $\bar u_{q, \Gamma} = \bar u_q + w^{(t)}_{q+1}$ from the flow of $\bar u_q$.

Before we estimate each of the identified error terms, we collect the following preliminary results. 

\begin{lem} \label{big_a_bar_estim}
The tensors $\bar A_{\xi,k,n}$ defined in~\eqref{def.bA} satisfy the following estimates: 
\begin{equation}
    \|\bar A_{\xi, k, n}\|_N \lesssim \delta_{q+1, n} \lambda_q^N, \,\,\, \forall N \in \{0,1,..., L-3\},
\end{equation}
\begin{equation}
    \|\bar D_{t, \Gamma} \bar A_{\xi, k, n} \|_N  \lesssim \delta_{q+1, n} \tau_q^{-1} \lambda_q^N, \,\,\, \forall N \in \{0,1,..., L-3\},
\end{equation}
\begin{equation}
    \|\bar A_{\xi, k, n}\|_{N+L-3} \lesssim \delta_{q+1, n} \lambda_q^{L-3} \ell_q^{-N},  \,\,\, \forall N \geq 0,
\end{equation}
\begin{equation}
    \|\bar D_{t, \Gamma} \bar A_{\xi, k, n} \|_{N+L-3} \lesssim \delta_{q+1, n} \lambda_q^{L-3} \tau_q^{-1} \ell_q^{-N}, \,\,\, \forall N \geq 0,
\end{equation}
where the implicit constants depend on $\Gamma$, $M$, $\alpha$ and $N$. Moreover, it holds that 
\begin{equation}
    \|\bar A_{\xi, k, n} - A_{\xi,k, n}\|_0 \lesssim  \delta_{q+1,n}\frac{\delta_{q+1}^{1/2} \lambda_q^{1/3}}{\delta_{q}^{1/2} \lambda_{q+1}^{1/3}}.
\end{equation}
\end{lem}

\begin{proof}
    All but the last estimate follow by repeating the proof of corollary \ref{a_cor}, but this time using lemma \ref{a_bar_estim} and corollary \ref{Flow_gam_estim}. For the final estimate, we write 
    \begin{eqnarray*}
        \bar A_{\xi, k, n} - A_{\xi, k, n} = &&  (\bar a_{\xi, k, n}^2 - a_{\xi, k, n}^2) (\nabla \tilde \Phi_k)^{-1}\xi \otimes \xi (\nabla \tilde \Phi_k)^{-T} \\ 
        && + a_{\xi, k, n}^2\big( (\nabla \tilde \Phi_k)^{-1} - (\nabla \Phi_k)^{-1}\big) \xi \otimes \xi (\nabla \tilde \Phi_k)^{-T} \\ 
        && + a_{\xi, k, n}^2 (\nabla \Phi_k)^{-1} \xi \otimes \xi \big ( (\nabla \tilde \Phi_k)^{-T} - (\nabla \Phi_k)^{-T} \big).
    \end{eqnarray*}
    Then, 
    \begin{equation*}
        \|\bar A_{\xi, k, n} - A_{\xi, k, n}\|_0 \lesssim \|\bar a_{\xi, k, n}^2 - a_{\xi, k, n}^2\|_0 + \|a_{\xi, k, n}^2\|_0 \|(\nabla \tilde \Phi_k)^{-1} - (\nabla \Phi_k)^{-1}\|_0.
    \end{equation*}
    In order to estimate the first term, we use the mean value inequality to obtain 
    \begin{eqnarray*}
        \|\bar a_{\xi, k, n}^2 - a_{\xi, k, n}^2\|_0 &\lesssim& \delta_{q+1, n} \left\|\nabla \tilde \Phi_k \frac{\delta_{q+1, n} \I - \bar R_{q,n}}{\delta_{q+1, n}}(\nabla \tilde \Phi_k)^T - \nabla \Phi_k \frac{\delta_{q+1, n} \I - R_{q,n}}{\delta_{q+1, n}}(\nabla \Phi_k)^T \right\|_0 \\ 
        & \lesssim & \|\nabla \tilde \Phi_k - \nabla \Phi_k\|_0 \|\delta_{q+1,n}\I - \bar R_{q, n}\|_0 + \|R_{q, n} - \bar R_{q, n}\|_0 \\
        && + \|\nabla \tilde \Phi_k - \nabla \Phi_k\|_0 \|\delta_{q+1,n}\I - R_{q, n}\|_0.
    \end{eqnarray*}
    By standard mollification estimates,  
    \begin{equation*}
        \|R_{q, n} - \bar R_{q,n}\|_0 \lesssim \|\bar D_{t, \Gamma} R_{q, n}\|_0 \ell_{t,q} \lesssim \delta_{q+1, n} \tau_q^{-1} \ell_{t,q}.
    \end{equation*}
    This fact, together with the estimates of lemmas \ref{a_estim}, \ref{flow_stabil}, \ref{tmoli_estim}, and \ref{a_bar_estim}, implies 
    \begin{eqnarray*}
        \|\bar A_{\xi, k, n} - A_{\xi, k, n}\|_0 &\lesssim & \delta_{q+1, n}\bigg( \tau_q \frac{\delta_{q+1} \lambda_q^2 \ell_q^{-\alpha}}{\mu_{q+1}} + \tau_q^{-1} \ell_{t,q}\bigg) \\ 
        & \lesssim & \delta_{q+1, n} \bigg( \frac{\delta_{q+1}^{1/2} \lambda_q^{1/3}}{\delta_{q}^{1/2}\lambda_{q+1}^{1/3}} + \frac{\lambda_{q}^{2/3}}{\lambda_{q+1}^{2/3}} \lambda_{q+1}^\alpha \bigg).
    \end{eqnarray*}
    The conclusion follows once we note that 
    \begin{equation*}
        \frac{\lambda_{q}^{2/3}}{\lambda_{q+1}^{2/3}} < \frac{\delta_{q+1}^{1/2} \lambda_q^{1/3}}{\delta_{q}^{1/2}\lambda_{q+1}^{1/3}},
    \end{equation*}
    and, thus, for all $\alpha$ sufficiently small in terms of $b$ and $\beta$, we have 
    \begin{equation*}
        \frac{\lambda_{q}^{2/3}}{\lambda_{q+1}^{2/3}} \lambda_{q+1}^\alpha \leq \frac{\delta_{q+1}^{1/2} \lambda_q^{1/3}}{\delta_{q}^{1/2}\lambda_{q+1}^{1/3}}.
    \end{equation*}
\end{proof}

We now turn to obtaining estimates for each of the three identified errors, starting with the high-high-high oscillation error. 

\begin{lem}
The following estimates hold for the high-high-high oscillation error:
\begin{equation} \label{high-high-estim}
    \bigg\|\mathcal{R} \div \bigg(\sum_{\xi, k, n} g^2_{\xi, k, n+1} \Omega_\xi(\lambda_{q+1} \tilde \Phi_k) \bar A_{\xi, k, n}\bigg)\bigg\|_{N} \lesssim \frac{\delta_{q+1} \lambda_q}{\lambda_{q+1}^{1 - 2\alpha}} \lambda_{q+1}^N, \,\,\, \forall N \geq 0,
\end{equation}
\begin{equation}
    \bigg\|\bar D_{t,\Ga} \mathcal{R} \div \bigg(\sum_{\xi, k, n} g^2_{\xi, k, n+1} \Omega_\xi(\lambda_{q+1} \tilde \Phi_k) \bar A_{\xi, k, n}\bigg)\bigg\|_{N} \lesssim \mu_{q+1}\delta_{q+1} \lambda_{q+1}^\alpha \lambda_{q+1}^N, \,\,\, \forall N \geq 0.
\end{equation}
where the implicit constants depend on $\Gamma$, $M$, $\alpha$, $b$, and $N$. 
\end{lem}

\begin{proof}
We first note that 
\begin{equation*}
    \bar A_{\xi, k, n} \nabla\big(\Omega_\xi(\lambda_{q+1} \tilde \Phi_k)\big) = \lambda_{q+1} \bar  a_{\xi,k,n}^2  (\xi \cdot \nabla \Omega_{\xi})(\lambda_{q+1} \tilde \Phi_k) (\nabla \tilde \Phi_k)^{-1}\xi = 0.
\end{equation*}
Consequently, it holds that
\begin{eqnarray*}
    \div \big( \Omega_\xi(\lambda_{q+1} \tilde \Phi_k) \bar A_{\xi, k, n}) = \Omega_{\xi}(\lambda_{q+1} \tilde \Phi_k) \div \bar A_{\xi, k, n}, 
\end{eqnarray*}
and, thus, 
\begin{equation*}
    \bigg\|\mathcal{R} \div \bigg(\sum_{\xi, k, n} g^2_{\xi, k, n+1} \Omega_\xi(\lambda_{q+1} \tilde \Phi_k) \bar A_{\xi, k, n}\bigg)\bigg\|_{N+\alpha} \lesssim \sup_{\xi, k, n}g_{\xi, k, n+1}^2 \big \|\mathcal{R}\big(\Omega_{\xi}(\lambda_{q+1} \tilde \Phi_k) \div \bar A_{\xi, k, n}\big)\big\|_{N+\alpha}.
\end{equation*}
Proposition \ref{prop.inv.div} implies that, for any $\tilde N \geq 1$, 
\begin{eqnarray*}
    \big \|\mathcal{R}\big(\Omega_{\xi}(\lambda_{q+1} \tilde \Phi_k) \div \bar A_{\xi, k, n}\big)\big\|_{\alpha} &\lesssim& \frac{\|\bar A_{\xi, k, n}\|_1}{\lambda_{q+1}^{1-\alpha}} + \frac{\|\bar A_{\xi, k, n}\|_{\tilde  N + 1 + \alpha} + \|\bar A_{\xi, k, n}\|_1 \|\nabla \tilde \Phi\|_{\tilde N + \alpha}}{\lambda_{q+1}^{\tilde N-\alpha}} \\
    &\lesssim & \frac{\delta_{q+1}\lambda_q}{\lambda_{q+1}^{1-\alpha}} + \frac{\delta_{q+1} \lambda_q \ell_q^{-\tilde N - \alpha}}{\lambda_{q+1}^{\tilde N - \alpha}} \\ 
    & \lesssim & \frac{\delta_{q+1}\lambda_q}{\lambda_{q+1}^{1-2\alpha}} \bigg( 1 + \lambda_{q+1}\frac{\ell_q^{- \tilde N}}{\lambda_{q+1}^{\tilde N}}\bigg).
\end{eqnarray*}
with implicit constant depending on $\alpha$ and $\tilde N$. Since 
\begin{equation*}
    \frac{\ell_q^{-\tilde{N}}}{\lambda_{q+1}^{\Tilde{N}}} \leq \bigg( \frac{\lambda_q}{\lambda_{q+1}} \bigg)^{\tilde N /2},
\end{equation*}
we can choose $\tilde N$ sufficiently large depending only on $b$ so that 
\begin{equation*}
    \lambda_{q+1}\frac{\ell_q^{- \tilde N}}{\lambda_{q+1}^{\tilde N}} \leq 1.
\end{equation*}
Therefore, 
\begin{equation*}
    \big \|\mathcal{R}\big(\Omega_{\xi}(\lambda_{q+1} \tilde \Phi_k) \div \bar A_{\xi, k, n}\big)\big\|_{\alpha} \lesssim \frac{\delta_{q+1}\lambda_q}{\lambda_{q+1}^{1 - 2\alpha}}.
\end{equation*}
To obtain the estimates for the higher derivatives, consider $N \geq 1$ and let $\theta$ be a multi-index with $|\theta| = N-1$, and $i \in \{1,2\}$. Then, using the fact that $\partial_i \mathcal{R}$ is of Calder\'on-Zygmund type, we estimate
\begin{eqnarray*}
    \big\|\partial_i \partial^\theta \mathcal{R}\big(\Omega_{\xi}(\lambda_{q+1} \tilde \Phi_k) \div \bar A_{\xi, k, n}\big)\big\|_\alpha &\lesssim& \|\Omega_{\xi}(\lambda_{q+1} \tilde \Phi_k) \div \bar A_{\xi, k, n}\|_{N-1+\alpha} \\ 
    & \lesssim & \|\Omega_{\xi}(\lambda_{q+1} \tilde \Phi_k)\|_{N-1+\alpha} \|\bar A_{\xi, k, n}\|_1 + \|\Omega_{\xi}(\lambda_{q+1} \tilde \Phi_k)\|_0 \|\bar A_{\xi, k, n}\|_{N+\alpha} \\
    & \lesssim & \delta_{q+1}\lambda_q(\lambda_{q+1}^{N-1+\alpha} + \ell_q^{-N + 1 - \alpha}) \\ 
    & \lesssim & \frac{\delta_{q+1}\lambda_q}{\lambda_{q+1}^{1-\alpha}}\lambda_{q+1}^{N},
\end{eqnarray*}
and \eqref{high-high-estim} follows.

To obtain estimates on the material derivative, we write
    \begin{align*}
        \bar D_{t,\Ga} \mathcal{R} \div \bigg(\sum_{\xi, k, n} g^2_{\xi, k, n+1} \Omega_\xi(\lambda_{q+1} \tilde \Phi_k) \bar A_{\xi, k, n}\bigg) &= \underbrace{\mathcal{R} \div \sum_{\xi, k, n} \Omega_\xi(\lambda_{q+1} \tilde \Phi_k) \bar D_{t,\Ga} \left(g^2_{\xi, k, n+1} \bar A_{\xi, k, n}\right) }_{T_1}\\
        &\qquad + \underbrace{[\bar u_{q,\Ga}\cdot\na, \mathcal{R \div}] \sum_{\xi, k, n} g^2_{\xi, k, n+1}\Omega_\xi(\lambda_{q+1} \tilde \Phi_k)\bar A_{\xi, k, n} }_{T_2}.
    \end{align*}
    For the first term, we have
    \begin{equation*}
        \|T_1\|_{N+\alpha} \lesssim \sup_{\xi, k, n} \big(\|\Omega_\xi (\lambda_{q+1} \tilde \Phi_k)\|_{N+\alpha} \|\bar D_{t, \Gamma}(g_{\xi, k, n+1}^2 \bar A_{\xi, k, n})\|_0 + \|\Omega_\xi (\lambda_{q+1} \tilde \Phi_k)\|_{0} \|\bar D_{t, \Gamma}(g_{\xi, k, n+1}^2 \bar A_{\xi, k, n})\|_{N+\alpha}\big).
    \end{equation*}
    Using lemma \ref{big_a_bar_estim} and arguing as before, we find 
    \begin{equation*}
        \|T_1\|_{N+\alpha} \lesssim \mu_{q+1}\delta_{q+1}(\lambda_{q+1}^{N+\alpha} + \ell_q^{-N-\alpha}) \lesssim \mu_{q+1}\delta_{q+1} \lambda_{q+1}^\alpha \lambda_{q+1}^N.
    \end{equation*}
    For the second term, we use proposition \ref{CZ_comm}, to obtain 
    \begin{eqnarray*}
        \|T_{2}\|_{N+\alpha} &\lesssim& \sup_{\xi, k, n} \big(\|\bar u_{q, \Gamma}\|_{1+\alpha} \|g_{\xi, k, n+1}^2 \Omega_\xi(\lambda_{q+1} \tilde \Phi_k) \bar A_{\xi, k, n} \|_{N+\alpha} \\
        && + \|\bar u_{q, \Gamma}\|_{N + 1+\alpha} \|g_{\xi, k, n+1}^2 \Omega_\xi(\lambda_{q+1} \tilde \Phi_k) \bar A_{\xi, k, n} \|_{\alpha} \big) \\ 
        &\lesssim& \tau_q^{-1}\delta_{q+1}(\lambda_{q+1}^{N+\alpha} + \lambda_{q+1}^\alpha \ell_q^{-N}) \lesssim \tau_q^{-1}\delta_{q+1} \lambda_{q+1}^\alpha \lambda_{q+1}^N,
    \end{eqnarray*}
    and the conclusion follows.
\end{proof}

Next, we aim to obtain estimates for the flow-mollification error. 

\begin{lem}
The following estimates hold for the flow-mollification error:
\begin{equation}
    \bigg\|\mathcal{R} \div \bigg(\sum_{\xi, k, n} g_{\xi, k, n+1}^2 (\bar A_{\xi, k, n} - A_{\xi, k, n})\bigg)\bigg\|_N \lesssim \delta_{q+1}\frac{\delta_{q+1}^{1/2} \lambda_{q}^{1/3}}{\delta_{q}^{1/2}\lambda_{q+1}^{1/3}} \lambda_{q+1}^\alpha \lambda_{q+1}^N, \,\,\, \forall N \geq 0.
\end{equation}
\begin{equation}
    \bigg\|\bar D_{t, \Gamma}\mathcal{R} \div \bigg(\sum_{\xi, k, n} g_{\xi, k, n+1}^2 (\bar A_{\xi, k, n} - A_{\xi, k, n})\bigg)\bigg\|_{N} \lesssim \mu_{q+1}\delta_{q+1} \lambda_{q+1}^\alpha \lambda_{q+1}^{N}, \,\,\, \forall N \geq 0
\end{equation}
where the implicit constants depend on $\Gamma$, $M$, $\alpha$ and $N$. 
\end{lem}

\begin{proof}
    Since $\mathcal{R} \div$ is of Calder\'on-Zygmund type, we have 
    \begin{equation*}
        \bigg\|\mathcal{R} \div \bigg(\sum_{\xi, k, n} g_{\xi, k, n+1}^2 (\bar A_{\xi, k, n} - A_{\xi, k, n})\bigg)\bigg\|_{N} \lesssim \sup_{\xi, k, n}\|\bar A_{\xi, k, n} - A_{\xi, k, n}\|_{N+\alpha}.
    \end{equation*}
    Note that 
    \begin{equation*}
        \|\bar A_{\xi, k, n} - A_{\xi, k, n}\|_1 \leq \|\bar A_{\xi, k, n}\|_1 + \| A_{\xi, k, n}\|_1 \lesssim \delta_{q+1}\lambda_q. 
    \end{equation*}
    Interpolating this with the result of lemma \ref{big_a_bar_estim}, we find 
    \begin{equation*}
        \|\bar A_{\xi, k, n} - A_{\xi, k, n}\|_\alpha \lesssim \delta_{q+1} \bigg(\frac{\delta_{q+1}^{1/2}\lambda_{q}^{1/3}}{\delta_q^{1/2}\lambda_{q+1}^{1/3}}\bigg)^{1-\alpha} \lambda_q^\alpha \lesssim \delta_{q+1} \frac{\delta_{q+1}^{1/2}\lambda_{q}^{1/3}}{\delta_q^{1/2}\lambda_{q+1}^{1/3}} \lambda_{q+1}^\alpha,
    \end{equation*}
    where for the last inequality we have used the fact that 
    \begin{equation*}
        \bigg(\frac{\delta_{q}}{\delta_{q+1}}\bigg)^{1/2} \lambda_q^{2/3}\lambda_{q+1}^{1/3} \leq \lambda_{q+1}.
    \end{equation*}
    Moreover, using corollary \ref{a_cor} and lemma \ref{big_a_bar_estim}, we have that for all $N \geq 0$,
    \begin{equation*}
        \|\bar A_{\xi, k, n} - A_{\xi, k, n}\|_{N+1+\alpha} \lesssim \delta_{q+1} \lambda_q \lambda_{q+1}^{N+\alpha},
    \end{equation*}
    where, as above, we have simply used the triangle inequality. The claimed estimates follow once we notice that the parametric inequality above rearranges as 
    \begin{equation*}
        \delta_{q+1} \lambda_q \leq \delta_{q+1} \frac{\delta_{q+1}^{1/2}\lambda_q^{1/3}}{\delta_{q}^{1/2}\lambda_{q+1}^{1/3}}\lambda_{q+1}.
    \end{equation*}

    For the material derivative estimates, we have
    \begin{eqnarray*}
        \bar D_{t, \Gamma} \mathcal{R} \div \bigg(\sum_{\xi, k, n} g_{\xi, k, n+1}^2 (\bar A_{\xi, k, n} - A_{\xi, k, n})\bigg) &=&\underbrace{\mathcal{R} \div \bar D_{t, \Gamma}\bigg(\sum_{\xi, k, n} g_{\xi, k, n+1}^2 (\bar A_{\xi, k, n} - A_{\xi, k, n})\bigg)}_{T_1}\\
        && + \underbrace{[\bar u_{q, \Gamma} \cdot \nabla, \mathcal{R} \div] \bigg(\sum_{\xi, k, n} g_{\xi, k, n+1}^2 (\bar A_{\xi, k, n} - A_{\xi, k, n})\bigg)}_{T_2}.
    \end{eqnarray*}
    As argued before, we obtain 
    \begin{equation*}
        \|T_1\|_{N+\alpha} \lesssim \mu_{q+1}\delta_{q+1} \lambda_{q+1}^{\alpha} \lambda_{q+1}^N,
    \end{equation*}
    and 
    \begin{equation*}
      \|T_2\|_{N+\alpha} \lesssim \tau_q^{-1}\delta_{q+1} \lambda_{q+1}^{\alpha} \lambda_{q+1}^N,
    \end{equation*}
    where for the estimates involving the difference $\bar A_{\xi, k, n} - A_{\xi, k, n}$, we simply use the triangle inequality and the results of lemma \ref{big_a_bar_estim} and corollary \ref{a_cor}.
\end{proof}

Finally, we obtain the estimates for the divergence corrector error. 

\begin{lem} \label{div_corr_estim}
    The following estimates hold for the divergence corrector error:
\begin{align}
    \left\|\mathcal{R} \div (w_{q+1}^{(p)} \otimes w_{q+1}^{(c)} + w_{q+1}^{(c)}\otimes w_{q+1}^{(p)} + w_{q+1}^{(c)} \otimes w_{q+1}^{(c)})\right\|_{N} &\lesssim \frac{\delta_{q+1} \lambda_q}{\lambda_{q+1}^{1-\alpha}} \lambda_{q+1}^N, \,\,\,\ \forall N \geq 0, \\
    \left\|\bar D_{t,\Ga} \mathcal{R} \div (w_{q+1}^{(p)} \otimes w_{q+1}^{(c)} + w_{q+1}^{(c)}\otimes w_{q+1}^{(p)} + w_{q+1}^{(c)} \otimes w_{q+1}^{(c)})\right\|_{N} &\lesssim \mu_{q+1} \frac{\delta_{q+1} \lambda_q}{\lambda_{q+1}^{1-\alpha}} \lambda_{q+1}^N, \,\,\, \forall N \geq 0,
\end{align}
where the implicit constants depend on $\Gamma$, $M$, $\alpha$, and $N$. 
\end{lem}

\begin{proof}
Let us first collect estimates for $w_{q+1}^{(c)}$ and $w_{q+1}^{(p)}$. Directly from the definition of $w_{q+1}^{(c)}$ in \eqref{w.c} we calculate
\begin{align*}
    \|w_{q+1}^{(c)}\|_N &\lesssim \frac{1}{\lambda_{q+1}} \sup_{\xi, k, n} |g_{\xi, k, n+1}| \|\na^\perp \bar a_{\xi, k, n} \Psi_\xi(\lambda_{q+1} \tilde \Phi_k)\|_{N} \\ 
    & \lesssim \frac{1}{\lambda_{q+1}} \sup_{\xi, k, n} |g_{\xi, k, n+1}| \big(\|\Psi_\xi(\lambda_{q+1} \tilde \Phi_k)\|_{N} \|\na^\perp \bar a_{\xi, k, n}\|_0  + \|\Psi_\xi(\lambda_{q+1} \tilde \Phi_k)\|_0 \|\na^\perp \bar a_{\xi, k, n}\|_{N}\big)\\
    & \lec \frac{\la_q}{\la_{q+1}} \de_{q+1}^{1/2} \la_{q+1}^N,
\end{align*}
where for the last inequality we have used lemma \ref{a_bar_estim} and argued as in the proof of Lemma~\ref{velo_estim_fin}. Similarly, 
\begin{eqnarray*}
    \|w_{q+1}^{(p)}\|_N &\lesssim& \frac{1}{\lambda_{q+1}} \sup_{\xi, k, n} |g_{\xi, k, n+1}| \| \bar a_{\xi, k, n} \nabla^\perp (\Psi_\xi(\lambda_{q+1} \tilde \Phi_k))\|_N \\ 
    & \lesssim & \frac{1}{\lambda_{q+1}} \sup_{\xi, k, n} |g_{\xi, k, n+1}| \big( \|\Psi_\xi(\lambda_{q+1} \tilde \Phi_k)\|_{N+1} \|\bar a_{\xi, k, n}\|_0 + \|\Psi_\xi(\lambda_{q+1} \tilde \Phi_k)\|_{1} \|\bar a_{\xi, k, n}\|_N \big) \\ 
    & \lesssim & \delta_{q+1}^{1/2} \lambda_{q+1}^N.  
\end{eqnarray*}

For the material derivative estimates, note first that 
\begin{equation*}
    \bar D_{t, \Gamma} \nabla^\perp \bar a_{\xi, k, n} = \nabla^\perp \bar D_{t, \Gamma} \bar a_{\xi, k, n} - \nabla^\perp \bar u_{q, \Gamma} \nabla \bar a_{\xi, k, n},
\end{equation*}
and, thus, using lemma \ref{a_bar_estim} and corollary \ref{Gamma_velo_estim},
\begin{eqnarray*}
    \|\bar D_{t, \Gamma} \nabla^\perp \bar a_{\xi, k, n}\|_N &\lesssim& \|\bar D_{t, \Gamma} \bar a_{\xi, k, n}\|_{N+1} + \|\bar u_{q, \Gamma}\|_{N+1} \|\bar a_{\xi, k, n}\|_1 + \|\bar u_{q, \Gamma}\|_1 \|\bar a_{\xi, k, n}\|_{N+1} \\ 
    & \lesssim & \delta_{q+1, n}^{1/2}\tau_q^{-1} \lambda_q \ell_q^{-N} + \delta_{q+1, n} \delta_q^{1/2} \lambda_q^2 \ell_q^{-N} \lesssim \delta_{q+1, n}^{1/2} \tau_q^{-1} \lambda_q \ell_q^{-N}.
\end{eqnarray*}
We have, then, 
\begin{align*}
    \|\bar D_{t,\Ga} w_{q+1}^{(c)}\|_N &\lesssim \frac{1}{\lambda_{q+1}} \sup_{\xi, k, n} \|\Psi_\xi(\lambda_{q+1} \tilde \Phi_k) \bar D_{t,\Ga}( g_{\xi, k, n+1} \na^\perp \bar a_{\xi, k, n})\|_{N} \\ 
    & \lesssim  \frac{1}{\lambda_{q+1}} \sup_{\xi, k, n} |\partial_t g_{\xi, k, n+1}| \big(\|\Psi_\xi(\lambda_{q+1} \tilde \Phi_k)\|_{N} \|\na^\perp \bar a_{\xi, k, n}\|_0  + \|\Psi_\xi(\lambda_{q+1} \tilde \Phi_k)\|_0 \|\na^\perp \bar a_{\xi, k, n}\|_{N}\big)\\
    &\qquad + \frac{1}{\lambda_{q+1}} \sup_{\xi, k, n} |g_{\xi, k, n+1}| \big(\|\Psi_\xi(\lambda_{q+1} \tilde \Phi_k)\|_{N} \|\bar D_{t,\Ga}\na^\perp\bar a_{\xi, k, n}\|_0  + \|\Psi_\xi(\lambda_{q+1} \tilde \Phi_k)\|_0 \|\bar D_{t,\Ga}\na^\perp \bar a_{\xi, k, n}\|_{N}\big)\\
    &\lec \frac{\mu_{q+1}}{\la_{q+1}} \de_{q+1}^{1/2} \la_q \la_{q+1}^N + \frac{\tau_q^{-1}}{\la_{q+1}} \de_{q+1}^{1/2} \la_q \la_{q+1}^N \lesssim \frac{\mu_{q+1} \delta_{q+1}^{1/2} \lambda_q}{\lambda_{q+1}} \lambda_{q+1}^N,
\end{align*}
and, similarly, 
\begin{eqnarray*}
    \|\bar D_{t, \Gamma} w_{q+1}^{(p)}\|_N &\lesssim & \frac{1}{\lambda_{q+1}} \big \|\bar D_{t, \Gamma} \big(g_{\xi, k, n+1} \bar a_{\xi, k, n} \nabla^\perp(\Psi_\xi(\lambda_{q+1} \tilde \Phi_k))\big)\big\|_N \\ 
    &\lesssim& \frac{1}{\lambda_{q+1}} \sup_{\xi, k, n} |\partial_t g_{\xi, k, n+1}| \big( \|\Psi_\xi(\lambda_{q+1} \tilde \Phi_k)\|_{N+1} \|\bar a_{\xi, k, n}\|_0 + \|\Psi_\xi(\lambda_{q+1} \tilde \Phi_k)\|_{1} \|\bar a_{\xi, k, n}\|_N \big) \\ 
    && + \frac{1}{\lambda_{q+1}} \sup_{\xi, k, n} |g_{\xi, k, n+1}| \big(  \|\Psi_\xi(\lambda_{q+1} \tilde \Phi_k)\|_{N+1} \|\bar D_{t, \Gamma}\bar a_{\xi, k, n}\|_0 + \|\Psi_\xi(\lambda_{q+1} \tilde \Phi_k)\|_{1} \|\bar D_{t, \Gamma} \bar a_{\xi, k, n}\|_N \big) \\ 
    && + \frac{1}{\lambda_{q+1}} \sup_{\xi, k, n} |g_{\xi, k, n+1}|\big( \|\bar D_{t, \Gamma} \nabla^\perp(\Psi_\xi(\lambda_{q+1} \tilde \Phi_k))\|_{N} \|\bar a_{\xi, k, n}\|_0 + \|\bar D_{t, \Gamma} \nabla^\perp(\Psi_\xi(\lambda_{q+1} \tilde \Phi_k))\|_{0} \|\bar a_{\xi, k, n}\|_N \big) \\ 
    &\lesssim& \frac{1}{\lambda_{q+1}} \delta_{q+1}^{1/2} \lambda_{q+1}^{N+1} ( \mu_{q+1} + \tau_q^{-1} + \delta_q^{1/2} \lambda_q) \lesssim \mu_{q+1} \delta_{q+1}^{1/2} \lambda_{q+1}^N,
\end{eqnarray*}
where we have used that 
\begin{equation*}
    \|\bar D_{t, \Gamma} \nabla^\perp(\Psi_\xi(\lambda_{q+1} \tilde \Phi_k))\|_N = \|\nabla^\perp \bar u_{q, \Gamma} \nabla (\Psi_\xi(\lambda_{q+1} \tilde \Phi_k))\|_N \lesssim \delta_q^{1/2}\lambda_q \lambda_{q+1}^{N+1}.
\end{equation*}

We can now denote 
\begin{equation*}
    T_1 = \mathcal{R} \div (w_{q+1}^{(p)} \otimes w_{q+1}^{(c)} + w_{q+1}^{(c)}\otimes w_{q+1}^{(p)} + w_{q+1}^{(c)} \otimes w_{q+1}^{(c)}),
\end{equation*}
\begin{equation*}
    T_2 = \bar D_{t,\Ga} \mathcal{R} \div (w_{q+1}^{(p)} \otimes w_{q+1}^{(c)} + w_{q+1}^{(c)}\otimes w_{q+1}^{(p)} + w_{q+1}^{(c)} \otimes w_{q+1}^{(c)}),
\end{equation*}
and estimate 
\begin{eqnarray*}
    \left\|T_1\right\|_{N+\alpha} &\lesssim & \|w_{q+1}^{(p)}\|_{N+\alpha}\|w_{q+1}^{(c)}\|_0 + \|w_{q+1}^{(p)}\|_0 \|w_{q+1}^{(c)}\|_{N+\alpha}  + \|w_{q+1}^{(c)}\|_{N+\alpha} \|w_{q+1}^{(c)}\|_0 \\ 
    & \lesssim & \frac{\delta_{q+1}\lambda_q}{\lambda_{q+1}^{1-\alpha}} \lambda_{q+1}^N.
\end{eqnarray*}
For $T_2$, we write 
\begin{eqnarray*}
    T_2 &=& \underbrace{\mathcal R \div \bar D_{t, \Gamma} (w_{q+1}^{(p)} \otimes w_{q+1}^{(c)} + w_{q+1}^{(c)}\otimes w_{q+1}^{(p)} + w_{q+1}^{(c)} \otimes w_{q+1}^{(c)})}_{T_{21}} \\
    && + \underbrace{[\bar u_{q, \Gamma} \cdot \nabla, \mathcal{R} \div](w_{q+1}^{(p)} \otimes w_{q+1}^{(c)} + w_{q+1}^{(c)}\otimes w_{q+1}^{(p)} + w_{q+1}^{(c)} \otimes w_{q+1}^{(c)})}_{T_{22}}.
\end{eqnarray*}
Then, 
\begin{eqnarray*}
    \|T_{21}\|_{N+\alpha} &\lesssim& \|\bar D_{t, \Gamma} w_{q+1}^{(p)}\|_{N+\alpha} \|w_{q+1}^{(c)}\|_0 + \|\bar D_{t, \Gamma} w_{q+1}^{(p)}\|_0 \|w_{q+1}^{(c)}\|_{N+\alpha} + \|w_{q+1}^{(p)}\|_{N+\alpha} \|\bar D_{t, \Gamma}w_{q+1}^{(c)}\|_0 \\
    && + \|w_{q+1}^{(p)}\|_{0} \|\bar D_{t, \Gamma}w_{q+1}^{(c)}\|_{N+\alpha} +  \|\bar D_{t, \Gamma}w_{q+1}^{(c)}\|_{N+\alpha} \|w_{q+1}^{(c)}\|_0 + \|\bar D_{t, \Gamma}w_{q+1}^{(c)}\|_{0} \|w_{q+1}^{(c)}\|_{N+\alpha} \\ 
    & \lesssim & \mu_{q+1} \frac{\delta_{q+1} \lambda_q}{\lambda_{q+1}^{1-\alpha}} \lambda_{q+1}^N,
\end{eqnarray*}
and, by proposition \ref{CZ_comm}, 
\begin{eqnarray*}
    \|T_{22}\|_{N+\alpha} &\lesssim& \|\bar u_{q, \Gamma}\|_{1+\alpha} \|w_{q+1}^{(p)} \otimes w_{q+1}^{(c)} + w_{q+1}^{(c)} \otimes w_{q+1}^{(p)} + w_{q+1}^{(c)} \otimes w_{q+1}^{(c)}\|_{N+\alpha} \\
    && + \|\bar u_{q, \Gamma}\|_{N+1+\alpha} \|w_{q+1}^{(p)} \otimes w_{q+1}^{(c)} + w_{q+1}^{(c)} \otimes w_{q+1}^{(p)} + w_{q+1}^{(c)} \otimes w_{q+1}^{(c)}\|_{\alpha} \\ 
    & \lesssim&  \tau_q^{-1} \frac{\delta_{q+1} \lambda_q}{\lambda_{q+1}^{1-\alpha}} \lambda_{q+1}^N,
\end{eqnarray*}
and the conclusion follows. 
\end{proof}

\subsection{Estimates for the residual error \texorpdfstring{$R_{q+1, R}$}{rrrrr}}\label{err.residual} 

Recalling the expression \eqref{big_P_err} for $P_{q+1, \Gamma}$, we write 
\begin{equation*}
    R_{q+1, R} = \underbrace{R_{q, \Gamma}}_{\text{Gluing error}} + \underbrace{w_{q+1}^{(t)} \mathring \otimes w_{q+1}^{(t)}}_{\text{Newton error}} + \underbrace{w_{q+1} \mathring \otimes (u_q - \bar u_q) + (u_q - \bar u_q) \mathring \otimes w_{q+1} + R_{q} - R_{q, 0}}_{\text{Spatial mollification error}}.
\end{equation*}
We once again embark on estimating the three identified errors separately. 

\begin{lem}
    The following estimates hold for the gluing error:
    \begin{equation}
        \|R_{q, \Gamma}\|_N \lesssim \frac{\delta_{q+1} \lambda_q}{\lambda_{q+1}} \lambda_{q+1}^N, \,\,\, \forall N \geq 0,
    \end{equation}
    \begin{equation}
        \|\bar D_{t, \Gamma} R_{q, \Gamma}\|_N \lesssim \tau_q^{-1} \frac{\delta_{q+1} \lambda_q}{\lambda_{q+1}} \lambda_{q+1}^N, \,\,\, \forall N \geq 0,
    \end{equation}
    where the implicit constants depend on $\Gamma$, $M$, $\alpha$, and $N$.
\end{lem}

\begin{proof}
    The first estimate follows immediately from proposition \ref{NewIter}, once we notice that 
    \begin{equation*}
        \delta_{q+1, \Gamma} = \delta_{q+1} \bigg(\frac{\lambda_q}{\lambda_{q+1}}\bigg)^{\Gamma (1/3 - \beta)} \leq \delta_{q+1} \frac{\lambda_q}{\lambda_{q+1}}. 
    \end{equation*}
    For the second one, we write 
    \begin{equation*}
        \bar D_{t, \Gamma} R_{q, \Gamma} = \bar D_t R_{q, \Gamma} + w_{q+1}^{(t)}\cdot \nabla R_{q, \Gamma}.
    \end{equation*}
    By proposition \ref{NewIter} and lemma \ref{w_t_estim}, we have
    \begin{equation*}
        \|w_{q+1}^{(t)} \cdot \nabla R_{q, \Gamma}\|_N \lesssim \|w_{q+1}^{(t)}\|_N \|R_{q, \Gamma}\|_1 + \|w_{q+1}^{(t)}\|_0\|R_{q, \Gamma}\|_{N+1} \lesssim \frac{\delta_{q+1}\lambda_q \ell_q^{-\alpha}}{\mu_{q+1}} \frac{\delta_{q+1}\lambda_q^2}{\lambda_{q+1}} \lambda_{q+1}^N.
    \end{equation*}
    The wanted estimates follow once we notice that 
    \begin{equation*}
        \frac{\delta_{q+1}\lambda_q^2 \ell_q^{-\alpha}}{\mu_{q+1}} \leq \delta_{q+1}^{1/2} \lambda_q  \bigg(\frac{\lambda_q}{\lambda_{q+1}}\bigg)^{1/3} \leq \tau_q^{-1}.
    \end{equation*}
\end{proof}

\begin{lem} \label{Newt_err_estim}
The following estimates hold for the Newton error: 
\begin{equation}
    \|w_{q+1}^{(t)} \mathring \otimes w_{q+1}^{(t)}\|_N \lesssim \frac{\delta_{q+1} \lambda_q^{2/3}}{\lambda_{q+1}^{2/3}} \lambda_{q+1}^N, \,\,\, \forall N \geq 0,
\end{equation}
\begin{equation}
    \|\bar D_{t, \Gamma}(w_{q+1}^{(t)} \mathring \otimes w_{q+1}^{(t)})\|_N \lesssim \mu_{q+1} \frac{\delta_{q+1} \lambda_q^{2/3}}{\lambda_{q+1}^{2/3}} \lambda_{q+1}^N, \,\,\, \forall N \geq 0,
\end{equation}
where the implicit constants depend on $\Gamma$, $M$, $\alpha$, and $N$.
\end{lem}

\begin{proof}
    For the first bound we compute using lemma \ref{w_t_estim} and the definition of $\mu_{q+1}$: 
    \begin{equation*}
        \|w_{q+1}^{(t)} \mathring \otimes w_{q+1}^{(t)}\|_N \lesssim \|w_{q+1}^{(t)}\|_N \|w_{q+1}^{(t)}\|_0 \lesssim \bigg(\frac{\delta_{q+1}\lambda_q \ell_q^{-\alpha}}{\mu_{q+1}}\bigg)^2 \lambda_{q+1}^N \lesssim \delta_{q+1} \bigg( \frac{\lambda_q}{\lambda_{q+1}}\bigg)^{2/3} \lambda_{q+1}^N,
    \end{equation*}
    as wanted. In fact, since $L \geq 4$, we have the better estimate 
    \begin{equation*}
        \|w_{q+1}^{(t)} \mathring \otimes w_{q+1}^{(t)}\|_{N+1} \lesssim \frac{\delta_{q+1} \lambda_q^{2/3}}{\lambda_{q+1}^{2/3}} \lambda_q \lambda_{q+1}^N, \,\,\, \forall N \geq 0,
    \end{equation*}
    which we will use in the following calculations.

    We have 
    \begin{eqnarray*}
        \|\bar D_{t, \Gamma}(w_{q+1}^{(t)} \mathring \otimes w_{q+1}^{(t)})\|_N &\lesssim& \|\bar D_{t}(w_{q+1}^{(t)} \mathring \otimes w_{q+1}^{(t)})\|_N + \|w_{q+1}^{(t)} \cdot \nabla ( w_{q+1}^{(t)} \mathring \otimes w_{q+1}^{(t)})\|_N \\ 
        &\lesssim& \|\bar D_{t} w_{q+1}^{(t)}\|_N \|w_{q+1}^{(t)}\|_0 + \|\bar D_{t} w_{q+1}^{(t)}\|_0 \|w_{q+1}^{(t)}\|_N \\ 
        && + \|w_{q+1}^{(t)}\|_N \|w_{q+1}^{(t)} \mathring\otimes w_{q+1}^{(t)}\|_1 + \|w_{q+1}^{(t)}\|_0 \|w_{q+1}^{(t)} \mathring\otimes w_{q+1}^{(t)}\|_{N+1}\\ 
        & \lesssim & \mu_{q+1} \frac{\delta_{q+1} \lambda_q^{2/3}}{\lambda_{q+1}^{2/3}} \lambda_{q+1}^N + \frac{\delta_{q+1} \lambda_q \ell_q^{-\alpha}}{\mu_{q+1}} \frac{\delta_{q+1} \lambda_q^{2/3}}{\lambda_{q+1}^{2/3}} \lambda_q \lambda_{q+1}^N \\ 
        & \lesssim & \mu_{q+1} \frac{\delta_{q+1} \lambda_q^{2/3}}{\lambda_{q+1}^{2/3}} \lambda_{q+1}^N,
    \end{eqnarray*}
    where we have used that 
    \begin{equation*}
        \frac{\delta_{q+1} \lambda_q^2 \ell_q^{-\alpha}}{\mu_{q+1}} < \tau_q^{-1} < \mu_{q+1}.
    \end{equation*}
\end{proof}

We, finally, prove the required bounds for the spatial mollification error.

\begin{lem} \label{spatial_moli_estim_err_fin}
    The following estimates hold for the spatial mollification error:
\begin{equation} \label{spat_moli_estim_1}
    \left\|w_{q+1} \mathring\otimes (u_q - \bar u_q) + (u_q - \bar u_q) \mathring\otimes w_{q+1}\right\|_{N} \lesssim \frac{\delta_{q+1}^{1/2} \delta_q^{1/2} \lambda_q}{\lambda_{q+1}} \lambda_{q+1}^N, \,\,\, \forall N \in \{0,1,..., L\}
\end{equation}
\begin{eqnarray} \label{spat_moli_estim_2}
    \left\|\bar D_{t,\Ga} \left(w_{q+1} \mathring\otimes (u_q - \bar u_q) + (u_q - \bar u_q) \mathring\otimes w_{q+1}\right)\right\|_{N} &\lesssim& \delta_{q+1}^{1/2} \lambda_q^{1/3}\lambda_{q+1}^{2/3} \frac{\delta_{q+1}^{1/2} \delta_q^{1/2} \lambda_q}{\lambda_{q+1}} \lambda_{q+1}^N, \\ 
    && \quad \quad \quad \forall N \in \{0,1,...,L-1\}, \nonumber
\end{eqnarray}
\begin{eqnarray}   
\label{spat_moli_estim_3}
    \|R_q - R_{q, 0}\|_N \lesssim \frac{\delta_{q+1} \lambda_q}{\lambda_{q+1}} \lambda_{q+1}^N, \,\,\, \forall N \in \{0,1,...,L\},
\end{eqnarray} 
\begin{equation} \label{spat_moli_estim_4}
    \|\bar D_{t, \Gamma}(R_{q} - R_{q, 0})\|_N \lesssim \delta_{q+1} \delta_q^{1/2} \lambda_q \lambda_{q+1}^N, \,\,\, \forall N \in \{0,1,...,L-1\}, 
\end{equation}
where the implicit constants depend on $\Gamma$, $M$, $\alpha$, and $N$. 
\end{lem}
\begin{proof}
    By proposition \ref{moli}, we have 
    \begin{equation*}
        \|u_q - \bar u_q\|_0 \lesssim \delta_q^{1/2} \lambda_q^2 \ell_q^2 \lesssim \delta_q^{1/2} \frac{\lambda_q}{\lambda_{q+1}}.
    \end{equation*}
    Moreover, for $N \in \{1,2,...,L\}$, the inductive assumptions on $ u_q$ imply 
    \begin{equation*}
        \|u_q - \bar u_q\|_N \lesssim \delta_{q}^{1/2} \frac{\lambda_q}{\lambda_{q+1}} \lambda_{q+1}^{N}.
    \end{equation*}
    Then, using the estimates of lemma \ref{velo_estim_fin}, we conclude 
    \begin{eqnarray*}
        \left\|w_{q+1} \mathring\otimes (u_q - \bar u_q) + (u_q - \bar u_q) \mathring\otimes w_{q+1}\right\|_{N} &\lesssim& \|w_{q+1}\|_N \|u_q - \bar u_q\|_0 + \|w_{q+1}\|_0 \|u_q - \bar u_q\|_N \\
        &\lesssim&\frac{\delta_{q+1}^{1/2} \delta_q^{1/2}\lambda_q}{\lambda_{q+1}} \lambda_{q+1}^N,
    \end{eqnarray*}
    and \eqref{spat_moli_estim_1} is proven. 

    For \eqref{spat_moli_estim_2}, we have on the one hand 
    \begin{equation*}
        \|\bar D_{t, \Gamma} w_{q+1}^{(s)}\|_N \lesssim \mu_{q+1} \delta_{q+1}^{1/2} \lambda_{q+1}^N,
    \end{equation*}
    as seen in the proof of lemma \ref{div_corr_estim}, and 
    \begin{equation*}
        \|\bar D_{t, \Gamma} w_{q+1}^{(t)}\|_N \lesssim \mu_{q+1} \frac{\delta_{q+1}^{1/2}\lambda_q^{1/3}}{\lambda_{q+1}^{1/3}} \lambda_{q+1}^N,
    \end{equation*}
    by arguments similar to the ones given in the proof of lemma \ref{Newt_err_estim}. On the other hand, 
    \begin{equation*}
        \bar D_{t, \Gamma} (u_q - \bar u_q) = w_{q+1}^{(t)} \cdot \nabla (u_q - \bar u_q) + \bar D_t (u_q - \bar u_q),
    \end{equation*}
    where 
    \begin{equation*}
        \|w_{q+1}^{(t)} \cdot \nabla (u_q - \bar u_q)\|_N \lesssim \frac{\delta_{q+1}^{1/2}\lambda_q^{1/3}}{\lambda_{q+1}^{1/3}} \delta_q^{1/2} \frac{\lambda_q}{\lambda_{q+1}} \lambda_{q+1}^{N+1}, \,\,\, \forall N \in \{0,1...,L-1\},
    \end{equation*}
    and it remains to collect estimates for $\bar D_t (u_q - \bar u_q)$. We write
    \begin{equation*}
        \bar D_{t} (u_q - \bar u_q) =  (\partial_t u_q + u_q \cdot \nabla u_q) - (\partial_t u_q + u_q \cdot \nabla u_q)*\zeta_{\ell_q} + (\bar u_q - u_q) \cdot \nabla u_q + \div((u_q \otimes u_q)*\zeta_{\ell_q} - \bar u_q \otimes \bar u_q).
    \end{equation*}
    Using the Euler-Reynolds system \ref{ER}, we have 
    \begin{equation*}
        \|(\partial_t u_q + u_q \cdot \nabla u_q) - (\partial_t u_q + u_q \cdot \nabla u_q)*\zeta_{\ell_q}\|_N \lesssim \|\nabla p_q - \nabla p_q * \zeta_{\ell_q}\|_N + \|\div R_q - \div R_{q, 0}\|_N
    \end{equation*}
    By the inductive estimates and proposition \ref{moli}, we have 
    \begin{equation*}
        \|\nabla p_q - \nabla p_q * \zeta_{\ell_q}\|_N \lesssim \frac{\delta_q \lambda_q^2}{\lambda_{q+1}}\lambda_{q+1}^N, \,\,\, \forall N \in \{0,1,...,L-1\},
    \end{equation*}
    and 
    \begin{equation*}
        \|\div R_q - \div R_{q, 0}\|_N \lesssim \frac{\delta_{q+1} \lambda_q^2}{\lambda_{q+1}} \lambda_{q+1}^N, \,\,\, \forall N \in \{0,1,...,L-1\}.
    \end{equation*}
    Moreover, 
    \begin{equation*}
        \|(\bar u_q - u_q) \cdot \nabla u_q\|_N \lesssim \|\bar u_q - u_q\|_N \|u_q\|_1 + \|\bar u_q - u_q\|_0\|u_q\|_{N+1} \lesssim \frac{\delta_q \lambda_q^2}{\lambda_{q+1}} \lambda_{q+1}^N, \,\,\, \forall N \in \{0,1,...,L-1\},
    \end{equation*}
    and, finally, by the Constantin-E-Titi commutator estimate of proposition \ref{CET_comm}, 
    \begin{equation*}
        \|\div((u_q \otimes u_q)*\zeta_{\ell_q} - \bar u_q \otimes \bar u_q)\|_N \lesssim \ell_q^{3 - (N+1)}\|u_q\|_2\|u_q\|_1 \lesssim \frac{\delta_q \lambda_q^2}{\lambda_{q+1}}\lambda_{q+1}^{N}, \,\,\, \forall N \geq 0. 
    \end{equation*}
    Upon noting that 
    \begin{equation*}
        \delta_q^{1/2} \lambda_q < \delta_{q+1}^{1/2}\lambda_q^{1/3}\lambda_{q+1}^{2/3},
    \end{equation*}
    we conclude that 
    \begin{equation*}
        \|\bar D_{t, \Gamma} (u_q - \bar u_q)\|_N \lesssim \delta_{q+1}^{1/2} \lambda_q^{1/3}\lambda_{q+1}^{2/3} \frac{\delta_q^{1/2} \lambda_q}{\lambda_{q+1}} \lambda_{q+1}^N, \,\,\, \forall N \in \{0,1,...,L-1\}.
    \end{equation*}
    Then, \eqref{spat_moli_estim_2} follows since 
    \begin{equation*}
        \mu_{q+1} = \delta_{q+1}^{1/2} \lambda_q^{2/3}\lambda_{q+1}^{1/3} \lambda_{q+1}^{4\alpha} \leq \delta_{q+1}^{1/2}\lambda_q^{1/3}\lambda_{q+1}^{2/3},
    \end{equation*}
    whenever $\alpha$ is chosen sufficiently small.

    The inequalities \eqref{spat_moli_estim_3} follow  from proposition \ref{moli} by the same reasoning we have used above to estimate $u_q - \bar u_q$. Finally, for \eqref{spat_moli_estim_4}, we write 
    \begin{equation*}
        \bar D_{t, \Gamma}(R_{q} - R_{q, 0}) = w_{q+1}^{(t)} \cdot \nabla (R_{q} - R_{q, 0}) + D_t R_q + (\bar u_q - u_q)\cdot \nabla R_q - \bar D_t R_{q,0},
    \end{equation*}
    and compute:
    \begin{equation*}
        \|w_{q+1}^{(t)} \cdot \nabla (R_q - R_{q, 0})\|_N \lesssim \frac{\delta_{q+1}^{1/2} \lambda_q^{1/3}}{\lambda_{q+1}^{1/3}} \delta_{q+1} \lambda_q \lambda_{q+1}^N \lesssim \delta_{q+1}\delta_q^{1/2}\lambda_q \lambda_{q+1}^N, \,\,\, \forall N \in \{0,1,...,L-1\},
    \end{equation*}
    \begin{equation*}
        \|D_t R_q\|_N + \|\bar D_t R_{q,0}\|_N \lesssim \delta_{q+1} \delta_q^{1/2}\lambda_q \lambda_{q+1}^N, \,\,\, \forall N\in\{0,1,....,L-1\},
    \end{equation*}
    where we use lemma \ref{smoli_estim} and the inductive assumption; and 
    \begin{equation*}
        \|(\bar u_q - u_q)\cdot \nabla R_q\|_N \lesssim \delta_{q+1} \lambda_q \delta_q^{1/2}\frac{\lambda_q}{\lambda_{q+1}} \lambda_{q+1}^N \lesssim \delta_{q+1}\delta_q^{1/2} \lambda_q \lambda_{q+1}^N \,\,\, \forall N \in \{0,1,...,L-1\}.
    \end{equation*}
    The final conclusion, then, follows. 
\end{proof}

\subsection{Estimates for the pressure} Recalling the definition \eqref{def.pq1} of $p_{q+1}$ and the expression \eqref{p_q_Gam} for $p_{q, \Gamma}$, we write 
\begin{equation*}
    p_{q+1} = p_q + \sum_{n=1}^\Gamma p_{q+1, n}^{(t)} - \sum_{n=0}^{\Gamma-1} \Delta^{-1} \div \div\big(R_{q,n} + \sum_{\xi, k} g_{\xi, k, n+1}^2 A_{\xi, k, n}\big) - \frac{|w_{q+1}^{(t)}|^2}{2} + \langle \bar u_q - u_q, w_{q+1} \rangle.
\end{equation*}

\begin{lem} 
The following estimates hold for the new pressure: 
\begin{equation}
    \|p_{q+1}\|_N \lesssim \frac{\delta_{q+1}^{1/2}\delta_q^{1/2}\lambda_q^{1/3}}{\lambda_{q+1}^{1/3}} \lambda_{q+1}^N, \,\,\, \forall N \in \{1, 2,...,L\},
\end{equation}
where the implicit constants depend on $\Gamma$, $M$, $\alpha$, and $N$.
\end{lem}

\begin{proof}
    We have 
    \begin{eqnarray*}
        \|p_{q+1}\|_N &\lesssim & \|p_q\|_N + \sum_{n=1}^\Gamma \|p_{q+1, n}^{(t)}\|_N + \sum_{n = 0}^{\Gamma-1}\|\Delta^{-1} \div \div \big( R_{q, n} + \sum_{\xi, k} g_{\xi, k, n+1}^2 A_{\xi, k, n }\big)\|_N \\
        && + \||w_{q+1}^{(t)}|^2\|_N + \|\langle \bar u_q - u_q, w_{q+1} \rangle\|_N 
    \end{eqnarray*}
    
    \textit{Estimates for $p_q$.} By the inductive assumptions, we have 
    \begin{equation*}
        \|p_q\|_N \lesssim \frac{\delta_{q}\lambda_q}{\lambda_{q+1}}\lambda_{q+1}^N, \,\,\, \forall N \in \{1, 2,...,L\}.
    \end{equation*}
    Note, then, that 
    \begin{equation*}
        \frac{\delta_q \lambda_q}{\lambda_{q+1}} \leq \frac{\delta_{q+1}^{1/2}\delta_q^{1/2} \lambda_q^{1/3}}{\lambda_{q+1}^{1/3}}.
    \end{equation*}
    
    \textit{Estimates for $p_{q+1,n}^{(t)}$.} Recall that $p_{q+1, n}^{(t)} = \sum_k \tilde \chi_k p_{k, n}$, so it suffices to obtain estimates for $p_{k, n}$ on $\supp \tilde \chi_k$. We note, then, that \eqref{LocalNewt} implies 
    \begin{equation*}
        p_{k,n} = - 2\Delta^{-1} \sum_{i,j =1}^2 \partial_i w_{k, n}^j \partial_j \bar u_q^{i} = -2 \Delta^{-1} \div (w_{k, n} \cdot \nabla \bar u_q) =  2 \Delta^{-1} \div \div (\psi_{k, n} \nabla^\perp \bar u_q).
    \end{equation*}
    Therefore, since $\Delta^{-1} \div \div$ is an operator of Calder\'on-Zygmund type, we obtain, for $N \geq 0$, 
    \begin{eqnarray*}
        \| p_{k,n}\|_{N+\alpha} &\lesssim & \|\psi_{k, n}\|_{N+\alpha}\|\nabla \bar u_q\|_\alpha + \|\psi_{k, n}\|_{\alpha} \|\nabla \bar u_q\|_{N+\alpha} \\ 
        & \lesssim & \frac{\delta_{q+1} }{\mu_{q+1}}\delta_q^{1/2} \lambda_q \lambda_{q+1}^{N + 2\alpha} \lesssim \frac{\delta_{q+1}^{1/2} \delta_{q}^{1/2}\lambda_q^{1/3}}{\lambda_{q+1}^{1/3}} \lambda_{q+1}^N.
    \end{eqnarray*}

    \textit{Estimates for $\Delta^{-1} \div \div \big( R_{q, n} + \sum_{\xi, k} g_{\xi, k, n+1}^2 A_{\xi, k, n }\big)$.} We have, for all $N \geq 1$, 
    \begin{equation*}
        \|\Delta^{-1} \div \div \big( R_{q, n} + \sum_{\xi, k} g_{\xi, k, n+1}^2 A_{\xi, k, n }\big) \|_{N+ \alpha} \lesssim \|R_{q,n}\|_{N+\alpha} + \|A_{\xi, k, n}\|_{N+\alpha} \lesssim \frac{\delta_{q+1} \lambda_q}{\lambda_{q+1}^{1-\alpha}} \lambda_{q+1}^N,
    \end{equation*}
    where we have appealed to lemma \ref{a_cor} and proposition \ref{NewIter}. Once again, we note that for all $\alpha > 0$ sufficiently small, it holds that 
    \begin{equation*}
        \frac{\delta_{q+1}\lambda_q}{\lambda_{q+1}^{1-\alpha}} \leq \frac{\delta_{q+1}^{1/2}\delta_q^{1/2}\lambda_q^{1/3}}{\lambda_{q+1}^{1/3}}.
    \end{equation*}
  
    The estimates for $|w_{q+1}^{(t)}|^2$ and $\langle \bar u_q - u_q, w_{q+1} \rangle$ are the same as those obtained in lemmas \ref{Newt_err_estim} and \ref{spatial_moli_estim_err_fin}, respectively. Both satisfy better estimates than those claimed here. 
\end{proof}

\begin{cor} \label{press_corol_fin}
The following hold: 
\begin{equation}
    \|p_{q+1}\|_N \leq M^2 \delta_{q+1} \lambda_{q+1}^N, \,\,\, \forall N\in\{1,2,...,L\}.
\end{equation}
\end{cor}

\begin{proof}
    Since $\beta < 1/3$, we have that 
    \begin{equation*}
        \frac{\delta_{q+1}^{1/2}\delta_q^{1/2}\lambda_q^{1/3}}{\lambda_{q+1}^{1/3}} < \delta_{q+1}\lambda_{q+1}^{-\alpha},
    \end{equation*}
    whenever $\alpha > 0$ is sufficiently small. Then, the previous lemma implies that there exists a constant $C$ depending on $L$, $M$, $\alpha$, and $\beta$ such that
    \begin{equation*}
        \|p_{q+1}\|_N \leq C \lambda_{q+1}^{-\alpha} \delta_{q+1}\lambda_{q+1}^N, \,\,\, \forall N \in \{1,2,...,L\}.
    \end{equation*}
    The conclusion follows once we choose $a_0$ sufficiently large so that 
    \begin{equation*}
        C \lambda_{q+1}^{-\alpha} < M^2.
    \end{equation*}
\end{proof}

\subsection{Conclusion} Corollaries \ref{velo_corrol_fin} and \ref{press_corol_fin} show, respectively, the inductive propagation of the estimates concerning the velocity field and the pressure. It remains, then, to check the propagation of the estimates on the Reynolds stress. In the following, we denote the material derivative corresponding to the vector field $u_{q+1}$ by  
\begin{equation*}
    D_{t,q+1} = \pa_t + u_{q+1}\cdot \na.
\end{equation*}
\begin{cor}
    The following estimates hold for the new Reynolds stress:
    \begin{equation} \label{Spatial_estim_qplus1}
        \|R_{q+1}\|_N \leq \delta_{q+2} \lambda_{q+1}^{N-3\alpha}, \,\,\, \forall N \in \{0,1,...,L\},
    \end{equation}
    \begin{equation} \label{Material_estim_qplus1}
        \|D_{t, q+1} R_{q+1}\|_N \leq \delta_{q+2} \delta_{q+1}^{1/2} \lambda_{q+1}^{N+1-2\alpha}, \,\,\, \forall N \in \{0,1,...,L-1\}.
    \end{equation}
\end{cor}
\begin{proof}
    The estimates obtained in the previous sections imply 
    \begin{equation*}
        \|R_{q+1}\|_N \lesssim \bigg(\frac{\delta_q^{1/2}\delta_{q+1}^{1/2}\lambda_q}{\lambda_{q+1}} + \frac{\delta_{q+1}\lambda_q^{2/3}}{\lambda_{q+1}^{2/3}} + \frac{\delta_{q+1}\lambda_q}{\lambda_{q+1}} + \frac{\delta_{q+1}^{3/2} \lambda_q^{1/3}}{\delta_q^{1/2}\lambda_{q+1}^{1/3}}\bigg) \lambda_{q+1}^{5\alpha} \lambda_{q+1}^N.
    \end{equation*}
    One can check directly that 
    \begin{equation*}
        \frac{\delta_{q+1}\lambda_q}{\lambda_{q+1}} < \frac{\delta_q^{1/2}\delta_{q+1}^{1/2}\lambda_q}{\lambda_{q+1}} < \frac{\delta_{q+1}\lambda_q^{2/3}}{\lambda_{q+1}^{2/3}} < \frac{\delta_{q+1}^{3/2} \lambda_q^{1/3}}{\delta_q^{1/2}\lambda_{q+1}^{1/3}},
    \end{equation*}
    which implies 
    \begin{equation*}
        \|R_{q+1}\|_N \lesssim \frac{\delta_{q+1}^{3/2} \lambda_q^{1/3}}{\delta_q^{1/2}\lambda_{q+1}^{1/3}} \lambda_{q+1}^{5\alpha} \lambda_{q+1}^N.
    \end{equation*}
    Note that since 
    \begin{equation*}
        1 < b < \frac{1 + 3 \beta}{6 \beta},
    \end{equation*}
    we have 
    \begin{equation*}
        \frac{\delta_{q+1}^{3/2} \lambda_q^{1/3}}{\delta_q^{1/2}\lambda_{q+1}^{1/3}} < \delta_{q+2},
    \end{equation*}
    and, thus, by choosing $\alpha$ sufficiently small depending on $\beta$ and $b$, we can ensure that 
    \begin{equation*}
        \frac{\delta_{q+1}^{3/2} \lambda_q^{1/3}}{\delta_q^{1/2}\lambda_{q+1}^{1/3}} \lambda_{q+1}^{5 \alpha} \leq \delta_{q+2} \lambda_{q+1}^{-4\alpha}.
    \end{equation*}
    Then, there exists a constant $C$ depending on $L$, $\beta$, $b$, $\alpha$ and $M$ such that 
    \begin{equation*}
        \|R_{q+1}\|_N \leq C \lambda_{q+1}^{-\alpha} \delta_{q+2} \lambda_{q+1}^{N-3\alpha}, \,\,\, \forall N \in \{0,1,...,L\}.
    \end{equation*}
    The estimates in \eqref{Spatial_estim_qplus1} follow once we choose $a_0$ sufficiently large so that 
    \begin{equation*}
        C \lambda_{q+1}^{-\alpha} < 1.
    \end{equation*}
    
    It remains to show the validity of the estimates on the material derivative corresponding to $u_{q+1}$. For this purpose, we write
    \begin{align*}
        \|D_{t,q+1} R_{q+1}\|_N \lesssim \|\bar D_{t,\Ga} R_{q+1}\|_N + \|(u_q - \bar u_q) \cdot\na R_{q+1}\|_N + \|w_{q+1}^{(s)} \cdot \nabla R_{q+1}\|_{N}.
    \end{align*}
    Plugging in the expressions for $\tau_q$, $\mu_{q+1}$ and $\ell_{t,q}$ into the estimates obtained in the previous sections, we obtain that, for all $N \in \{0,1,..., L-1\}$,
    \begin{eqnarray*}
        \|\bar D_{t, \Gamma} R_{q+1}\|_N &\lesssim& \bigg(\frac{\delta_{q+1}\delta_q^{1/2}\lambda_q^{5/3}}{\lambda_{q+1}^{2/3}} + \delta_{q+1}\delta_q^{1/2} \lambda_q + \delta_{q+1}^{3/2}\lambda_q^{2/3}\lambda_{q+1}^{1/3} + \frac{\delta_{q+1}^{3/2}\lambda_q^{5/3}}{\lambda_{q+1}^{2/3}}  \\ 
        && + \frac{\delta_{q+1}\delta_q^{1/2}\lambda_q^2}{\lambda_{q+1}} + \frac{\delta_{q+1}^{3/2}\lambda_q^{4/3}}{\lambda_{q+1}^{1/3}} + \frac{\delta_{q+1}\delta_q^{1/2}\lambda_q^{4/3}}{\lambda_{q+1}^{1/3}} \bigg) \lambda_{q+1}^{5\alpha}\lambda_{q+1}^N.
    \end{eqnarray*}
    One can check that 
    \begin{equation*}
        \frac{\delta_{q+1}\delta_q^{1/2}\lambda_q^2}{\lambda_{q+1}} < \frac{\delta_{q+1}^{3/2}\lambda_q^{5/3}}{\lambda_{q+1}^{2/3}} < \frac{\delta_{q+1}\delta_q^{1/2}\lambda_q^{5/3}}{\lambda_{q+1}^{2/3}} < \frac{\delta_{q+1}^{3/2}\lambda_q^{4/3}}{\lambda_{q+1}^{1/3}} < \frac{\delta_{q+1}\delta_q^{1/2}\lambda_q^{4/3}}{\lambda_{q+1}^{1/3}} < \delta_{q+1}\delta_q^{1/2} \lambda_q < \delta_{q+1}^{3/2}\lambda_q^{2/3}\lambda_{q+1}^{1/3},
    \end{equation*}
    and, thus, 
    \begin{equation*}
        \|\bar D_{t, \Gamma} R_{q+1}\|_N \lesssim \delta_{q+1}^{3/2} \lambda_q^{2/3}\lambda_{q+1}^{1/3} \lambda_{q+1}^{5\alpha} \lambda_{q+1}^N.
    \end{equation*}
    Since, 
    \begin{equation*}
        b < \frac{1 + 3 \beta}{6 \beta} < \frac{1}{3\beta},
    \end{equation*}
    we argue as before to conclude that 
    \begin{equation*}
        \|\bar D_{t, \Gamma} R_{q+1}\|_N \leq \delta_{q+2} \delta_{q+1}^{1/2} \lambda_{q+1}^{N+1-3\alpha}, \,\,\, \forall N \in \{0,1,...,L-1\},
    \end{equation*}
    whenever $\alpha > 0$ is sufficiently small. Recalling that 
    \begin{equation*}
        \|\bar u_q - u_q\|_N \lesssim \delta_q^{1/2}\frac{\lambda_q}{\lambda_{q+1}}\lambda_{q+1}^N, \,\,\, \forall N \in \{0,1,...,L\},
    \end{equation*}
    we obtain 
    \begin{eqnarray*}
        \|(\bar u_q - u_q) \cdot \nabla R_{q+1}\|_N &\lesssim& \|\bar u_q - u_q\|_N \|R_{q+1}\|_1 + \|\bar u_q - u_q\|_0\|R_{q+1}\|_{N+1} \\
        &\lesssim & \delta_{q}^{1/2} \frac{\lambda_q}{\lambda_{q+1}} \delta_{q+2} \lambda_{q+1}^{N+1-3\alpha} \lesssim \delta_{q+2} \delta_{q+1}^{1/2} \lambda_{q+1}^{N+1-3\alpha},
    \end{eqnarray*}
    where for the last inequality we use the fact that 
    \begin{equation*}
        \delta_q^{1/2}\frac{\lambda_q}{\lambda_{q+1}} \leq \delta_{q+1}^{1/2}.
    \end{equation*}
    Finally, by lemma \ref{velo_estim_fin}, 
    \begin{eqnarray*}
        \|w_{q+1}^{(s)} \cdot \nabla R_{q+1}\|_N &\lesssim& \|w_{q+1}^{(s)}\|_N \|R_{q+1}\|_1 + \|w_{q+1}^{(s)}\|_0 \|R_{q+1}\|_{N+1} \\ 
        &\lesssim& \delta_{q+2}\delta_{q+1}^{1/2}\lambda_{q+1}^{N +1 - 3\alpha}.
    \end{eqnarray*}
    We conclude, then, that there exists a constant $C$ depending on $L$, $\beta$, $b$, $\alpha$ and $M$ so that 
    \begin{equation*}
        \|D_{t, q+1} R_{q+1}\|_N \leq C \lambda_{q+1}^{-\alpha} \delta_{q+2}\delta_{q+1}^{1/2} \lambda_{q+1}^{N+1-2\alpha}, \,\,\, \forall N \in \{0,1,...,L-1\}.
    \end{equation*}
    The conclusion follows once we choose $a_0$ sufficiently large.
\end{proof}

\appendix

\section{H\"older spaces, compositions, mollification}\label{sec.a}

Let $N \in \mathbb N$ and $\alpha \in [0,1)$. For functions $f:\mathbb T^2 \rightarrow \mathbb R$, we denote the $C^0$ norm by 
\begin{equation*}
    \|f\|_0 := \sup_{x \in \mathbb T^2} |f(x)|,
\end{equation*}
and the H\"older semi-norms by 
\begin{equation*}
    [f]_N = \sup_{|\theta|=N} \|D^\theta f\|_0,
\end{equation*}
\begin{equation*}
    [f]_{N+\alpha} = \sup_{|\theta|= N} \sup_{x \neq y} \frac{|D^\theta f(x) - D^\theta f(y)|}{|x-y|^\alpha},
\end{equation*}
where in the above $\theta$ denotes a multi-index. We denote the H\"older norms by 
\begin{equation*}
    \|f\|_N = \sum_{j = 0}^N [f]_j,
\end{equation*}
\begin{equation*}
    \|f\|_{N+\alpha} = \|f\|_N + [f]_{N+\alpha}.
\end{equation*}
We keep the notation above also in the case $f$ is a vector field or a tensor field. Moreover, if $f$ is time-dependent, we denote by a slight abuse of notation, 
\begin{equation*}
    \|f\|_N = \sup_t \|f(\cdot, t)\|_N,
\end{equation*}
\begin{equation*}
    \|f\|_{N+\alpha} = \sup_t \|f(\cdot, t)\|_{N+\alpha}.
\end{equation*}
When we are interested in the spatial H\"older norms at a particular time-slice, we use the notation on the right-hand side of the two equations above. Finally, with $A \subset \mathbb R$, we denote 
\begin{equation*}
    \|f\|_{N, A} = \sup_{t \in A} \|f(\cdot, t)\|_N.
\end{equation*}

Let us recall the classical interpolation inequality 
\begin{equation*}
    \|f\|_{N+\alpha} \leq C \|f\|_{N_1 + \alpha_1}^\lambda \|f\|_{N_2 + \alpha_2}^{1-\lambda},
\end{equation*}
where 
\begin{equation*}
    N + \alpha = \lambda(N_1 + \alpha_1) + (1-\lambda)(N_2 + \alpha_2),
\end{equation*}
and $C$ is a constant depending on $N$, $N_1$, $N_2$, $\alpha$, $\alpha_1$ and $\alpha_2$. Recall also the following standard estimate for products:
\begin{equation*}
    \|fg\|_{N+\alpha} \leq C(\|f\|_{N+\alpha}\|g\|_0 + \|f\|_0\|g\|_{N+\alpha}).
\end{equation*}

Regarding compositions, we recollect the classical estimate below. The reader is referred to \cite{DLS14} for a proof. 

\begin{prop} \label{comp_estim}
    Let $\Psi:\Omega \rightarrow \mathbb R$ and $u: \mathbb R^n \rightarrow \Omega$ be two smooth functions, with $\Omega \subset \mathbb R^N$. Then, for any $m \in \mathbb N \setminus \{0\}$, there exists a constant $C = C(m, N,n)$ such that 
    \begin{equation*}
        [\Psi \circ u]_m \leq C\big([\Psi]_1 \|Du\|_{m-1} + \|D\Psi\|_{m-1} \|u\|_0^{m-1} \|u\|_m \big),
    \end{equation*}
    \begin{equation*}
        [\Psi \circ u]_m \leq C\big([\Psi]_1 \|Du\|_{m-1} + \|D\Psi\|_{m-1} [u]_1^m \big).
    \end{equation*}
\end{prop}

We also use the following results concerning mollification. A detailed proof is given in lemma 2.1 of \cite{CDLS12} 

\begin{prop} \label{moli}
Let $\phi$ be a symmetric mollifier with $\int \phi = 1$. Then, for any $f \in C^\infty(\mathbb T^2)$ and $N \geq 0$,
\begin{equation*}
    \|f - f*\phi_\ell\|_N \lesssim \ell^2 \|f\|_{N+2},
\end{equation*}
where the implicit constant depends only on $N$.
\end{prop}

\begin{prop} \label{CET_comm}
Let $\phi$ be a standard mollifier. Then, for any $f, g \in C^\infty(\mathbb T^2)$ and $N \geq M \geq 0$,
\begin{equation*}
    \|(fg)*\phi_\ell - (f*\phi_\ell)(g*\phi_\ell)\|_{N} \lesssim \ell^{2 - N + M}\big( \|f\|_{M+1}\|g\|_1 + \|f\|_{1}\|g\|_{M+1} \big),
\end{equation*}
where the implicit constant depends only on $N$ and $M$. 
\end{prop}

\begin{proof}
The case $M = 0$ is proved in \cite{CDLS12}. If $M > 0$, the previous case implies 
\begin{equation*}
    \|(fg)*\phi_\ell - (f*\phi_\ell)(g*\phi_\ell)\|_{N-M} \lesssim \ell^{2-N+M} \|f\|_1 \|g\|_1.
\end{equation*}
Let $\theta$ be a multi-index with $|\theta| \leq M$. Then, 
\begin{eqnarray*}
    \|\partial^\theta[(fg)*\phi_\ell - (f*\phi_\ell)(g*\phi_\ell)]\|_{N-M} & \lesssim & \sum_{\alpha \leq \theta} \|(\partial^\alpha f \partial^{\theta - \alpha}g) *\phi_\ell - (\partial^\alpha f * \phi_\ell)(\partial^{\theta - \alpha} g * \phi_\ell)\|_{N-M}  \\ 
    &\lesssim & \ell^{2 - N + M} \sum_{\alpha \leq \theta} \|\partial^\alpha f\|_1 \|\partial^{\theta - \alpha} g\|_1 \\
    &\lesssim& \ell^{2 - N + M}\big( \|f\|_{M+1}\|g\|_1 + \|f\|_1\|g\|_{M+1}\big ),
\end{eqnarray*}
where, for the last inequality, we use interpolation and Young's inequality for products. The conclusion follows once we note that $\|(fg)*\phi_\ell - (f*\phi_\ell)(g*\phi_\ell)\|_{N}$ is bounded by the sum of the terms on the left hand side of the previous relation. 
\end{proof}

The proposition above is a version of the Constantin-E-Titi commutator estimate (\cite{ConstantinETiti94}).

\section{Estimates for transport equations}

We recall standard estimates for solutions to the transport equation 

\begin{equation} \label{transport}
    \begin{cases}
    \partial_t f + u \cdot \nabla f = g, \\ 
    f \big | _{t = t_0} = f_0. 
    \end{cases}
\end{equation}

The following proposition is stated in \cite{cltv} and it follows by interpolation from the corresponding result in \cite{BDLIS15}. 

\begin{prop} \label{transport_estim}
Assume $|t-t_0| \|u\|_1 \leq 1$. Any solution $f$ of \eqref{transport} satisfies 
\begin{equation*}
    \|f(\cdot, t)\|_0 \leq \|f_0\|_0 + \int_{t_0}^t \|g(\cdot, \tau)\|_0 d \tau, 
\end{equation*}
\begin{equation*}
    \|f(\cdot, t)\|_\alpha \leq 2 \big( \|f_0\|_\alpha + \int_{t_0}^t \|g(\cdot, \tau)\|_\alpha d\tau\big),
\end{equation*}
for $\alpha \in [0,1]$.
More generally, for any $N \geq 1$ and $\alpha \in [0,1)$,
\begin{equation*}
[f(\cdot, t)]_{N+\alpha} \lesssim [f_0]_{N+\alpha} + |t-t_0| [u]_{N+\alpha} [f_0]_1 + \int_{t_0}^t \big( [g(\cdot, \tau)]_{N+\alpha} + (t-\tau) [u]_{N+\alpha} [g(\cdot, \tau)]_1\big) d\tau,
\end{equation*}
where the implicit constant depends on $N$ and $\alpha$.
Consequently, the backwards flow $\Phi$ of $u$ starting at time $t_0$ satisfies 
\begin{equation*}
    \|D\Phi(\cdot, t) - \I\|_0 \lesssim |t-t_0|[u]_1, 
\end{equation*}
\begin{equation*}
    [\Phi(\cdot, t)]_N \lesssim  |t - t_0| [u]_N, \,\,\,  \forall N \geq 2. 
\end{equation*}
\end{prop}

\section{Singular integral operators} \label{SIO}

In this paper, we consider the following class of Calder\'on-Zygmund operators: Let $K:\mathbb R^2 \rightarrow \mathbb R$ be a kernel which is homogeneous of degree $-2$, smooth away from the origin, and has zero mean on circles centered at the origin. Consider the periodization of $K$ 
\begin{equation*}
    K_{\mathbb T^2}(z) = K(z) + \sum_{n \in \mathbb Z^2 \setminus \{0\}} (K(z+n) - K(n)).
\end{equation*}
Define, then, 
\begin{equation*}
    T_K f(x) = p.v. \int_{\mathbb T^2} K_{\mathbb T^2}(x-y)f(y) dy,
\end{equation*}
to be the $\mathbb T^2$-periodic Calder\'on-Zygmund operator acting on $\mathbb T^2$-periodic functions $f$ of zero mean. The following is, then, a classical result (\cite{CZ}). 

\begin{prop}
    For $\alpha \in (0,1)$, the periodic Calder\'on-Zygmund operators are bounded on the space of zero mean $\mathbb T^2$-periodic $C^\alpha$ functions.
\end{prop}

 We also recall the following commutator estimate, which is a variant of lemma 1 in \cite{constantin}. There, the result is stated on the whole space $\mathbb R^d$ and only for $N=0$. The adaptation to the case of the periodic torus and the extension to $N > 0$ were given in proposition D.1. of \cite{cltv}. We remark that, while the corresponding proposition in \cite{cltv} is stated on $\mathbb T^3$, the arguments carry over to any dimension without modification.

 \begin{prop} \label{CZ_comm}
     Let $\al \in (0,1)$ and $N\geq 0$. Let $T_K$ be a Calder\'on-Zygmund operator with kernel $K$. Let $b \in C^{N+\al}(\T^2)$ be a vector field. Then, we have
     \begin{align*}
         \|[T_K, b \cdot \na] f \|_{N+\al} \lec \|b\|_{1+\al} \|f\|_{N+\al} + \|b\|_{N+1+\al} \|f\|_{\al}
     \end{align*}
     for any $f \in C^{N+\al}(\T^2)$, where the implicit constant depends on $\al,N,K$.
 \end{prop}

 \section{Tools of convex integration} \label{conv_int_tool}

\subsection{A geometric lemma}

The following geometric lemma, which is essentially due to Nash (\cite{Nash}) and which was reformulated in the form below in \cite{S12}, is used to decompose the Reynolds stress into simple tensors.

\begin{lem} \label{geom}
Let $B_{1/2}(\I)$ denote the metric ball centered around the identity in the space $\mathcal{S}^{2 \times 2}$ of symmetric $2 \times 2$ matrices. There exist a finite set $\Lambda \subset \mathbb Z^2$ and smooth functions $\gamma_\xi: B_{1/2}(\I) \rightarrow \mathbb{R}$, for each $\xi \in \Lambda$, such that
\begin{equation*}
    R = \sum_{\xi \in \Lambda} \gamma_\xi^2 (R) \xi \otimes \xi,
\end{equation*}
whenever $R \in B_{1/2}(\I)$. 
\end{lem}
 
\subsection{An inverse divergence operator}
We use the following inverse-divergence operator
\begin{equation} \label{invdiv}
    (\mathcal{R} u)^{ij} =  \Delta^{-1} (\partial_i u^j + \partial_j  u^i - \div u \delta_{ij}),
\end{equation}
which maps smooth, mean-zero vector fields $u$ to smooth, symmetric and trace-free $2$-tensors $\mathcal{R}u$. This operator was first defined in \cite{ChDLS12} and is the $2$-dimensional counterpart of the $3$-dimensional inverse divergence operator introduced in \cite{DLS13}, \cite{DLS14}. The main properties of $\mathcal{R}$ are gathered in the following proposition, which is proven in \cite{ChDLS12}. 

\begin{prop}
If $u$ is a smooth, mean-zero vector field, then the $2$-tensor field $\mathcal{R} u$ defined by \eqref{invdiv} is symmetric and satisfies 
\begin{equation*}
    \div \mathcal{R} u = u.  
\end{equation*}
\end{prop} 

\subsection{A stationary phase lemma}

We refer the reader to \cite{DaneriSzekelyhidi17} for the proof of the following stationary phase lemma. We remark, however, that, while the lemma is therein stated on $\mathbb T^3$, an inspection of the proof immediately reveals that the arguments are independent of the dimension of the spatial domain.

\begin{prop}\label{prop.inv.div}
    Let $\alpha \in (0,1)$ and $N \geq 1$. Let $a \in C^\infty(\mathbb{T}^2)$, $\Phi \in C^\infty(\mathbb{T}^2;\mathbb{R}^2)$ be smooth functions and assume that
    \[\hat C^{-1} \leq |\na \Phi| \leq \hat C\]
    holds on $\mathbb{T}^2$. Then
    \begin{equation}
        \left|\int_{\T^2} a(x)e^{ik\cdot\Phi}\,dx\right| \lesssim \frac{\|a\|_N + \|a\|_0\|\nabla \Phi\|_N}{|k|^N}\,,
    \end{equation}
    and for the operator $\mathcal{R}$ defined in \eqref{invdiv}, we have
    \begin{equation}
        \left\|\mathcal{R}\left(a(x)e^{ik\cdot\Phi}\right)\right\|_\alpha \lesssim \frac{\|a\|_0}{|k|^{1-\alpha}} + \frac{\|a\|_{N+\alpha} + \|a\|_0\|\nabla \Phi\|_{N+\alpha}}{|k|^{N-\alpha}}
    \end{equation}
    where the implicit constants depend on $\hat C$, $\alpha$ and $N$, but not on $k$.
\end{prop}

\section{Global well-posedness for the linearized Euler equations} \label{wl-psd}

Consider the Newtonian linearization of the Euler equations on the domain $[-T, T] \times \mathbb T^d$, $T>0$ and $d \geq 2$.  

\begin{equation} \label{LinEul}
    \begin{cases}
        \partial_t w + u \cdot \nabla w + w \cdot \nabla u + \nabla p = F \\ 
        \div w = 0, \\ 
        w \big|_{t = 0} = w_0,
    \end{cases}
\end{equation}
where $u:[-T,T]\times\mathbb T^d \rightarrow \mathbb R^d$ is a divergence-free vector field, the forcing $F:[-T,T] \times \mathbb T^d \rightarrow \mathbb R^d$ is taken, without loss of generality, to be divergence-free, and the initial data $w_0:\mathbb T^d \rightarrow \mathbb R^d$ is also assumed to be divergence-free. The unknowns are the vector-field $w:[-T,T] \times \mathbb T^d \rightarrow \mathbb R^d$ and the scalar pressure $p:[-T,T] \times \mathbb T^d \rightarrow \mathbb R$.

The pressure is determined up to an additive constant by 
\begin{equation} \label{press_LinEul}
    - \Delta p = \div \div (u \otimes w + w \otimes u) = 2 \div (w \cdot \nabla u).
\end{equation}
Therefore, we can rewrite the system \eqref{LinEul} as 
\begin{equation} \label{LinEul2}
    \begin{cases}
        \partial_t w + u \cdot \nabla w + (\I - 2 \nabla \Delta^{-1}\div)(w \cdot \nabla u) = F \\ 
        w \big|_{t = 0} = w_0.
    \end{cases}
\end{equation}

The following proposition shows global well-posedness for the equations \eqref{LinEul}. We will impose on a solution $(w, p)$ that $p$ has zero-mean, so that it is uniquely determined by \eqref{press_LinEul}.

\begin{prop}
Let $T > 0$, $N \in \mathbb N \setminus \{0\}$, $0<\alpha<1$, and assume $u\in C([-T,T]; C^{N + 1+\alpha}(\mathbb T^d))$, $F \in C([-T, T]; C^{N+\alpha}(\mathbb T^d))$ and $w_0 \in C^{N+\alpha}(\mathbb T^d)$ are divergence-free vector fields. Then, there exists a unique solution $(w, p)$ of \eqref{LinEul} such that 
$$w \in C([-T, T]; C^{N+\alpha}(\mathbb T^d)) \cap C^{1}([-T,T]; C^{N -1 + \alpha }(\mathbb T^d))$$ and $p \in C([-T,T]; C^{N+\alpha}(\mathbb T^d))$.
\end{prop}

\begin{proof}
    Let us first note that it suffices to prove the well-posedness of \eqref{LinEul2} in the required regularity class. The regularity of the pressure $p$ will then follow from that of $w$ since $\Delta^{-1}\div \div$ is an operator of Calder\'on-Zygmund type. 

    \textit{Local existence.} Let $m \in \mathbb N \setminus \{0\}$ and let $w_m \in C_tC_x^{N+\alpha} \cap C_t^1C_x^{N-1+\alpha}$ be the inductively defined solution of 
    \begin{equation*}
        \begin{cases}
            \partial_t w_m + u\cdot \nabla w_m + T(w_{m-1}\cdot \nabla u) = F \\ 
            w_m \big|_{t=0} = w_0,
        \end{cases}
    \end{equation*}
    where $T$ is the Calder\'on-Zygmund type operator $T = \I - 2 \nabla \Delta^{-1}\div$. The global existence and uniqueness of $w_m$ follows from the classical Cauchy-Lipschitz theory for ODE's in view of the expression 
    \begin{equation*}
        w_m(x,t) = w_0(\Phi(x,t)) - \int_0^t T(w_{m-1}\cdot \nabla u)(X(\Phi(x, t), s),s)ds + \int_0^t F(X(\Phi(x,t),s),s)ds,
    \end{equation*}
    where $X$ and $\Phi$ denote the Lagrangian and backwards flows of $u$, respectively. The fact that $w_m \in C_tC_x^{N+\alpha}$ can also be seen from the expression above by induction on $m$, while $w_m \in C^1_tC_x^{N-1+\alpha}$ follows from 
    \begin{equation*}
       \partial_t w_m = - u\cdot \nabla w_m - T(w_{m-1} \cdot \nabla u) + F.
    \end{equation*}
    We claim that the sequence $\{w_m\}$ is Cauchy in $C([-\tau, \tau]; C^{N+\alpha}(\mathbb T^d)) \cap C^{1}([-\tau,\tau]; C^{N -1 + \alpha }(\mathbb T^d))$,
    provided $\tau > 0$ is chosen sufficiently small in terms of $N$, $\alpha$, and $\|u\|_{N+1+\alpha}$. To prove this, we consider $v_m = w_{m+1} - w_m$ and note that it solves 
    \begin{equation*}
        \begin{cases}
            \partial_t v_m + u \cdot \nabla v_m + T(v_{m-1}\cdot\nabla u) = 0 \\ 
            v_m \big|_{t = 0} = 0.
        \end{cases}
    \end{equation*}
    Then, assuming $\tau \|u\|_{N+\alpha} < 1$, proposition \ref{transport_estim} implies 
    \begin{equation*}
        \|v_m\|_\alpha \lesssim \tau \|T(v_{m-1}\cdot \nabla u)\|_\alpha \lesssim \tau \|v_{m-1}\|_\alpha,
    \end{equation*}
    and, for all $1 \leq k \leq N$,
    \begin{equation*}
        [v_m]_{k+\alpha} \lesssim \tau \|T(v_{m-1} \cdot \nabla u)\|_{k+\alpha} \lesssim \tau \|v_{m-1}\|_{k+\alpha}.
    \end{equation*}
    Thus, there exists a constant $C > 0$ depending on $N$, $\alpha$ and $\|u\|_{N+1+\alpha}$ so that 
    \begin{equation*}
        \|v_m\|_{N+\alpha} \leq C \tau \|v_{m-1}\|_{N+\alpha}.
    \end{equation*}
    By taking $\tau$ such that $C\tau < 1/2$, we obtain 
    \begin{equation*}
        \|v_m\|_{N+\alpha} \leq \frac{1}{2^m} \|v_0\|_{N+\alpha},
    \end{equation*}
    which immediately implies that $\{w_m\}$ is Cauchy in $C([-\tau, \tau]; C^{N+\alpha}(\mathbb T^d))$. To see that it is also Cauchy in $C^1([-\tau, \tau]; C^{N-1+\alpha}(\mathbb T^d))$, we note that, for $m, m' \in \mathbb N \setminus \{0\}$, 
    \begin{eqnarray*}
        \|\partial_t (w_m - w_{m'})\|_{N-1-\alpha} &\leq& \|u \cdot \nabla (w_{m} - w_{m'})\|_{N-1-\alpha} + \|T((w_{m-1} - w_{m'-1})\cdot \nabla u)\|_{N-1+\alpha} \\ 
       &\lesssim& \|w_m - w_{m'}\|_{N+\alpha} + \|w_{m-1} - w_{m' -1}\|_{N-1+\alpha}. 
    \end{eqnarray*}
    Let, then, 
    \begin{equation*}
        w = \lim_{m \rightarrow \infty} w_m \in C([-\tau, \tau]; C^{N+\alpha}(\mathbb T^d)) \cap C^{1}([-\tau,\tau]; C^{N -1 + \alpha }(\mathbb T^d)).
    \end{equation*}
    It is clear that $w$ is a solution to \eqref{LinEul2} on the time interval $[-\tau, \tau]$.

    \textit{Uniqueness.} Let $w_1$ and $w_2$ be two solutions to \eqref{LinEul2}, and denote $v = w_2 - w_1$. Then, 
    \begin{equation*}
    \begin{cases}
        \partial_t v + u \cdot \nabla v + T(v \cdot \nabla u) = 0 \\
        v \big|_{t = 0} = 0.
    \end{cases} 
    \end{equation*}
    By proposition \ref{transport_estim}, we have
    \begin{equation*}
        \|v(\cdot, t)\|_{\alpha} \lesssim \int_0^t\|T(v\cdot \nabla u)(\cdot, s)\|_\alpha ds \lesssim \int_0^t \|v(\cdot, s)\|_\alpha ds,
    \end{equation*}
    whenever $|t| \|u\|_1 \leq 1$. Gr\"onwall's inequality implies that $v = 0$ on the time interval $(-\|u\|_1^{-1}, \|u\|_1^{-1})$. Of course, then, uniqueness holds globally be covering $[-T,T]$ with intervals of length $\|u\|_1^{-1}$.

    \textit{Global existence.} Denote $t_k = \frac{1}{2} k \tau$, $k \in \mathbb Z$. Then, let $w^0$ be the unique solution on the time interval $(-\tau, \tau)$ with initial data $w^0\big|_{t=0}=w_0$; and let $w^k$ be defined as the solution on $(t_k - \tau, t_k + \tau) \cap [-T, T]$ with initial condition 
    \begin{equation*}
        w^k \big|_{t = t_k} = 
        \begin{cases}
            w^{k-1}(t_k), \,\,\, \text{if } k > 0, \\ 
            w^{k+1}(t_k), \,\,\, \text{if } k < 0.
        \end{cases}
    \end{equation*}
    By uniqueness, we can define a global solution $w(\cdot, t) = w^{k}(\cdot, t)$ if $t \in (t_k - \tau, t_k + \tau)$, and, thus, conclude the proof. 
\end{proof}

\begin{rem}
If $u, F \in C^\infty ([-T,T] \times \mathbb T^d)$, and $w_0 \in C^\infty(\mathbb T^d)$ then the unique solution $w$ given by the proposition is also in the regularity class $C^\infty ([-T,T] \times \mathbb T^d)$. The fact that $w \in C_t C^\infty_x$ is clear. To see the regularity in time, we write
\begin{equation*}
    \partial_t w = - u \cdot \nabla w - (\I - 2\nabla \Delta^{-1} \div)(w\cdot \nabla u) + F,
\end{equation*}
and conclude that 
\begin{equation*}
    w \in C_t^k C_x^\infty \implies \partial_t w \in C_t^k C_x^\infty \implies w \in C_t^{k+1} C_x^\infty.
\end{equation*}
\end{rem}


\begin{thebibliography}{10}

\bibitem{B15}
T.~Buckmaster.
\newblock Onsager's conjecture almost everywhere in time.
\newblock {\em Communications in Mathematical Physics}, 333(3):1175--1198, 2015.

\bibitem{BDLIS15}
T.~Buckmaster, C.~De~Lellis, P.~Isett, and L.~Sz{\'e}kelyhidi~Jr.
\newblock Anomalous dissipation for {1/5}-{H}\"older {E}uler flows.
\newblock {\em Annals of Mathematics}, 182(1):127--172, 2015.

\bibitem{BDLS13}
T.~Buckmaster, C.~De~Lellis, and L.~Sz{\'e}kelyhidi~Jr.
\newblock Transporting microstructure and dissipative Euler
flows. 
\newblock {\em arXiv:1302.2815}, 2013.

\bibitem{BDLS16}
T.~Buckmaster, C.~De~Lellis, and L.~Sz{\'e}kelyhidi~Jr.
\newblock Dissipative Euler flows with Onsager-critical spatial regularity.
\newblock {\em Comm. Pure Appl. Math.}, 69(9):1613–1670, 2016.

\bibitem{cltv}
T.~Buckmaster, C.~De~Lellis, L.~Sz{\'e}kelyhidi~Jr., and V.~Vicol.
\newblock Onsager's conjecture for admissible weak solutions.
\newblock {\em Communications on Pure and Applied Mathematics}, 72(2):229--274, 2019.

\bibitem{BMNV21}
T.~Buckmaster, N.~Masmoudi, M.~Novack, V.~Vicol.
\newblock Intermittent Convex Integration for the 3D Euler equations.
\newblock\emph{Ann. of Math. Studies}, Vol. 217, 2023. Preprint arxiv: 2101.09278.

\bibitem{BV20}
T.~Buckmaster, and V.~Vicol.
\newblock Convex integration and phenomenologies in turbulence.
\newblock {\em EMS Surveys in Mathematical Sciences}, 6(1):173--263, 2020.

\bibitem{CZ}
A.P.~Calder\'on, and A.~Zygmund. 
\newblock Singular integrals and periodic functions.
\newblock {\em Studia Math.}, 14:249--271, 1954.

\bibitem{CCFS08}
A.~Cheskidov, P.~Constantin, S.~Friedlander, and R.~Shvydkoy.
\newblock Energy conservation and {O}nsager's conjecture for the {E}uler
equations.
\newblock {\em Nonlinearity}, 21(6):1233--1252, 2008.

\bibitem{CFLS16}
A.~Cheskidov, M.C.L.~Fihlo, H.J.N.~Lopez, R.~Shvydkoy.
\newblock Energy Conservation in Two-dimensional Incompressible Ideal Fluids.
\newblock {\em Commun. Math. Phys.}, 348, 129--143 (2016).

\bibitem{cluo}
A.~Cheskidov, and X.~Luo. 
\newblock Sharp nonuniqueness for the Navier–Stokes equations. 
\newblock {\em Invent. math.}, 229, 987--1054, 2022.

\bibitem{cluo2}
A.~Cheskidov, and X.~Luo.
\newblock $L^2$-critical nonuniquenss for the 2D Navier-Stokes equations. 
\newblock {\em arXiv:2105.12117}, 2021.

\bibitem{Choff}
A.~Choffrut. 
\newblock h-Principles for the Incompressible Euler Equations. 
\newblock {\em Arch Rational Mech Anal}, 210, 133–163 (2013).

\bibitem{ChDLS12}
A.~Choffrut, C.~De Lellis, L.~Sz{\'e}kelyhidi~Jr.
\newblock Dissipative continuous Euler flows in two and three dimensions. 
\newblock {\em arXiv:1205.1226}, 2012.

\bibitem{constantin}
P.~Constantin. 
\newblock Lagrangian–Eulerian methods for uniqueness in hydrodynamic systems.
\newblock {\em Advances in Mathematics}, 278:67--102, 2015.

\bibitem{ConstantinETiti94}
P.~Constantin, W.~E, and E.~Titi.
\newblock Onsager's conjecture on the energy conservation for solutions of
{E}uler's equation.
\newblock {\em Comm. Math. Phys.}, 165(1):207--209, 1994.

\bibitem{CDLS12}
S.~Conti, C.~De Lellis, L.~Sz{\'e}kelyhidi~Jr.
\newblock \textit{h}-principle and rigidity for $C^{1,\alpha}$ isometric embeddings.
\newblock {\em Nonlinear partial differential equations}, 83--116, Springer, 2012. 

\bibitem{DaneriSzekelyhidi17}
S.~Daneri and L.~Sz{{\'e}}kelyhidi, Jr.
\newblock Non-uniqueness and h-principle for {H}\"older-continuous weak
solutions of the {E}uler equations.
\newblock {\em Arch. Rational Mech. Anal.}, 224(2):471--514, 2017.

\bibitem{DLS09}
C.~De Lellis and L.~Sz\'ekelyhidi, Jr.
\newblock The Euler equations as a differential inclusion.  
\newblock{\em Ann. of Math.}, (2), 170 (2009), no. 3, 1417-1436.  

\bibitem{DLS12}
C.~De Lellis and L.~Sz\'ekelyhidi, Jr.
\newblock The \textit{h}-principle and the equations of fluid dynamics.
\newblock {\em Bull. Amer.
Math. Soc. (N.S.)}, 49(3):347–375, 2012.

\bibitem{DLS13}
C.~De~Lellis and L.~Sz{{\'e}}kelyhidi, Jr.
\newblock Dissipative continuous {E}uler flows.
\newblock {\em Invent. Math.}, 193(2):377--407, 2013.

\bibitem{DLS14}
C.~De~Lellis and L.~Sz\'ekelyhidi, Jr.
\newblock Dissipative {E}uler flows and {O}nsager's conjecture.
\newblock {\em J. Eur. Math. Soc. (JEMS)}, 16(7):1467--1505, 2014.

\bibitem{DLS17}
C.~De~Lellis and L.~Sz\'ekelyhidi, Jr.
\newblock High dimensionality and h-principle in PDE.
\newblock {\em Bull. Amer.
Math. Soc.} 54(2):247--282, 2017.

\bibitem{DLS22}
C.~De~Lellis and L.~Sz\'ekelyhidi, Jr.
\newblock Weak stability and closure in turbulence.
\newblock {\em Philosophical Transactions of the Royal Society A} 380(2218):20210091, 2022.

\bibitem{Drivas}
T.~Drivas
\newblock Anomalous Dissipation, Spontaneous Stochasticity \& Onsager's Conjecture.
\newblock 2017. Thesis (Ph.D.)--Johns Hopkins University.

\bibitem{eyink}
G.~L.~Eyink.
\newblock Energy dissipation without viscosity in ideal hydrodynamics. I. Fourier analysis and local
energy transfer.
\newblock {\em Phys. D}, 78(3-4):222--240, 1994.

\bibitem{IsetRegulTime}
P.~Isett.
\newblock Regularity in time along the coarse scale flow for the incompressible Euler equations.
\newblock {\em 	arXiv:1307.0565}, 2013.

\bibitem{IsPhD}
P.~Isett.
\newblock Holder continuous {E}uler flows with compact support in time.
\newblock ProQuest LLC, Ann Arbor, MI, 2013.
\newblock Thesis (Ph.D.)--Princeton University.

\bibitem{Isett18}
P.~Isett.
\newblock A proof of Onsager{\textquotesingle}s conjecture.
\newblock {\em Annals of Mathematics}, 188(3):871, 2018.

\bibitem{Nash}
J. Nash. 
\newblock $C^1$ isometric imbeddings
\newblock\emph{Ann. Math.}, 60:383–396, 1954.

\bibitem{N20}
M.~Novack.
\newblock Nonuniqueness of Weak Solutions to the 3 Dimensional Quasi-Geostrophic Equations
\newblock\emph{SIAM Journal on Mathematical Analysis}, 52(4):3301--3349, 2020.

\bibitem{NV22}
M.~Novack, and V.~Vicol.
\newblock An Intermittent Onsager Theorem.
\newblock\emph{Invent. Math.}, 233:223--323, 2023.

\bibitem{Onsager49}
L.~Onsager.
\newblock {S}tatistical hydrodynamics.
\newblock {\em Il {N}uovo {C}imento (1943-1954)}, 6:279--287, 1949.

\bibitem{Scheffer93}
V.~Scheffer. 
\newblock An inviscid flow with compact support in space-time.
\newblock {\em J. Geom. Anal.}, 3(4):343--401, 1993.

\bibitem{Shnirelman00} 
A.~Shnirelman. 
\newblock Weak solutions with decreasing energy of incompressible {E}uler equations. 
\newblock {\em Comm. Math. Phys.}, 210(3):541--603, 2000.

\bibitem{S12}
L.~Sz\'ekelyhidi, Jr.
\newblock From isometric embeddings to turbulence
\newblock {\em HCDTE lecture notes. Part II. Nonlinear hyperbolic PDEs, dispersive and transport equations}, 7, 2012.
  
\end{thebibliography}
\end{document}